\documentclass[12pt]{amsart}
\usepackage{fullpage}
\usepackage{bm}
\usepackage[all]{xy}
\usepackage{tikz}
\usepackage{amsmath, amsthm, amscd, amssymb, amsfonts, latexsym}
\usepackage{mathrsfs,amssymb}
\usepackage{amsfonts}
\usepackage{amssymb}
\usepackage{amsmath}
\usepackage{amsthm}
\usepackage{bbm}
\usepackage{dsfont}
\usepackage{relsize}
\usepackage[mathscr]{euscript}
\usepackage{upgreek}
\usepackage{dsfont}
\usepackage{marvosym}

\newtheorem{theorem}{Theorem}[section]
\newtheorem{lemma}[theorem]{Lemma}
\newtheorem{proposition}[theorem]{Proposition}
\newtheorem{corollary}[theorem]{Corollary}

\newtheorem{remark}[theorem]{Remark}

\DeclareMathOperator{\re}{Re}

\newcommand{\Z}{\mathbb{Z}}
\newcommand{\F}{\mathbb{F}}

\newcommand{\C}{\mathbb{C}}

\renewcommand{\pmod}[1]{\,(\mathrm{mod}\,#1)}
\def\cH{\mathcal{H}}

\def\cH{\mathcal{H}}
\def\cL{\mathcal{L}}
\def\cM{\mathcal{M}}

\newcommand{\legl}[2]{\left( \frac{#1}{#2} \right)_\ell}
\newcommand{\kommentar}[1]{ }

\begin{document}

\title{Nonvanishing of $L$--functions associated to fixed order characters over function fields }
\author{Chantal David}
\address{Department of Mathematics and Statistics, Concordia University, 1455 de Maisonneuve West, Montr\'eal, Qu\'ebec, Canada H3G 1M8}
\email{chantal.david@concordia.ca}
\author{Alexandra Florea}
\address{Department of Mathematics, University of California Irvine, 340 Rowland Hall, Irvine CA 92697, USA}
\email{\href{mailto:floreaa@uci.edu}{floreaa@uci.edu}}
\date{\today}
\author{Matilde Lalin}
\address{D\'epartement de math\'ematiques et de statistique, Universit\'e de Montr\'eal, CP 6128, succ. Centre-ville, Montr\'eal, QC H3C 3J7, Canada}\email{matilde.lalin@umontreal.ca}
\numberwithin{equation}{section}

\begin{abstract}
We show that a positive proportion of the values $L(1/2,\chi_c)$ are non-zero, where $\chi_c$ is the $\ell^{\text{th}}$ residue symbol for $\ell \geq 3$ over $\mathbb{F}_q[t]$, when averaging over square-free polynomials $c$ in $\mathbb{F}_q[t]$, as $q \equiv 1 \pmod{2\ell}$ is fixed and the degree of $c$ goes to infinity. In the case of $\ell=3$, we show that at least $1/6$ of $L(1/2,\chi_c) \neq 0$, while for $\ell>3$, the proportion depends on the order of the character.
This improves a previous result of Ellenberg, Li, and Shusterman showing that there are infinitely many  $\chi$ of (prime) order $\ell$ such that $L(1/2, \chi) \neq 0$ (with completely different techniques). Our result is achieved by computing the one-level density of zeros in the family of $L$--functions and surpassing the $(-1,1)$ barrier for the support of the Fourier transform of the test function, necessary to obtain a positive proportion of non-vanishing result.
Using similar techniques, we also prove a result towards the equidistribution of the angles of the order $\ell$ shifted Gauss sums when summing over prime arguments, a result which may be of independent interest.
\end{abstract}

\maketitle

\section{Introduction}
Central values of $L$--functions are of fundamental importance in number theory. Many results and conjectures concern the values of $L$--functions at special points, and the problem of proving non-vanishing results at the central point $s=1/2$ is major. A famous conjecture of Chowla predicts that $L(1/2,\chi) \neq 0$ for the $L$--function attached to any Dirichlet character $\chi$.
Focusing on the family of Dirichlet $L$--functions associated to primitive Dirichlet characters modulo $q$, Iwaniec and Sarnak \cite{is} showed that more than $(1/3-\varepsilon)$ of $L(1/2,\chi) \neq 0$, for $q$ sufficiently large. This proportion was subsequently improved in several works, the best current proportion of nonvanishing being $34.11\%$, due to Bui \cite{bui}. In the case of prime moduli, Khan, Mili\'cevi\'c and Ngo \cite{kmn} showed that more than $5/13$ of the central values of $L$--functions do not vanish.

Restricting to quadratic characters, Soundararajan \cite{Sound} showed that more than $7/8$ of $L(1/2,\chi) \neq 0$, as $\chi$ ranges over real primitive characters, by computing mollified moments in the family. Under the Generalized Riemann Hypothesis (GRH), \"Ozl\"uk and Snyder \cite{OS} showed that at least $15/16$ of the $L$--functions attached to quadratic characters $\chi$ do not vanish, by computing the one-level density for the low-lying zeroes in the family. The density conjectures of Katz and Sarnak \cite{KS} on the zeroes of $L$--functions imply
that $L(1/2, \chi) \neq 0$ for almost all quadratic Dirichlet $L$--functions. The results become sparser and sparser as one considers $L$--functions associated to higher-order characters. Focusing on the family of $L$--functions associated to cubic characters over $\mathbb{Q}$, Baier and Young \cite{BY} showed that the number of primitive Dirichlet characters $\chi$ of order $3$ with conductor less than $Q$ for which $L(1/2,\chi) \neq 0$ is bounded below by $Q^{6/7-\varepsilon}$. Note that while this result provides an infinite family of $L$--functions that are non-vanishing at the central point, it does not obtain a positive proportion of non-vanishing, since the family has size $Q$. More recently, David and G\"ulo\u{g}lu \cite{dg} showed that under GRH, more than $2/13$ of the $L$--functions of the  cubic Dirichlet characters $\chi_c$, for square-free $c$ in the Eisenstein field, do not vanish at the central point. This was achieved by computing the one-level density in the family, as in the present paper. While the approach of computing the one-level density of zeros usually yields non-vanishing results conditional on the GRH, this assumption was removed in recent breakthrough work of David, de Faveri, Dunn, and Stucky \cite{ddds}, who showed that more than $14\%$ of the central values in the same family of cubic $L$--functions are non-zero, by computing mollified moments.

In the function field setting, Bui and Florea \cite{Bui-Florea} obtained that at least $94\%$ of the quadratic Dirichlet $L$--functions do not vanish at the central point, by computing the one-level density, and the authors of the present article \cite{DFL2} showed that a positive proportion of the central values of $L$--functions associated to cubic characters over $\mathbb{F}_q[t]$ do not vanish when $q \equiv 2 \pmod3$. We note that this is the non-Kummer setting, when the base field does not contain the cubic roots of unity, and it is a regime more similar to that considered by Baier and Young \cite{BY} in the integer setting. Contrastingly, \cite{DFL2} provides a positive proportion result for the \textit{full} family of cubic $L$--functions, as opposed to the thin families in \cite{dg}, \cite{ddds} or the present article.

In the case of $L$--functions associated to quartic characters, Gao and Zhao \cite{gao_zhao} showed that more than $5\%$ of quartic $L$--functions in a thin family over the Gaussian field are non-zero at the central point, by computing the one-level density of zeros, under GRH. They also showed \cite{gao_zhao2} that a positive proportion of the cubic and quartic $L$--functions over $\mathbb{Q}$ are non-zero at the central point, under GRH, by following the method in \cite{DFL2}. This was also achieved for the full family of cubic $L$--functions over the Eisenstein field by G\"ulo\u{g}lu and Yesilyurt \cite{G-Y}, generalizing the  technique of  \cite{DFL2} to this case, again under GRH.

We emphasize that all the non-vanishing results mentioned above hold for small order characters ($\ell=2,3,4$). For general $\ell$, in the number field setting, Blomer, Goldmakher, and Louvel \cite{bgl} showed that there are infinitely many Hecke $L$--functions associated to order $\ell$ characters over number fields containing the $\ell^{\text{th}}$ roots of unity, for which the $L$--functions do not vanish at the central point; however, they did not obtain a positive proportion. In the function field setting, Ellenberg, Li, and Shusterman \cite{ELS} obtained a similar result, using algebraic geometry tools.

One of the reasons for the scarcity of positive proportion non-vanishing results for higher-order characters is our limited understanding of the Gauss sums associated to fixed order characters.
When the character is quadratic, the Gauss sums have a multiplicative structure, which makes them amenable to study. They are much more complicated as the order of the character increases, as the Gauss sums are no longer multiplicative for $\ell>2$.
While quadratic Gauss sums are basically constant, cubic and higher-order Gauss sums are chaotic, and their generating series are related to theta functions associated to metaplectic forms, which are not well-understood.

In this article, we consider the thin family of order $\ell$ Dirichlet characters over $\mathbb{F}_q[t]$ with square-free discriminants and their associated $L$--functions, for general $\ell$. We will show that for fixed $\ell$, a positive proportion of such characters has the property that the associated $L$--functions do not vanish at the central point $1/2$, when the size $q$ of the base field $\mathbb{F}_q$ is fixed. In our work, $q \equiv 1 \pmod{2\ell}$, so we are in the Kummer setting, when the base field contains the order $\ell$ roots of unity in $\mathbb{F}_q$. This setting is the analogue of the regimes considered in \cite{dg, ddds, gao_zhao}. To obtain the non-vanishing results, we compute the one-level density of zeros in the family of $L$--functions.

More specifically, for any positive integer $d$ such that $\ell \nmid d$, we consider the family
\begin{equation}
\label{family}
\mathcal{F}_\ell(d):=\left\{\chi_c =\left(\frac{\cdot}{c}\right)_\ell :\, c\, \mbox{ square-free}, \deg(c) = d \right\},
\end{equation}
where  $(\frac{\cdot}{c})_{\ell}$ denotes the $\ell^{\text{th}}$ residue symbol modulo $c$ defined in \eqref{chi-c-notation}. Since $\ell \nmid d$, the character $\chi_c$ is odd, and 
the $L$--function associated to $\chi_c$ can be written as
$$\cL(u, \chi_c) = \prod_{j=1}^{d-1} \left( 1 -  q^{1/2} e^{- 2 \pi i \theta_{c, j}} u \right).$$

Let $\phi$ be an even test function in the Schwartz space. We consider
$$\Phi(\theta) := \sum_{k \in \mathbb{Z}} \phi ( (d-1) ( \theta-k) ).$$
Note that this function is periodic with period 1, and localized in an interval of size $ \sim 1/d$ in $\mathbb{R}/ \mathbb{Z}$. We define
\begin{equation}
\label{sigma}
\Sigma(\phi,\chi_c):=\sum_{j=1}^{d-1}\Phi\left( \theta_{c, j}\right).
\end{equation}
The one-level density of zeros is defined to be the limit as $d \to \infty$ of
\begin{equation}
\label{1ld_def}
\left\langle\Sigma(\phi,\chi_c)\right\rangle_{\mathcal{F}_\ell(d)} := \frac{1}{|\mathcal{F}_\ell(d)|}\sum_{\chi_c \in \mathcal{F}_\ell(d)}\Sigma(\phi,\chi_c).
\end{equation}

We will prove the following results.
\begin{theorem}
\label{main_thm}
Let $\ell \geq 3$ be an integer, and assume that $q \equiv 1 \pmod{2\ell}$. Let $d \not \equiv 0 \pmod{\ell}$, and $\varepsilon>0$. Let $\phi$ be an even test function in the Schwartz space whose Fourier transform is supported in $(-v,v)$, where
\[v=\begin{cases}\frac{6}{5} & \ell=3,\\
\frac{26}{23} & \ell=4,\\
1+\frac{2(\ell-2)}{2\ell^2-\ell+2} & 5\leq \ell \leq 8,\\
1+\frac{2(\ell-2)}{3\ell^2-9\ell+2} & \ell=9, 10,\\
1 + \frac{6 (\ell - 2)}{9 \ell^2 - 31 \ell + 6} & 11\leq  \ell.\end{cases}
\]
Then
\begin{equation*}
\left\langle\Sigma(\phi,\chi_c)\right\rangle_{\mathcal{F}_\ell(d)}  = \widehat{\phi}(0) +O \left(\frac{1}{d}\right) .
\end{equation*}
\end{theorem}

From Theorem \ref{main_thm}, we also obtain the following non-vanishing result. We remark that since characters of order $\ell \geq 3$ have unitary symmetry, we need $v > 1$ for all $\ell$ in Theorem \ref{main_thm} to obtain a positive proportion of non-vanishing result.
\begin{corollary}
\label{cor_nonvanishing}
Let $\mathcal{F}_{\ell}(d)$ be the family defined in \eqref{family}.

If $\ell=3$, then $L(1/2,\chi) \neq 0$ for at least $1/6$ of the characters in $\mathcal{F}_{\ell}(d)$.

If $\ell=4$, then $L(1/2,\chi) \neq 0$ for at least $3/26$ of the characters in $\mathcal{F}_{\ell}(d)$.

If $5\leq \ell\leq 8$, then $L(1/2,\chi) \neq 0$ for at least $\frac{2(\ell-2)}{2\ell^2+\ell-2}$ of the characters in $\mathcal{F}_{\ell}(d)$.

If $\ell=9,10$, then $L(1/2,\chi) \neq 0$ for at least $\frac{2(\ell-2)}{3\ell^2-7\ell-2}$ of the characters in $\mathcal{F}_{\ell}(d)$.

If $\ell \geq 11$, then $L(1/2,\chi) \neq 0$ for at least $\frac{6(\ell-2)}{9\ell^2-25\ell-6}$ of the characters in $\mathcal{F}_{\ell}(d)$.

\end{corollary}
\begin{remark}
We note that when $\ell \to \infty$, the proportion of non-vanishing provided by Corollary \ref{cor_nonvanishing} approaches $0$.
\end{remark}

The following are crucial ingredients in the proof of Theorem \ref{main_thm}.

\begin{theorem}(Large Sieve for characters of order $\ell$)
\label{largesieve}
Let $\lambda(N)$ be a sequence of complex numbers defined on monic polynomials in $\mathbb{F}_q[t]$. Then we have 
$$
\sum_{M \in \cH_m} \left| \sum_{N \in \cH_n} \lambda(N) \legl{M}{N} \right|^2 \ll q^{\varepsilon(m+n)} \left( q^m + q^n + q^{2(m+n)/3} \right) \sum_{N \in \cH_n} |\lambda(N)|^2,
$$
where $\mathcal{H}_d$ denotes the set of monic, square-free polynomials of degree $d$ over $\mathbb{F}_q[t]$.
\end{theorem}

We also need to prove the following variant.
\begin{theorem}
\label{largesieve2}
We have
\begin{align*}
\sum_{M \in \mathcal{M}_m} \left| \sum_{N \in \mathcal{H}_n}   \lambda(N) \legl{M}{N} \right|^2 \ll q^{\varepsilon (m+n)} \left( q^m + q^{n+m/3} +  q^{2(m+n)/3} \right) \sum_{N \in \mathcal{H}_n} | \lambda(N)|^2,
\end{align*}
where $\mathcal{M}_m$ denotes the set of monic polynomials of degree $m$ in $\mathbb{F}_q[t]$.
\end{theorem}
In the course of proving Theorem \ref{main_thm}, we will also prove the following result about averages of Gauss sums at prime arguments, which may be of independent interest. Let $G_{\ell}(V,c)$ denote the $\ell^{\text{th}}$ order Gauss sum over function fields, defined in \eqref{def-GS}. 

\begin{theorem}
\label{theorem_cancellation}
Let $\mathcal{P}_n$ denotes the set of monic, irreducible polynomials of degree $n$ in $\mathbb{F}_q[t]$.
For $\ell>8$ and any $\varepsilon>0$, we have that
$$ \sum_{\pi \in \mathcal{P}_n} G_{\ell}(R,\pi) \ll q^{n (\frac{3}{2}-\frac{\ell-2}{\ell^2})+\varepsilon n} |R|^{\frac{\ell^2-7\ell-2}{2 \ell^2}+\varepsilon} .$$For $3<\ell\leq 8$, then
$$ \sum_{\pi \in \mathcal{P}_n} G_{\ell}(R,\pi) \ll q^{n (\frac{3}{2}-\frac{\ell-2}{\ell^2})+\varepsilon n} |R|^{\frac{\ell-2}{2 \ell^2}+\varepsilon} .$$
When $\ell=3$, then
$$ \sum_{\pi \in \mathcal{P}_n} G_{\ell}(R,\pi) \ll q^{\frac{4n}{3}+\varepsilon n} |R|^{\varepsilon}  .$$
\end{theorem}
By Corollary \ref{size-GS},  $|G_{\ell}(R,\pi)|= |\pi|^{1/2}$ for $(\pi, R)=1$. Therefore, Theorem \ref{theorem_cancellation} provides cancellation in averages of the shifted Gauss sums  $G_\ell(R, \pi)$ at prime arguments when $\deg(R)<\frac{2(\ell-2)}{\ell^2-7\ell-2}n$ for $\ell>8$, when
$\deg(R) <  2 n$ for $3<\ell \leq 8$, and for all shifts $R$ for $\ell = 3$. By Weyl's criterion,  this could potentially lead  to the  equidistribution of the angles $\theta_{\ell, R}$ given by
$\frac{G_{\ell}(R,\pi)}{|\pi|^{1/2}} = e^{i \theta_{\ell,R}}$ in those cases, by considering more general generating functions than the functions $\psi^{(i)}_v(r,a,u)$ from Section \ref{sec:remove} to deal with the powers of $G_{\ell}(R,\pi)$.

Over number fields, the equidistribution of the cubic Gauss sums $G_3(\pi) := G_{3}(1,\pi)$ at prime arguments over the Eisenstein field was proven by Heath-Brown and Patterson \cite{hbp} and refined  in \cite{HB00}, falling just short of determining an asymptotic formula conjectured by Patterson for averages of cubic Gauss sums at prime arguments. This was subsequently proved by Dunn and Radziwi\l\l~ \cite{DR}. For the general Gauss sums $G_\ell(\pi)$, an equidistribution result was proven by Patterson \cite{LMS},  and refined  by \cite{DDHL} for $\ell=4$.
All the number field results hold for $R=1$, while the
result in Theorem \ref{theorem_cancellation} provides uniformity in the shift $R$.

\subsection{Main ideas}

To prove the main result, Theorem \ref{main_thm}, we start in the usual way by employing the explicit formula which relates sums over zeros of $L$--functions to sums over monic irreducible polynomials. After removing the square-free condition on the sum over the moduli $c$ which parametrize the family and using Poisson summation, we obtain a sum of Gauss sums over a dual parameter which ranges over monic polynomials $V$.

By switching the sums over the primes and the sum over the dual parameter, we are led to bound sums of Gauss sums over the primes. Since Gauss sums are not multiplicative for $\ell>2$, we cannot use the standard techniques to bound the sums over the primes (i.e., monic irreducible polynomials). Instead, we
use Vaughan's identity to rewrite the sums over primes to sums over monic polynomials with certain divisibility conditions in some ranges. These terms come in two types, Type I and Type II sums. We note that this circle of ideas
was first used in the work of Heath-Brown and Patterson \cite{hbp} in showing the uniform distribution of the angles of cubic Gauss sums around the unit circle.

We then deal with the Type I and Type II terms differently. To bound  the Type II terms, we use the large sieve inequality for characters of order $\ell$, as proved in Theorems \ref{largesieve} and \ref{largesieve2}. This builds on ideas of Heath-Brown \cite{HB95, HB00}, who obtained large sieve inequalities for quadratic and cubic characters over $\mathbb{Q}$, and work of Blomer, Goldmakher, and Louvel \cite{bgl} on the large sieve for (higher) fixed order characters. We obtain a bound on the Type II sums which does not depend on the order $\ell$ of the characters.

Dealing with the Type I terms is more subtle, and is the main focus of the paper. After using Poisson summation and Vaughan's identity as described above, we roughly need to consider sums which can be written in a simplified form as
$$ \sum_{V \in \mathcal{M}_{n-d}} \sum_{b \in \mathcal{M}_{\leq U}} \mu^2(b) \sum_{\substack{c \in \mathcal{M}_n \\ b|c \\ (c,V)=1}} G_{\ell}(V,c),$$ where recall that $G_{\ell}(V,c)$ is the $\ell^{\text{th}}$ order Gauss sum over function fields as defined in \eqref{def-GS}, and $n \leq v(d-1)$, where $v$ is the support of the Fourier transform of the test function $\phi$. In the above, $\mathcal{M}_{\leq U}$ denotes the set of monic polynomials of degree less than or equal to $U$, where $U$ is a parameter which grows with $n$, and $U<n/3$.
We then need to consider the generating series of the Gauss sums with the coprimality and divisibility conditions as above. After (tedious) combinatorial manipulations (see Theorem \ref{thm:gettingridofa}), we express this generating series in terms of the generating series
\[
 \psi^{(i)}(V, u)  := (1-u^{\ell}  q^{\ell})^{-1}\sum_{\substack{F \in \mathcal{M}\\ \deg(F) \equiv i \pmod{\ell} }} G_\ell(V, F) u^{\deg(F)},\]
for $i \in \{0,1,\ldots, \ell-1\}$, with the coprimality and divisibility conditions omitted. We are then led to bound averages of the form
 \begin{equation} 
\label{sumv2}\sum_{V \in \mathcal{M}_{n-d}} | \psi^{(i)}(V,u)|.\end{equation}
Therefore, one needs to understand the properties of $\psi^{(i)}(V,u)$. This function has been studied by Hoffstein and Patterson in the function field setting \cite{patterson, hoffstein}. In the case of quadratic characters ($\ell=2$), the Gauss sums are multiplicative, so the generating series has an Euler product, and its structure is fully understood. However, for $\ell>2$, the Gauss sums are no longer multiplicative, and the current understanding of $\psi^{(i)}(V,u)$ remains incomplete in the case of higher-order Gauss sums. These functions are known to be related to theta functions on the $\ell$-fold metaplectic covers of $\text{SL}(2)$ \cite{hoffstein}.

It can be shown that $\psi^{(i)}(V,u)$ is a rational function (see also \cite{DFL} in the case $\ell=3$) with possible poles at $u^{\ell} = q^{-\ell-1}$. Crucially, the generating series $\psi^{(i)}(V,u)$ satisfies a functional equation which relates $\psi^{(i)}(V,u)$ to $\psi^{(i)}(V; 1/(q^2u))$. From the functional equation and the Phragm\'en--Lindel\"{o}f principle, one can obtain a convexity bound on the size of $\psi^{(i)}(V,u)$ in the critical strip (Theorem \ref{residue}). However, using the pointwise convexity bound is not optimal for $\ell$ sufficiently large, and that would lead to a weaker proportion than that in Theorem \ref{main_thm}.
To get a better proportion, we use the functional equation again to obtain an ``approximate functional equation'' for $\psi^{(i)}(V,u)$. Since $\psi^{(i)}(V,u)$ is a rational function, the numerator of $\psi^{(i)}(V,u)$ is a polynomial whose degree is roughly equal to $\deg(V)$. The approximate functional equation allows us to express the numerator of $\psi^{(i)}(V,u)$ as a sum of two terms, each of length roughly $\deg(V)/2$ (similar to the case of  $L$--functions over number fields). We are now in a position to use the large sieve inequality again and exploit the fact that $(u^\ell - q^{\ell-1}) \psi^{(i)}(V,u)$ is written in terms of shorter polynomials. This yields a Lindel\"{o}f on average bound for \eqref{sumv2}, which beats the trivial bound obtained from convexity. We note that a similar idea was exploited in \cite{ddds} in the number field setting, in bounding the residual integral after extracting the contribution of the pole (in that instance, since the authors are considering the case $\ell = 3$, the contribution of the poles is exact, see Theorem \ref{residue}). In the present work, we use the average bound when we are close to the residue, and not for the residual integral at the central critical point (which, in the case of $\psi^{(i)}(V,u)$, is at $u=1/q$.)

We also need to consider more general averages of the form
\begin{equation} \sum_{V \in \mathcal{M}_{n-d}} \Big|  \psi_V^{(i)} (V,b,u)\Big|,
\label{intro2}
\end{equation} where  
$$ \psi_V^{(i)} (V,b,u)\ = \sum_{\substack{c \in \mathcal{M} \\ b|c \\ (c,V)=1}} G_{\ell}(V,c) u^{\deg(c)}.$$
Using the combinatorial manipulations in Theorem \ref{thm:gettingridofa}, evaluating  \eqref{intro2} roughly reduces to considering sums which can be written in a simplified form as 
\begin{equation*}
%\label{intro3}
\sum_{V \in \mathcal{H}_{n-d}} \Big|  \psi^{(i)}(b^{\ell-2} V^{\ell-3},u)\Big|,
\end{equation*}
where $\mathcal{H}_{n-d}$ denotes the set of monic, square-free polynomials of degree $n-d$. When $\ell$ grows, the high powers of $b$ and $V$ become problematic, and we prove in Section \ref{section_lov} several versions of Lindel\"{o}f on average to be used in different ranges to optimize the bounds.

Finally, we note that simply using the convexity bound in \eqref{intro2} would also lead to a positive proportion of non-vanishing; however, the proportion would be weaker than what we prove using Lindel\"{o}f on average for \eqref{intro2} (for $\ell\geq 11$). While we believe an appropriate version of the Lindel\"{o}f on average bound should give better final error terms than the convexity bound for small values of $\ell$ as well (i.e., $\ell <11$), we have decided not to optimize that, in order to keep the article at a reasonable length.

\subsection{Overview of this article}
This article is organized as follows. In Section \ref{background}, we gather various facts about order $\ell$ characters and their $L$--functions. We introduce the order $\ell$ Gauss sums in Section \ref{gauss-sums-section}, and prove some of their properties. In Section \ref{gen-series}, we gather some necessary facts from \cite{DFL, hoffstein, patterson} about the generating series of the Gauss sums, and prove the approximate functional equation. The main result in Section \ref{sec:remove} expresses the generating series of Gauss sums with coprimality and divisibility relations in terms of the usual generating series. Section \ref{section_lov} proves several Lindel\"{o}f on average-type bounds for the generating series of the Gauss sums. We begin the proof of Theorem \ref{main_thm} in Section \ref{section_setup}, where we use Vaughan's identity for the sums over primes. We bound the Type II terms  in Section \ref{typeII_section}, and we deal with the Type I terms in Section \ref{typeI}. We finally prove Theorem \ref{main_thm} in Section \ref{proof}, Theorem \ref{theorem_cancellation}
 in Section \ref{proof-thm-cancellation}, and  the large sieve inequalities in Section \ref{ls}.

\subsection*{Acknowledgements.}
The authors would like to thank Crystel Bujold for useful discussions related to this paper, and Mark Shusterman for helpful comments. CD thanks Alex Dunn for helpful discussions, and AF also thanks K.  Soundararajan for many discussions related to the quadratic large sieve over function fields, which helped with the section on the order $\ell$ large sieve in this paper.
This research has been supported by the National Science Foundation (DMS-2101769 and CAREER grant DMS-2339274 to AF), the Natural Sciences and Engineering Research Council of Canada (RGPIN-2019-05536 to CD and  RGPIN-2022-03651 to ML) and the Fonds de recherche du Qu\'ebec - Nature et technologies (Projets de recherche en \'equipe 300951 to CD and ML, and 345672 to CD and ML).

\section{Background and setting}
\label{background}
\subsection{Generalities about characters} For references on fixed order characters, see \cite{ldlw}. Let $q$ be a power of an odd prime $p$. We denote by $\mathcal{M}$ the set of monic polynomials of $\F_q[t]$, by $\mathcal{M}_{n}$ the set of monic polynomials of degree exactly $n$, by
$\mathcal{M}_{\leq n}$ the set of monic polynomials of degree smaller than or equal to $n$,  
 by $\mathcal{H}$ the set of monic square-free polynomials of $\F_q[t]$ and analogously for
 $\mathcal{H}_{n}$ and $\mathcal{H}_{\leq n}$. Note that $| \mathcal{M}_{n}| = q^n$ and for $n \geq 2$, we have that $| \mathcal{H}_{n}| = q^n (1-\tfrac{1}{q}).$
 Similarly define the sets $\mathcal{P}, \mathcal{P}_{n}$ and $\mathcal{P}_{\leq n}$ to be the corresponding sets consisting of monic irreducible polynomials in $\mathbb{F}_q[t]$. For simplicity, whenever we talk about primes in $\F_q[t]$, we refer to monic irreducible polynomials.

We define the norm of a polynomial $f(t) \in \F_q[t]$ over $\F_q[t]$ by
$$ |f| := q^{\deg(f)}.
$$

Let $\ell$ be such that $p\nmid \ell$.  We consider the Kummer case, in other words, $q\equiv 1 \pmod{\ell}$. We fix  an isomorphism $\Omega$ from the $\ell$-th roots of unity in $\F_{q}^*$ to the subgroup $\mu_\ell$ of $\ell$-th roots of unity in $\C^*$.

Let $\pi \in \mathcal{P}$, and let $a \in \F_q[t]$. Then the $\ell^{\text{th}}$ residue symbol modulo $\pi$ is constructed as follows. If $\pi|a$, then we define $(\frac{a}{\pi})_{\ell}=0$. If $\pi \nmid a$, we define
$$\legl{a}{\pi} =\alpha,$$ where $\alpha $ is the unique $\ell^{\text{th}}$ root of unity in $\mu_{\ell}$ such that
\begin{align*} %\label{def-chiP}
 a^{\frac{q^{\deg(\pi)}-1}{\ell}} \equiv \Omega^{-1}(\alpha) \pmod \pi. \end{align*}
We extend the symbol by multiplicativity for $c=\pi_1^{e_1} \dots \pi_k^{e_k}$ by
\begin{align*}
\legl{a}{c} = \legl{a}{\pi_1}^{e_1} \dots \legl{a}{\pi_k}^{e_k}.
\end{align*}
We also denote
\begin{align} \label{chi-c-notation} 
\chi_c(\cdot) := \legl{\cdot}{c}.\end{align}

Notice that if $\alpha \in \F_q$, we have 
\begin{align*} %\label{over-Fp}
\legl{\alpha}{\pi} = \alpha^{\frac{q-1}{\ell} \deg(\pi)}.
\end{align*}
We will often use the $\ell^{\text{th}}$ power reciprocity law, which says the following. For $A, B \in \cM$, 
\begin{align*}
%\label{thm-REC}
\legl{A}{B} \legl{B}{A}^{-1} = (-1)^{\frac{q-1}{\ell} \deg(A) \deg(B)}.
\end{align*}
Throughout the paper, we will assume that $q \equiv 1 \pmod{2\ell}$, so
that $\legl{A}{B} =\legl{B}{A}$.

Each character $\chi_c$ with $c$ square-free of degree $d$ is associated to a
 cover $C$  over  $\mathbb{P}^1_{\F_q}$ with affine equation
$C: y^\ell = c(t)$.  Denoting by $P_\infty$ the prime at infinity, there are two cases:
\begin{itemize}
 \item[] If $\ell \mid d$, then  $\chi$ is even (that is, $\chi$ is trivial on $\F_q$) and  $\chi_c(P_\infty)=1$.
 \item[] If $\ell \nmid d$, then $\chi$ is odd (that is, $\chi$ is non-trivial on $\F_q$) and $\chi_c(P_\infty)=0$.
  \end{itemize}
By the Riemann--Hurwitz formula, the genus $g$ of $C$ is given by
\[g=-(\ell-1)+\frac{1}{2}\sum_P(e_P-1),\]
where the sum is taken over all the $P$ in the extension that ramify and $e_P$ denotes the ramification index. In the particular case where $\ell$ is prime or when $c$ is square-free (see Corollary  3.7.4 in \cite{Stichtenoth}), the above formula reduces to
 \begin{equation}
 \label{genus}
 g=\frac{\ell-1}{2}(d-1-\chi_c(P_\infty)).
 \end{equation}
  
Let $\chi$ be any of the characters defined by \eqref{chi-c-notation}, and let $\mathcal{L}(u,\chi)$ be the Dirichlet $L$--function defined by
\begin{equation*}
\mathcal{L}(u,\chi) = \sum_{f \in \mathcal{M}} \chi(f) u^{\deg(f)}.
\end{equation*}
The $L$--function also has an Euler product given by
\begin{equation}\label{eq:Eulerprod}\mathcal{L}(u, \chi)=\prod_P (1-\chi(P)u^{\deg(P)})^{-1},\end{equation}
where the product includes the prime at infinity, and the above holds for $|u|<1/q$.

In this work, we consider $\chi \in \mathcal{F}_{\ell}(d)$, where $\mathcal{F}_{\ell}(d)$ is the family given in equation \eqref{family}.
For such a primitive character $\chi$, it follows from the work of Weil \cite{Weil2} that when $\ell$ is prime,  $\cL( u,\chi)$ is a polynomial of degree $d-1-\chi(P_\infty)$ by \eqref{genus}. In the case where $c$ is square-free (and $\ell$ is not necessarily prime), there are $\ell-1$ $L$--functions associated to the characters $\chi_c, \dots, \chi_c^{\ell-1}$ respectively. By Proposition 4.3 in \cite{Rosen}, each of them has degree at most $d-1-\chi(P_\infty)$. By \eqref{genus}, the total degree must be $2g=(\ell-1)(d-1-\chi(P_\infty))$, and therefore each $\cL( u,\chi)$ has degree exactly $d-1-\chi(P_\infty)$ in this case as well. Recall that we are taking $d \not \equiv 0 \pmod{\ell}$ and therefore $\chi$ is odd and $\chi(P_\infty)=0$.

Thus, we can  write
\begin{equation}\label{eq:rootsprod}\cL( u,\chi)=\prod_{j=1}^{d-1} (1-u \sqrt{q} e^{-2\pi i \theta_{\chi,j}}).\end{equation}
Moreover, $\cL( u,\chi)$
 satisfies the functional equation 
 \begin{align*} %\label{FE-Lchi}
\cL(u,\chi) = \omega(\chi) \;(qu^2)^{\frac{d-1}{2}} \;\cL\left(\frac{1}{qu}, \overline{\chi} \right). \end{align*}
The sign of the functional equation is given by
\begin{equation}\label{eq:signfe}
\omega(\chi) = \begin{cases}-q^{(d-2)/2}\sum_{f\in \mathcal{M}_{d-1}\chi(f)}= \frac{G(\chi)}{|G(\chi)|}  & \text{when $\chi$ is even}, \\     \\
q^{(d-1)/2}\sum_{f\in \mathcal{M}_{d-1}}\chi(f)= \frac{\sqrt{q}}{\tau(\chi)}  \frac{G(\chi)}{|G(\chi)|}  & \text{when $\chi$ is odd},
\end{cases}
\end{equation}
where  if 
$\chi$ odd, 
\begin{equation}
\label{epsilonfactor}
\tau(\chi) = \sum_{a \in \F_q^*} \chi(a) e^{2 \pi i \text{Tr}_{\F_q/\F_p}(a)/p}, \;\;\;\; \varepsilon(\chi) = \frac{\tau(\chi)}{q^{1/2}},
\end{equation}
and if $\chi$ even, $ \varepsilon(\chi) =1$.

For any $\chi$, the Gauss sum is given by
$$
G(\chi) := \sum_{a \pmod{F}} \chi(a) e_q\left( \frac{a}{F} \right).$$
Here $e_q$ is the exponential defined by Hayes \cite{hayes} for any $b \in \mathbb{F}_q((1/t))$ by
\begin{equation}
\label{exp}
e_q(b) = e^{\frac{2 \pi i \text{Tr}_{\F_q/\F_p}(b_1) }{p}},
\end{equation}
where $b_1$ is the coefficient of $1/t$ in the Laurent expansion of $b$.

We have the following Weil bound for sums over primes. If $\chi$ is a primitive character modulo $Q$ and $\deg(Q)=m$, then
\begin{equation}
\label{weil}
\Big| \sum_{\pi \in \mathcal{P}_n} \chi(\pi) \Big| \ll q^{n/2} \frac{m}{n}.
\end{equation}
 (See equation (2.5) in \cite{rudnick}; the bound there is stated for quadratic characters, but it is easily generalized to any primitive characters.)

We will also use the P\'olya--Vinogradov theorem in function fields (see \cite[ Proposition $2.1$]{hsu}). Namely, for $\chi$ a nonprincipal character modulo $Q$, we have that
\begin{equation}
\label{pv}
\Big|  \sum_{f \in \mathcal{M}_n} \chi(f) \Big| \ll \sqrt{|Q|}.
\end{equation}

\section{Order $\ell$ Gauss sums}
\label{gauss-sums-section}

Let $c, r \in \cM$. We define the (shifted) Gauss sum
\begin{align} \label{def-GS}
G_\ell(r,c) &:=\sum_{a \pmod{c}} \Big( \frac{a}{c} \Big)_\ell \; {e_q} \Big( \frac{r a}{c} \Big),
\end{align}
where $e_q$ is defined in \eqref{exp}.
More generally, we also define for any integer $j$,
\begin{align*} %\label{def-GS-generalized}
G_{\ell, j}(r,c) &:=\sum_{a \pmod{c}} \legl{a}{c}^j \; {e_q} \Big( \frac{r a}{c} \Big).
\end{align*}

We will also use the following alternate notation for more general characters $\chi$ modulo $c$, namely,
\[
G(r,\chi)=\sum_{a \pmod{c}}\chi(a) \; {e_q} \Big( \frac{r a}{c} \Big).
\]

If $(b,c)=1$, then it is easy to see that 
\begin{equation}\label{eq:GS}G_{\ell,j}(br,c)= \overline{\legl{b}{c}^j} G_{\ell,j}(r,c).\end{equation}

\begin{lemma} \label{gauss}
Let $r \in \F_q[t]$, and $c_1, c_2 \in \cM$.
\begin{itemize}
 \item[(i)] If $(c_1 ,c_2)=1$, then 
\begin{eqnarray*} G_{\ell,j}(r,{c_1 c_2}) 
&=& \legl{c_2}{c_1}^{2j} G_{\ell,j}(r, c_1 )G_{\ell,j}(r, {c_2}) \\
&=& G_{\ell,j}(rc_2^{\ell-2}, c_1) G_{\ell,j}(r,c_2). \end{eqnarray*}
\item[(ii)] If $r=\tilde{r} \pi^\alpha$ where $\pi \nmid \tilde{r}$, then
\[G_{\ell,j}(r, \pi^i) = \left\{\begin{array}{ll}
0 & \mbox{ if }i \leq \alpha \mbox{ and } ij \not \equiv 0 \pmod{\ell},\\
                      \phi(\pi^{i})& \mbox{ if }i \leq \alpha \mbox{ and } ij \equiv 0 \pmod{\ell},\\
                        -q^{(i-1)\deg(\pi) }&\mbox{ if }i=\alpha+1 \mbox{ and } ij \equiv 0 \pmod{\ell},\\
                       \varepsilon(\chi_\pi^{ij})  \omega(\chi_\pi^{ij}) \legl{\tilde{r}}{\pi^i}^{-1}  q^{(i-\frac{1}{2})\deg(\pi) } &\mbox{ if }i=\alpha+1 \mbox{ and }  ij \not \equiv 0 \pmod{\ell},\\
                                               0 & \mbox{ if }i \geq \alpha+2,
                         \end{array}
 \right.
\]
\end{itemize}
where $\phi$ is the Euler $\phi$-function for polynomials and we recall that $\varepsilon(\chi)=1$ when $\chi$ is even and $\varepsilon(\chi)$ is given in \eqref{epsilonfactor} for $\chi$ odd. We note that $\chi_{\pi}^{ij}$ is even if $\ell \mid \deg(\pi^{ij})$.  \end{lemma}

\begin{proof} 
The proof of (i) follows the same lines as in  \cite{Florea1} and Lemma 2.12 in \cite{DFL}.
After writing $a \pmod {c_1c_2}$ as $a=a_1 c_1 + a_2 c_2$ for $a_1 \pmod {c_2}$ and $a_2 \pmod {c_1}$, we have,
\begin{align*}
G_{\ell,j}(r, c_1c_2)=&\legl{c_1}{c_2}^j  \legl{c_2}{c_1}^j   \sum_{a_1 \pmod{c_2}}\sum_{a_2 \pmod{c_1}}
\legl{a_2}{c_1}^j \legl{a_1}{c_2}^je_q\left(\frac{a_1r}{c_2}\right)e_q\left(\frac{a_2r}{c_1}\right)\\
=& \legl{c_2}{c_1} ^{2j} G_{\ell,j}(r, c_1)G_{\ell,j}(r, c_2),
\end{align*}
by reciprocity. The second line of (i) follows from \eqref{eq:GS}.

To prove (ii), first we assume that $i \leq \alpha$. Then
$$G_{\ell,j}(r,\pi^i)= \sum_{a \pmod {\pi^i}} \legl{a}{\pi^i}^je_q(a\tilde{r} \pi^{\alpha-i}).$$
The exponential above is equal to $1$ since $a\tilde{r}\pi^{\alpha-i} \in \F_q[t]$, and if $ij \equiv 0 \pmod{\ell}$, then $\legl{a}{\pi^i}^j=1$ when $(a,\pi)=1$ and therefore the character is even. The conclusion  follows directly  in this case. If $ij \not \equiv 0 \pmod \ell$, the character is odd and the conclusion follows from orthogonality of characters.

Now assume that $i = \alpha+1$. Write $a \pmod {\pi^i}$ as $a = \pi A+C$, with $A \pmod {\pi^{i-1}}$ and $C \pmod \pi$. Then
$$G_{\ell,j}(r, \pi^i) = \sum_{A \pmod {\pi^{i-1}}} \sum_{C \pmod \pi} \legl{C}{\pi^i}^j e_q\left( \frac{C\tilde{r}}{\pi} \right )= q^{(i-1)\deg(\pi) } \legl{\tilde{r}^{-1}}{\pi^i}^j \sum_{C \pmod \pi}  \legl{C}{\pi^i}^j e_q\left( \frac{C}{\pi}\right ).$$
If $ij \equiv 0 \pmod \ell$, then $ \legl{\tilde{r}^{-1}}{\pi^i}^j = 1$ and
$$  \sum_{C \pmod \pi}  \legl{C}{\pi^i}^j  e_q\left( \frac{C}{\pi}\right ) = \sum_{\substack{C  \pmod \pi \\ C \neq 0}}  e_q\left( \frac{C}{\pi}\right )  = -1,$$ and the conclusion follows. Now assume that $ij \not \equiv 0 \pmod \ell$.
Then
\[ \sum_{C \pmod \pi} \legl{C}{\pi^i}^j  e_q\left( \frac{C}{\pi}\right ) = \left\{\begin{array}{ll}-q \sum_{f \in \mathcal{M}_{\deg(\pi)-1}} \legl{f}{\pi^i}^j  & \ell\mid ij \deg(\pi), \\
           \tau(\chi_\pi^{ij})\sum_{f \in \mathcal{M}_{\deg(\pi)-1}} \legl{f}{\pi^i}^j  &\ell \nmid ij \deg(\pi).
\end{array} \right. \]
(See equation (3.5) in \cite{Florea1}.) Using \eqref{eq:signfe}, we can rewrite this as
\[\sum_{C \pmod \pi} \legl{C}{\pi^i}^j  e_q\left( \frac{C}{\pi}\right ) = \left\{\begin{array}{ll}\omega(\chi_\pi^{ij})q^{\deg(\pi)/2} & \ell\mid ij\deg(\pi), \\
          \varepsilon(\chi_\pi^{ij}) \omega(\chi_\pi^{ij})q^{\deg(\pi)/2} &\ell \nmid ij\deg(\pi). \end{array} \right.\]
(See Lemma 2.3 and Corollary 2.4 in \cite{DFL}.)

Thus, we get 
\[G_{\ell,j}(r, \pi^i) = \left\{\begin{array}{ll}\omega(\chi_\pi^{ij}) \legl{\tilde{r}^{-1}}{\pi^i}  q^{(i-\frac{1}{2})\deg(\pi) }   &\ell \mid ij\deg(\pi),\\ \\ \varepsilon(\chi_\pi^{ij}) \omega(\chi_\pi^{ij}) \legl{\tilde{r}^{-1}}{\pi^i}  q^{(i-\frac{1}{2})\deg(\pi) }   &\ell \nmid ij \deg(\pi). \end{array} \right.\]

If $ i \geq \alpha+2$, then again the proof goes through exactly as in \cite{Florea1}.
\end{proof}

\begin{corollary} \label{size-GS} If $(c, r)=1$, then $|G_\ell(r, c)| = \mu^2(c) |c|^{1/2}$.
\end{corollary}

\begin{lemma} (Poisson Summation)
\label{poisson}Let $\chi$ be a character modulo a monic polynomial $f$, and let $m$ be a positive integer.
If $\deg(f) \equiv 0 \pmod \ell$, then
$$\sum_{h \in \mathcal{M}_m} \chi(h) = \frac{q^m}{|f|} \Big[  G_{\ell}(0,\chi)+ (q-1) \sum_{V \in \mathcal{M}_{\leq \deg(f)-m-2}} G_{\ell}(V,\chi) - \sum_{V \in \mathcal{M}_{\deg(f)-m-1}} G_{\ell}(V,\chi)\Big].$$
If $\deg(f) \not\equiv 0 \pmod \ell$,  then 
$$\sum_{h \in \mathcal{M}_m} \chi(h)= \frac{q^{m+1/2}}{|f|} \overline{\varepsilon(\chi)} \sum_{V \in \mathcal{M}_{\deg(f)-m-1}} G_{\ell}(V,\chi).$$
\end{lemma}

\begin{proof}
The proof follows the proof of Proposition $3.1$ in \cite{Florea1}. See also the proof of Proposition $2.13$ in \cite{DFL}.
\end{proof}
\section{The generating series for order $\ell$ Gauss sums}
\label{gen-series}
For $r \in \F_q[t]$, we define
\begin{align*}
 \psi^{(i)}(r,u)  := (1-q^{\ell}u^\ell)^{-1}\sum_{\substack{F \in \cM\\ \deg(F) \equiv i \pmod{\ell}}} G_\ell(r, F) u^{\deg(F)}.
\end{align*}
With the change of variables $u=q^{-s}$ and by abuse of notation, we will also use the alternate definition
\begin{align*}
 \psi^{(i)}(r,s)  := (1-q^{\ell-\ell s})^{-1}\sum_{\substack{F \in \cM\\ \deg(F) \equiv i \pmod{\ell}}} G_\ell(r, F) |F|^{-s}.
\end{align*}

We will gather here some facts from \cite{DFL} about  the generating series $\psi^{(i)}(r; u)$ that will be needed throughout this article.

We have that $\psi^{(i)}(r,u)$ satisfies a functional equation (see Proposition 2.1 in \cite{hoffstein} as well as Proposition $3.2$ and equation $(25)$ in \cite{DFL}). Namely, we have the following:
\begin{align}
\label{funct_eq}
(1-q^{\ell+1} u^{\ell}) \psi^{(i)}(r,u) =(qu)^{\deg(r)}\Big[a_1(u) \psi^{(i)} \Big(r; \frac{1}{q^2u} \Big)+a_2(u) \psi^{([\deg(r)+1-i]_{\ell})}\Big( r;  \frac{1}{q^2u} \Big) \Big],
\end{align}
where $$a_1(u) = -q^2u (qu)^{-[\deg(r)+1-2i]_{\ell}}(1-q^{-1})$$ and $$a_2(u)= -W_{r,i} (qu)^{1-\ell} (1-q^{\ell} u^{\ell}),$$
and $W_{r,i}=\tau(\chi_\ell^{2i-1}\overline{\chi}_r)$, with $\chi_\ell(\alpha):=\Omega^{-1}(\alpha^\frac{q-1}{\ell})$.
From \cite{DFL}, equation $(26)$, we know that
\begin{equation}
\label{psi-formula}
\psi^{(i)}(r,u) = \frac{u^i P(r,i,u^{\ell})}{1-q^{\ell+1} u^{\ell}},
\end{equation} where 
$P(r,i,x)$ is a polynomial of degree at most $[(1+\deg(r)-i)/\ell]$ in $x$.
We let 
\begin{equation}C(r, k) = \sum_{F \in \mathcal{M}_{k}} G_{\ell}(r, F).\label{ccoeff}
\end{equation}
Then, for any $B > [(1+\deg(r)-i)/\ell]$, we have
\begin{eqnarray}
 P(r,  i, x) 
\label{formula-for-P}
&=& \frac{1-q^{\ell+1}x}{1-q^{\ell}x}  \sum_{0 \leq j < B} C(r, i+\ell j) x^j +  \frac{C(r, i + \ell B)}{1-q^{\ell}x} x^{B}.
\end{eqnarray}
(See equation (27) in \cite{DFL}.\footnote{Note that the condition is wrongly stated as $B \geq [(1+\deg(r)-i)/\ell]$ in \cite{DFL}.})

We will now prove the following convexity bound on the size of $\psi^{i)}(r,u)$, for $u=q^{-s}$.

\begin{theorem}
\label{residue}
Let $u=q^{-s}$ and $\sigma=\re(s)$. For $i \in \{0,1,\ldots,\ell-1\}$, $q^{-3/2} \leq |u| \leq q^{-1/2}$ and $|u^\ell-q^{-\ell-1}|>\delta$, we have
$$\psi^{(i)}(r,s) \ll_{\delta,\varepsilon} |r|^{\frac12 \left( \frac{3}{2} - \sigma+ \varepsilon \right)}.$$
Moreover, for the residue at the poles  $u = q^{-1-\frac{1}{\ell}} \xi_{\ell}$ (where $\xi_\ell$ is any $\ell$th root of unity), we have
$$
\mathrm{Res}_{u=q^{-1-\frac{1}{\ell}} \xi_{\ell}
}\psi^{(i)}(r,s) \ll \begin{cases}|r|^{\frac{\ell-2}{4\ell} + \varepsilon} & \ell>3,\\
  |r_1|^{-\frac{1}{6}} & \ell=3,
                                       \end{cases}
$$
where $r_1$ is such that $r=r_1r_2^2r_3^3$ with $r_1,r_2$ square-free and $(r_1,r_2)=1$.
\end{theorem}
\begin{remark}
 The improvement on the bound for the residue in the case $\ell=3$ follows from \cite[Lemma 3.9]{DFL}.
\end{remark}
\begin{proof}
The bound for the residue when $\ell>3$ follows directly from the convexity bound that we will now prove, following a very similar argument to the proof of Lemma $3.10$ in \cite{DFL}. We will only briefly sketch it for the sake of completeness.

 Exactly as in the argument leading to equation (38) in \cite{DFL}, we have that for $\sigma \leq 3/2$, 
$$|\psi^{(i)} (r,s)| \ll |r|^{3/2-\sigma +\varepsilon}.$$ Therefore,
$$ |\psi^{(i)} (r,s)| \ll |r|^{\varepsilon}$$ when $\sigma = 3/2$.
Now using the functional equation \eqref{funct_eq}, we have that when $\sigma=1/2$ and $|u^\ell-q^{-\ell-1}|>\delta$,
$$|\psi^{(i)} (r,s)| \ll |r|^{1/2+\varepsilon}.$$
We consider the function $\Phi(r, s) = (1-q^{\ell+1-\ell s}) (1-q^{\ell s-\ell+1}) \psi^{(i)}(r,s) \psi^{(i)}(r, 2-s).$  Then $\Phi(r,s)$ is holomorphic for $1/2 \leq \sigma \leq 3/2$ and $\Phi(r,s) \ll |r|^{1/2+\varepsilon}$ for $\sigma=1/2$ and $\sigma = 3/2$. Using the Phragm\'en--Lindel\"of principle, it follows that
$$\Phi(r,s) \ll |r|^{1/2+\varepsilon},$$ for $1/2 \leq \sigma \leq 3/2$.
For ease of notation, let $j = [\deg(f)+1-i]_{\ell}$. Similarly as in equation (40) in \cite{DFL},  we have that for $1/2\leq \sigma \leq 3/2$ and $|u^{\ell}-q^{-\ell-1}| > \delta$,
\begin{align*}
(1-q^{\ell+1-\ell s}) ( 1-q^{\ell s-\ell+1}) \Big(a_1(s) \psi^{(i)}(r,2-s)^2  +a_2(s) \psi^{(i)}(r,2-s) \psi^{(j)}(r,2-s) \Big) \ll |r|^{\sigma-1/2+\varepsilon}.
\end{align*} 
If $i \equiv j \pmod{\ell}$ (or, equivalently, $\deg(f) +1 \equiv 2i \pmod{\ell}$), then from the above we get that
$$ \psi^{(i)}(r,2-s) + \psi^{(j)}(r,2-s) \ll |r|^{\frac{1}{2}(\sigma-1/2+\varepsilon)}.$$
If $i \not \equiv j \pmod{\ell}$, then we proceed as in equation (43) of \cite{DFL}, and we get that
\begin{align}
(1-q^{\ell+1-\ell s})& ( 1-q^{\ell s-\ell+1}) \Big(  \psi^{(i)}(r,2-s)+ \psi^{(j)}(r,2-s)\Big) \Big( a_1(s) \psi^{(i)}(r,2-s) + a_2(s) \psi^{(j)}(r,2-s)\Big) \nonumber \\
& \ll |r|^{\sigma-1/2+\varepsilon}.\label{bd11}
\end{align}
Switching between $i$ and $j$ (since $i \not \equiv j \pmod{\ell}$), we get, from  the functional equation \eqref{funct_eq}, that there exist  absolutely bounded constants $b_1(s)$ and $b_2(s)$ such that
$$\psi^{(j)}(r,s) = b_1(s) |r|^{1-s} \psi^{(j)}(r,2-s) + b_2(s) |r|^{1-s} \psi^{(j)}(r,2-s).$$
If $(a_1,a_2)$ and $(b_2,b_1)$ are not linearly independent, then from the above and the functional equation \eqref{funct_eq}, we get that there exists some absolutely bounded $\lambda(s)$ such that $\psi^{(j)}(r,s) = \lambda(s) \psi^{(i)}(r,s)$ and hence using equation \eqref{bd11}, we have that
$$\psi^{(i)}(r,2-s) \ll |r|^{1/2(\sigma-1/2+\varepsilon)}.$$
If $(a_1,a_2)$ and $(b_2,b_1)$ are linearly independent, then we proceed similarly as in equation \eqref{bd11}, and we get that
\begin{align*}
(1-q^{\ell+1-\ell s})& ( 1-q^{\ell s-\ell+1}) \Big(  \psi^{(j)}(r,2-s)+ \psi^{(i)}(r,2-s)\Big) \Big( b_1(s) \psi^{(j)}(r,2-s) + b_2(s) \psi^{(i)}(r,2-s)\Big) \\
& \ll |r|^{\sigma-1/2+\varepsilon}.
\end{align*}
Combining this and equation \eqref{bd11}, we get that
$$(\psi^{(i)}(r,2-s)+\psi^{(j)}(r,2-s)) \psi^{(i)}(r,s) \ll |r|^{\sigma-1/2+\varepsilon},$$ and
$$(\psi^{(i)}(r,2-s)+\psi^{(j)}(r,2-s)) \psi^{(j)}(r,s) \ll |r|^{\sigma-1/2+\varepsilon}.$$
From the two equations above we get that 
$$\psi^{(i)}(r,2-s) + \psi^{(j)}(r,2-s) \ll |r|^{1/2(\sigma-1/2+\varepsilon)}.$$
Summing the equation above with $i =0, 1, \ldots, \ell-1$ and replacing $2-s$ by $s$ gives the bound for $\Psi(r,u):=\sum_{i=0}^{\ell-1}\psi^{(i)}(r,u)$. Since
\[\ell \psi^{(i)}(r,u)=\sum_{j=0}^{\ell-1}\xi_\ell^{-ij} \Psi(r,\xi_\ell^ju),\]
where $\xi_\ell$ is any primitive $\ell$th root of unity,
we deduce the bound for each individual $|\psi^{(i)}(r,s)|$, which finishes the proof.
 \end{proof}

We can also prove the following upper bound for $\psi^{(i)}(r,u)$.
\begin{lemma} Let $\ell=3$, and 
 $u=q^{-s}$ with $q^{-3/2} \leq |u| \leq q^{-1/2}$, and $|u^{3}-q^{-4}|>\delta$. Let $\sigma=\re(s)$. Then,
 $$\psi^{(i)}(r,u)\ll |r|^{\varepsilon} + |r|^{4/3-\sigma}|r_1|^{-1/6},$$
where $r_1$ is as in Theorem \ref{residue}.
\label{new_lemma}
\end{lemma}
\begin{proof}
Using the equation right after equation $(28)$ in \cite{DFL}, we have that 
$$\text{Res}_{u^3=q^{-4}} \psi^{(i)}(r,u) = \frac{C(r,i')}{(1-q^{-1})q^{4/3(i'-i)}},$$ where $i' \equiv i \pmod 3$ and $i'>\deg(r)$, and where recall that $C(r,i')$ is given in \eqref{ccoeff}. Since $$\text{Res}_{u^3=q^{-4}} \psi^{(i)}(r,u) \ll |r_1|^{-1/6},$$ where $r_1$ is chosen as in Theorem \ref{residue}, 
 we have
$$C(r,i') \ll q^{4i'/3}|r_1|^{-1/6},$$ for $i'> \deg(r)$. On the other hand, we also showed in the proof of the convexity bound (\cite[Theorem 3.10]{DFL}) that for any $k$,
$$C(r,k) \ll q^{3k/2}|r|^{\varepsilon}.$$
If $|u|=q^{-\sigma}$ and $1/2\leq \sigma \leq 4/3$, then from equations \eqref{psi-formula}, \eqref{formula-for-P} and the two equations above, we get that
\begin{equation*}
%\label{newbound}
|\psi^{(i)}(r,u)| \ll \sum_{0 \leq j <B} q^{3j/2-3j \sigma} |r|^{\varepsilon} + |r|^{4/3-\sigma}|r_1|^{-1/6} \ll |r|^{\varepsilon}+ |r|^{4/3-\sigma}|r_1|^{-1/6}.
\end{equation*}
\end{proof}

We also have the following approximate functional equation. 
\begin{theorem}[The approximate functional equation]
\label{afe}
For any $n \leq \deg(r)$, we have that
\begin{align*}
\psi^{(i)}(r,u) &= \frac{1}{1-q^{\ell+1} u^{\ell}} \Big[ \sum_{\substack{\deg(F) \leq n \\ \deg(F) \equiv i \pmod{\ell}}} G_\ell(r,F) u^{\deg(F)} b_1(\deg(F),u) \\
&+(qu)^{\deg(r)}  \sum_{\substack{\deg(F) \leq \deg(r)-n-\rho \\ \deg(F) \equiv i \pmod{\ell}}} \frac{G_\ell(r,F)}{(q^2u)^{\deg(F)}} b_2(\deg(F),u) \nonumber  \\
&+ (qu)^{\deg(r)}\sum_{\substack{\deg(F) \leq \deg(r)-n \\ \deg(F) \equiv j \pmod{\ell}}} \frac{G_\ell(r,F)}{(q^2u)^{\deg(F)}} b_3(\deg(F),u) \Big], \nonumber
\end{align*}
where $\rho = [\deg(r)+1-2i]_{\ell}$, $j=[\deg(r)+1-i]_{\ell}$, and
where $b_1(\deg(F),u)$, $b_2(\deg(F),u)$, $b_3(\deg(F),u)$ are such that when $|u|=q^{-\sigma}$ and $\sigma>1$, $b_1(\deg(F),u)= O(1)$, $b_3(\deg(F),u) =O(1)$ and $b_2(\deg(F),u) = O \Big( q^{(\sigma-1)( \deg(r) - n - \deg(F))}\Big)$.
\end{theorem}
\begin{proof}
We start by considering the integral
$$I_{-3}(n,u) = \frac{1}{2 \pi i} \oint_{|v|=q^{-3}} \frac{ (1- q^{\ell+1} (uv)^{\ell}) \psi^{(i)}(r, uv)}{v^{n+1}(1-v)} \, dv.$$
Recall that
$$\psi^{(i)}(r,uv) = \frac{1}{1-(quv)^{\ell}} \sum_{\deg(F) \equiv i \pmod{\ell}} G_\ell(r,F) (uv)^{\deg(F)}.$$
Integrating in the expression for $I_{-3}(n,u)$ term by term, we get that
\begin{align*}
%\label{first_piece}
I_{-3}(n,u) =& \sum_{\substack{\deg(F) \leq n \\ \deg(F) \equiv i \pmod{\ell}}} G_\ell(r,F) u^{\deg(F)} \sum_{k \leq [\frac{n-\deg(F)}{\ell} ]} (qu)^{\ell k}\\
& - q^{\ell+1} u^{\ell} \sum_{\substack{\deg(F) \leq n-\ell \\ \deg(F) \equiv i \pmod{\ell}}} G_\ell(r,F) u^{\deg(F)} \sum_{k \leq [\frac{n-\deg(F)-\ell}{\ell} ]} (qu)^{\ell k}.\nonumber
\end{align*}
In the expression for $I_{-3}(n,u)$, we shift the contour to $|v|=q^3$, and encounter the pole at $v=1$. It follows that 
$$I_{-3}(n,u) = (1-q^{\ell+1}u^{\ell}) \psi^{(i)}(r,u) + I_3(n,u).$$
In the expression for $I_3(n,u)$, we use the functional equation \eqref{funct_eq} for $\psi$, and we have that
\begin{align*}
I_3(n,u) &= \frac{1}{2 \pi i} \oint_{|v|=q^3} \frac{(quv)^{\deg(r)}}{v^{n+1}(1-v)} \Big[ a_1(uv) \psi^{(i)} \Big(r,\frac{1}{q^2uv} \Big)+ a_2(uv) \psi^{(j)} \Big( r, \frac{1}{q^2uv}\Big)\Big] \, dv,
\end{align*}
where  we have written $j=[\deg(r)+1-i]_{\ell}$. Using the fact that
$$\frac{1}{2 \pi i} \oint_{|v|=q^3} \frac{v^a}{v^n(1-v)} \, dv = 
\begin{cases}
0 & \mbox{ if } n>a,\\
-1 &\mbox{ if } n \leq a,
\end{cases}
$$ 
we get that
\begin{align*}
I_3(n,u) &= (qu)^{\deg(r)} q^{2-\rho} u^{1-\rho} (1-q^{-1}) \sum_{\substack{\deg(F) \leq \deg(r)-n-\rho \\ \deg(F) \equiv i \pmod{\ell}}} \frac{G_\ell(r,F)}{(q^2u)^{\deg(F)}} \sum_{k \leq \Big[ \frac{\deg(r)-n-\rho-\deg(F)}{\ell} \Big]} \frac{1}{(qu)^{\ell k}} \\
&+(qu)^{\deg(r)+1-\ell}  W_{r,i} \sum_{\substack{\deg(F) \leq \deg(r)-n-\ell \\ \deg(F) \equiv j \pmod{\ell}}} \frac{G_\ell(r,F)}{(q^2u)^{\deg(F)}} - (qu)^{\deg(r)+1} W_{r,i} \sum_{\substack{\deg(F) \leq \deg(r)-n \\ \deg(F) \equiv j \pmod{\ell}}} \frac{G_\ell(r,F)}{(q^2u)^{\deg(F)}},
\end{align*}
where we have written $\rho = [\deg(r)+1-2i]_{\ell}$.

 Putting things together, we get that
\begin{align*}
(1-q^{\ell+1} & u ^{\ell}) \psi^{(i)}(r,u) = \sum_{\substack{\deg(F) \leq n \\ \deg(F) \equiv i \pmod{\ell}}} G_\ell(r,F) u^{\deg(F)} \sum_{k \leq [\frac{n-\deg(F)}{\ell} ]} (qu)^{\ell k}\\
& - q^{\ell+1} u^{\ell} \sum_{\substack{\deg(F) \leq n-\ell \\ \deg(F) \equiv i \pmod{\ell}}} G(r,F) u^{\deg(F)} \sum_{k \leq [\frac{n-\deg(F)-\ell}{\ell} ]} (qu)^{\ell k} \\
&-(qu)^{\deg(r)} q^{2-\rho} u^{1-\rho} (1-q^{-1}) \sum_{\substack{\deg(F) \leq \deg(r)-n-\rho \\ \deg(F) \equiv i \pmod{\ell}}} \frac{G(r,F)}{(q^2u)^{\deg(F)}} \sum_{k \leq \Big[ \frac{\deg(r)-n-\rho-\deg(F)}{\ell} \Big]} \frac{1}{(qu)^{\ell k}} \\
&-(qu)^{\deg(r)+1-\ell}  W_{r,i} \sum_{\substack{\deg(F) \leq \deg(r)-n-\ell \\ \deg(F) \equiv j \pmod{\ell}}} \frac{G(r,F)}{(q^2u)^{\deg(F)}} + (qu)^{\deg(r)+1} W_{r,i} \sum_{\substack{\deg(F) \leq \deg(r)-n \\ \deg(F) \equiv j \pmod{\ell}}} \frac{G(r,F)}{(q^2u)^{\deg(F)}}.
\end{align*}
Rearranging the above, the conclusion follows. 
\end{proof}

\section{The generating series for order $\ell$ Gauss sums with coprimality and divisibility conditions}
\label{sec:remove}
Let $v, r, a \in \F_q[t]$, and $i \in \Z$.
We show in this section how to express the more general generating series
\begin{align*}
 \psi_v^{(i)}(r, a, u)  := (1-q^{\ell} u^{\ell})^{-1}\sum_{\substack{F \in \cM\\ \deg(F) \equiv i \pmod{\ell} \\(F, v)=1\\a \mid F}}  G_\ell(r, F) u^{\deg(F)}
 \end{align*}
 in terms of the generating series $\psi^{(i)}(r, u)$ of the previous section. Notice that $\psi^{(i)}(r, u)  =  \psi_1^{(i)}(r, 1, u)$. In addition, we will employ the following notation
 \begin{align*}
\psi_v^{(i)}(r, u) &=  \psi_v^{(i)}(r, 1, u).
\end{align*}
The following generalizes Lemma $3.5$ in \cite{DFL} to general $\ell$. 
\begin{lemma} \label{relations}
Let $\pi$ be a prime such that $\pi \nmid r$. Let $r_0 \mid r$. Then we have
\begin{align}\label{chantal1}
\psi_{r_0\pi}^{(i)}(r ,u) &=  \psi_{r_0}^{(i)}(r, u) -   G_\ell(r, \pi) u^{\deg(\pi)} \psi_{r_0\pi}^{(i-\deg(\pi))}(r \pi^{\ell-2},u), \\
\psi_{r_0 \pi}^{(i)}(r \pi^{\ell-1},u) &=  \left( 1 - q^{(\ell-1) \deg(\pi)} u^{\ell \deg(\pi)} \right)^{-1} \psi_{r_0}^{(i)}(r \pi^{\ell-1}, u),\label{chantal2}
\end{align}
and for $1 \leq j \leq \ell-2$,
\begin{align}
\psi_{r_0 \pi}^{(i)}(r \pi^j ,u) &= \psi_{r_0}^{(i)}(r \pi^j,u)
 - q^{j \deg(\pi)} G_{\ell, j+1}(r, \pi) u^{(j+1) \deg(\pi)} \psi_{r_0 \pi}^{(i-(j+1)\deg(\pi))}(r \pi^{\ell-j-2},u).\label{chantal3}
 \end{align}
 \end{lemma}

\begin{proof}
We consider the first identity. By Lemma \ref{gauss}, we have that for $\pi \nmid rF$,  $G_\ell(r, \pi^k F)=0$ unless, possibly, for $k=0,1$. Writing the sum $\psi_{r_0 \pi}^{(i)}(r,u)$ as the sum over $F$ with no divisibility condition with respect to $\pi$ minus the sum with the condition that $\pi \mid F$, we have
\begin{align}\label{eq:proofidentity1}
\psi_{r_0 \pi}^{(i)}(r,u) = \psi_{r_0}^{(i)}(r, u)
- (1-q^{\ell}u^\ell)^{-1} \sum_{\substack{F \in \cM \\ (F, \pi)=(F,r_0)=1 \\ \deg(F) \equiv i - \deg(\pi) \pmod{\ell}}} G_\ell(r, \pi F) u^{\deg(F)+ \deg(\pi)}.
\end{align}
By Lemma \ref{gauss}(i),
\begin{align*}
G_\ell(r, \pi F) &= G_\ell(\pi^{\ell-2} r , F)  G_\ell(r, \pi).
\end{align*}
Incorporating the above into \eqref{eq:proofidentity1} gives \eqref{chantal1}.

 For the second identity, we have 
\begin{align*}
\psi_{r_0 \pi}^{(i)}(r \pi^{\ell-1},u)
= \psi_{r_0}^{(i)}(r \pi^{\ell-1}, u)
- (1-q^{\ell}u^\ell)^{-1} \sum_{\substack{F \in \cM \\ (F,\pi)=(F, r_0)=1 \\ \deg(F) \equiv i \pmod{\ell}}} G_\ell(r \pi^{\ell-1}, \pi^{\ell} F) u^{\deg(F)+\ell \deg(\pi)},
\end{align*}
where we have used that  for $k\not =0$, $ G_\ell(r \pi^{\ell-1}, \pi^{k} F)$ can only be nonzero for $k=\ell$.

By Lemma \ref{gauss}(i),
\begin{align*}
G_\ell(r \pi^{\ell-1}, \pi^{\ell} F)
&= - G_\ell(r \pi^{\ell-1}, F)  q^{(\ell-1)\deg(\pi)} ,
\end{align*}
and 
$$
\psi_{r_0\pi}^{(i)}(r \pi^{\ell-1},u) =  \psi_{r_0}^{(i)}(r \pi^{\ell-1}, u) +    q^{(\ell-1) \deg(\pi)} u^{\ell \deg(\pi)} \psi_{r_0 \pi}^{(i)}(r \pi^{\ell-1},u),
$$
which can be rewritten as \eqref{chantal2}.

For \eqref{chantal3}, for any $1 \leq j \leq n-2$, we have
\begin{align}\label{eq:psij}
\psi_{r_0\pi}^{(i)}(r \pi^{j},u)
= \psi_{r_0}^{(i)}(r \pi^{j}, u)
-(1-q^{\ell} u^{\ell})^{-1} \sum_{\substack{F \in \cM \\ (F,\pi)=(F,r_0)=1 \\ \deg(F) \equiv i - (j+1) \deg(\pi) \pmod{n}}} G_\ell(r \pi^{j}, \pi^{j+1} F) u^{\deg(F)+(j+1) \deg(\pi)},
\end{align}
where, again, we have used that  for $k \not =0$, $ G_\ell(r \pi^{j}, \pi^{k} F)$ can only be nonzero for $k=j+1$.

By Lemma \ref{gauss}(i),
\begin{align*}
G_\ell(r \pi^{j}, \pi^{j+1} F)
&= G_\ell(r \pi^{\ell-j-2}, F)  G_\ell(r \pi^{j}, \pi^{j+1}),
\end{align*}
where
\begin{align*} 
G_\ell(r \pi^j, \pi^{j+1}) &= \sum_{\alpha \pmod{\pi^{j+1}}} \legl{\alpha}{\pi}^{j+1} e \left( \frac{\alpha r}{\pi} \right)  = q^{j\deg(\pi)} G_{\ell, j+1}(r, \pi).
\end{align*}
Combining with \eqref{eq:psij}, we get the result. 
\end{proof}
\begin{lemma} \label{lem:argh}
Let $r_1, \dots, r_{\ell-1} \in \F_q[t]$ be square-free and pairwise coprime.
We have 
\begin{align*}
\psi_{r_1 \cdots r_{\ell-1}}^{(i)}(r_1 r_2^2 \cdots r_{\ell-1}^{\ell-1},u)
= \prod_{\pi \mid r_{\ell-1}} \left( 1 - q^{(\ell-1)\deg(\pi) } u^{\ell\deg(\pi)} \right)^{-1} \psi^{(i)}_{r_1 \cdots r_{\ell-2}}(r_1 r_2^2 \cdots r_{\ell-1}^{\ell-1},u).
\end{align*}

For $0\leq j\leq \ell-2$, we have 
\begin{align}\label{eq:general}
&\psi^{(i)}_{r_1 r_2 \cdots r_{j}} (r_1 r_2^2 \cdots r_{\ell-1}^{\ell-1}, u) = \prod_{\pi \mid r_j}
\left(1 - q^{(\ell-1)\deg(\pi)} \legl{-1}{\pi}^{j+1}u^{\ell \deg(\pi)}\right)^{-1}\nonumber \\
&\times  \sum_{d \mid r_{j}} \mu(d) q^{j \deg(d)} G_{\ell,j+1} \left( \frac{r_1 r_2^2 \cdots r_{\ell-1}^{\ell-1}}{d^{j}} , d \right) u^{(j+1) \deg(d)}  \psi^{(i-(j+1)\deg(d))}_{r_1 r_2 \cdots r_{j-1}} \left( \frac{r_1 r_2^2 \cdots r_{\ell-1}^{\ell-1}}{d^{2j-\ell+2}}, u \right).
\end{align}

\end{lemma}

\begin{proof}

We first deal with the primes dividing $r_{\ell-1}$, which is a simpler case.
Let $\pi \mid r_{\ell-1}$, and write $r_1 \cdots r_{\ell-1}^{\ell-1} =  R \pi^{\ell-1}$ with $(R, \pi)=1$. Let $R_0$ be a divisor of $R$.
Then, by equation \eqref{chantal2},
\begin{align*}
\psi_{R_0\pi}^{(i)}(R \pi^{\ell-1},u)
&= \left( 1 - q^{(\ell-1)\deg(\pi) } u^{\ell\deg(\pi)} \right)^{-1} \psi_{R_0}^{(i)}(R \pi^{\ell-1},u),
\end{align*}
and working by induction on the prime factors of $r_{\ell-1}$, we get
\begin{align*}
\psi_{r_1 \cdots r_{\ell-1}}^{(i)}(r_1 r_2^2 \cdots r_{\ell-1}^{\ell-1},u)
= \prod_{\pi \mid r_{\ell-1}} \left( 1 - q^{(\ell-1)\deg(\pi) } u^{\ell\deg(\pi)} \right)^{-1} \psi^{(i)}_{r_1 \cdots r_{\ell-2}}(r_1 r_2^2 \cdots r_{\ell-1}^{\ell-1},u).
\end{align*}

Let $0 \leq j \leq \ell-2$ and remark that  $0 \leq \ell-j-2 \leq \ell-2$ as well. Assume that $\pi$ is a prime such that $\pi \nmid r$ and $r_0\mid r$.  Then, using equation \eqref{chantal3} from Lemma \ref{relations} twice, we obtain,
\begin{align*}
\psi_{r_0 \pi}^{(i)}(r \pi^j ,u) &= \psi_{r_0}^{(i)}(r \pi^j,u)
 -  q^{j \deg(\pi)} G_{\ell, j+1}(r, \pi)
 u^{(j+1) \deg(\pi)} \psi_{r_0 \pi}^{(i-(j+1)\deg(\pi))}(r \pi^{\ell-j-2},u),    \\
 \psi_{r_0 \pi}^{(i-(j+1)\deg(\pi))}(r \pi^{\ell-j-2},u)  &=  \psi_{r_0}^{(i-(j+1) \deg(\pi))}(r \pi^{\ell-j-2},u) \\
&\quad - q^{(\ell-j-2) \deg(\pi)} G_{\ell, \ell-j-1}(r , \pi) u^{(\ell-j-1) \deg(\pi)}  \psi_{r_0 \pi}^{(i)}(r \pi^j,u).
 \end{align*}
Replacing the second equation above in the first one, we find
 \begin{align*}
 \psi_{r_0 \pi}^{(i)}(r \pi^j ,u) &= \psi_{r_0}^{(i)}(r \pi^j,u)
 -   q^{j \deg(\pi)} G_{\ell, j+1}(r, \pi)
 u^{(j+1) \deg(\pi)}   \psi_{r_0}^{(i-(j+1) \deg(\pi))}(r \pi^{\ell-j-2},u) \\
& + q^{(\ell-2)\deg(\pi)}G_{\ell, j+1} (r, \pi)  G_{\ell, \ell-j-1}(r, \pi) u^{\ell \deg(\pi)}  \psi_{r_0 \pi}^{(i)}(r \pi^j,u).
 \end{align*}
 Notice that  
 \[G_{\ell, \ell-j-1}(r, \pi)=\sum_{a\pmod{\pi}}\overline{\legl{a}{\pi}^{j+1}} \; {e_q} \Big( \frac{r a}{\pi} \Big) =\overline{G_{\ell, j+1}(-r, \pi)}= \legl{-1}{\pi}^{j+1}\overline{G_{\ell, j+1}(r, \pi)}.\]
 Therefore, by Lemma \ref{gauss} (ii),  
 \[G_{\ell, j+1} (r, \pi)  G_{\ell, \ell-j-1}(r, \pi)=\legl{-1}{\pi}^{j+1}|G_{\ell, j+1} (r, \pi) |^2=\legl{-1}{\pi}^{j+1} q^{\deg(\pi)}.\]
 Thus, we finally have 
 \begin{align} \label{eq:induction}
 \psi_{r_0 \pi}^{(i)}(r \pi^j ,u) = \frac{\psi_{r_0}^{(i)}(r \pi^j,u)
 - q^{j \deg(\pi)} G_{\ell, j+1}(r, \pi)
u^{(j+1) \deg(\pi)}   \psi_{r_0}^{(i-(j+1) \deg(\pi))}(r \pi^{\ell-j-2},u)}{1 - q^{(\ell-1)\deg(\pi)}
\legl{-1}{\pi}^{j+1}u^{\ell \deg(\pi)} }.
 \end{align}
 Notice that if we replace $j$ by $\ell-2-j$, then the denominator in the above formula is the same, since 
 \[\legl{-1}{\pi}^{j+1}=\legl{-1}{\pi}^{\ell-1-j}\]
 because the above symbol is $\pm 1$, and is therefore invariant under conjugation.

We now prove \eqref{eq:general}  by induction on the number of primes dividing $r_j$, as in \cite{hbp}, page 125. (See also \cite{DFL} for the case $\ell=3$.) If $r_j=\pi$, the result follows directly by \eqref{eq:induction}. Now assume that \eqref{eq:general} is true for $r_j$ being the product of up to $k$ primes, and replace $r_j$ by $r_j\pi_0$ (with $\pi_0\nmid r_j$). By \eqref{eq:induction} we have
\begin{align*}
&\psi^{(i)}_{r_1 r_2 \cdots r_{j}\pi_0 } (r \pi_0^j , u)
\left(1 - q^{(\ell-1)\deg(\pi_0)} \legl{-1}{\pi_0}^{j+1}u^{\ell \deg(\pi_0)}\right)\nonumber \\
&=\psi^{(i)}_{r_1 r_2 \cdots r_{j}} (r\pi_0^j, u)  - q^{j \deg(\pi_0)} G_{\ell, j+1}(r, \pi_0) u^{(j+1) \deg(\pi_0)}   \psi_{r_1\cdots r_j}^{(i-(j+1) \deg(\pi_0))}(r \pi_0^{\ell-j-2},u).
\end{align*}
Multiplying by $ \prod_{\pi \mid r_j} \left(1 - q^{(\ell-1)\deg(\pi)} \legl{-1}{\pi}^{j+1}u^{\ell \deg(\pi)}\right)
$ and applying the induction hypothesis \eqref{eq:general}, we obtain
\begin{align}\label{eq:crazyinduction}
&\psi^{(i)}_{r_1 r_2 \cdots r_{j}\pi_0 } (r\pi_0^j , u) \left(1 - q^{(\ell-1)\deg(\pi_0)} \legl{-1}{\pi_0}^{j+1}u^{\ell \deg(\pi_0)}\right)
\prod_{\pi \mid r_j} \left(1 - q^{(\ell-1)\deg(\pi)} \legl{-1}{\pi}^{j+1}u^{\ell \deg(\pi)}\right)\nonumber \\
 &=\psi^{(i)}_{r_1 r_2 \cdots r_{j}} (r\pi_0^j, u) \prod_{\pi \mid r_j} \left(1 - q^{(\ell-1)\deg(\pi)} \legl{-1}{\pi}^{j+1}u^{\ell \deg(\pi)}\right)\nonumber\\& - q^{j \deg(\pi_0)} G_{\ell, j+1}(r, \pi_0) u^{(j+1) \deg(\pi_0)}   \psi_{r_1\cdots r_j}^{(i-(j+1) \deg(\pi_0))}(r \pi_0^{\ell-j-2},u)\prod_{\pi \mid r_j} \left(1 - q^{(\ell-1)\deg(\pi)} \legl{-1}{\pi}^{j+1}u^{\ell \deg(\pi)}\right)\nonumber\\
 &= \sum_{d \mid r_{j}} \mu(d) q^{j \deg(d)} G_{\ell,j+1} \left( \frac{r\pi_0^j}{d^{j}} , d \right) u^{(j+1) \deg(d)}  \psi^{(i-(j+1)\deg(d))}_{r_1 r_2 \cdots r_{j-1}} \left( \frac{r\pi_0^j}{d^{2j-\ell+2}}, u \right)\nonumber\\
 & - q^{j \deg(\pi_0)} G_{\ell, j+1}(r, \pi_0) u^{(j+1) \deg(\pi_0)} \nonumber \\
&\times \sum_{d \mid r_{j}} \mu(d) q^{j \deg(d)} G_{\ell,j+1} \left( \frac{r\pi_0^{\ell-j-2}}{d^{j}} , d \right) u^{(j+1) \deg(d)}  \psi^{(i-(j+1)(\deg(\pi_0)+\deg(d))}_{r_1 r_2 \cdots r_{j-1}} \left( \frac{r\pi_0^{\ell-j-2}}{d^{2j-\ell+2}}, u \right)\nonumber\\
&= \sum_{d \mid r_{j}} \mu(d) q^{j \deg(d)} G_{\ell,j+1} \left( \frac{r\pi_0^j}{d^{j}} , d \right) u^{(j+1) \deg(d)}  \psi^{(i-(j+1)\deg(d))}_{r_1 r_2 \cdots r_{j-1}} \left( \frac{r\pi_0^j}{d^{2j-\ell+2}}, u \right)\nonumber\\
 & +G_{\ell, j+1}(r, \pi_0) \nonumber  \\
&\times \sum_{d \mid r_{j}} \mu(d\pi_0) q^{j \deg(d\pi_0)} G_{\ell,j+1} \left( \frac{r\pi_0^{\ell-j-2}}{d^{j}} , d \right) u^{(j+1) \deg(d\pi_0)}  \psi^{(i-(j+1)\deg(d\pi_0))}_{r_1 r_2 \cdots r_{j-1}} \left( \frac{r\pi_0^j}{(d\pi_0)^{2j-\ell+2}}, u \right).
\end{align}

Now we have 
\begin{equation}\label{eq:G1}
G_{\ell,j+1} \left( \frac{r\pi_0^{\ell-j-2}}{d^{j}} , d \right)=\overline{\legl{\pi_0^{\ell-j-2}}{d}^{j+1}} G_{\ell,j+1} \left( \frac{r}{d^{j}} , d \right)=\legl{\pi_0}{d}^{(j+2)(j+1)} G_{\ell,j+1} \left( \frac{r}{d^{j}} , d \right).
\end{equation}
and 
\begin{equation}\label{eq:G2}
G_{\ell, j+1}(r, \pi_0) =\overline{ \legl{d^j}{\pi_0}^{j+1}}
G_{\ell,j+1} \left( \frac{r}{d^{j}} , \pi_0 \right)=\overline{ \legl{\pi_0}{d}^{j(j+1)}}
G_{\ell,j+1} \left( \frac{r}{d^{j}} , \pi_0 \right).
\end{equation}

Combining \eqref{eq:G1} and \eqref{eq:G2}, we get 
\begin{align*}
 G_{\ell, j+1}(r, \pi_0) G_{\ell,j+1} \left( \frac{r\pi_0^{\ell-j-2}}{d^{j}} , d \right)=& \legl{\pi_0}{d}^{(j+2)(j+1)} \overline{ \legl{\pi_0}{d}^{j(j+1)}}
G_{\ell,j+1} \left( \frac{r}{d^{j}} , \pi_0 \right)G_{\ell,j+1} \left( \frac{r}{d^{j}} , d \right)\\
=&\legl{\pi_0}{d}^{2(j+1)} 
G_{\ell,j+1} \left( \frac{r}{d^{j}} , \pi_0 \right)G_{\ell,j+1} \left( \frac{r}{d^{j}} , d \right)\\
=&G_{\ell,j+1} \left( \frac{r}{d^{j}} , d\pi_0 \right)
\end{align*}
by Lemma \ref{gauss}(ii). Replacing this in \eqref{eq:crazyinduction}, we complete the induction step.

\end{proof}

We are now ready to prove the main result of the section. 
\begin{theorem} \label{thm:gettingridofa} Let $a, r \in \F_q[t]$, where $a$ is monic  square-free and $(r,a)=1$. We write $r = r_1 r_2^2 \cdots r_{\ell-1}^{\ell-1}r_\ell^\ell$, where the $r_j$ are square-free and pairwise coprime
for $j=1,\dots,\ell-1$, and let $r_\ell^*$ be the product of the primes dividing $r_\ell$, but not $r_1\cdots r_{\ell-1}$.
Then 
\begin{align*}
  \psi_r^{(i)}(r, a,u) =&
 G_\ell(r,a) u^{\deg(a)}\sum_{\substack{E\in \mathcal{M}\\ E \mid r_\ell^* }}\mu(E)G_\ell(a^{\ell-2}r_1\cdots r_{\ell-1}^{\ell-1},E)u^{\deg(E)}\\&\times
  \prod_{\pi \mid aEr_1\cdots r_{\ell-1}}
\left(1 - q^{(\ell-1)\deg(\pi)} \legl{-1}{\pi}^{\nu_\pi(aEr_1\cdots r_{\ell-2}r_{\ell-1})+1}u^{\ell \deg(\pi)}\right)^{-1}\nonumber \\
&\times  \sum_{\substack{d_{\ell-2} \mid aEr_{\ell-2}\\ d_1\mid r_1, \dots, d_{\ell-3}\mid r_{\ell-3}}} \mu(d_1d_2\cdots d_{\ell-2} ) q^{\deg(d_1d_2^2\cdots d_{\ell-2}^{\ell-2})}  u^{\deg(d_1^2d_2^3\cdots d_{\ell-2}^{\ell-1}  )} \prod_{1\leq k< j \leq \ell-2} \overline{\legl{d_k}{d_j}^{k(j+1)}}\\
& \times \prod_{j=1}^{\ell-2} G_{\ell,j+1} \left( \frac{a^{\ell-2}E^{\ell-2}r_1 r_2^2 \cdots r_{\ell-1}^{\ell-1}}{d_{1}d_2\cdots d_{\ell-2} ^{\ell-2}} , d_{j} \right)\\
&  \times \psi^{(i-\deg(a)-\deg(E)- \deg(d_1^2d_2^3\cdots d_{\ell-2}^{\ell-1} ))} \left( \frac{a^{\ell-2}E^{\ell-2}r_1 r_2^2 \cdots r_{\ell-1}^{\ell-1}}{\prod_{j=1}^{\ell-2}d_j^{2j+2-\ell}}, u \right).
\end{align*}
\end{theorem}
\begin{proof}
Writing $F=aF'$, we have $(a, F')=1$ since $G_\ell(r,F)=0$ when $F$ is not square-free and $(r,F)=1$. Then, using Lemma \ref{gauss}(i), we have
\begin{align*}
  \psi_r^{(i)}(r, a,u)
&= G_\ell(r,a) u^{\deg(a)} \sum_{\substack{F\in \mathcal{M}\\ \deg(F)\equiv i - \deg(a) \pmod{\ell}\\ (F,ar)=1 }}G_\ell(a^{\ell-2}r, F)u^{\deg(F)}.
\end{align*} Since $(F,r)=1$ in the above sum, we have $$G_\ell(a^{\ell-2}r,F)=\overline{\chi_F(r_{\ell}^\ell)}G_\ell(a^{\ell-2}r_1\cdots r_{\ell-1}^{\ell-1},F)=G_\ell(a^{\ell-2}r_1\cdots r_{\ell-1}^{\ell-1},F),$$
and using $(a,F)=1$, 
\begin{align*}
 &\sum_{\substack{F\in \mathcal{M}\\ \deg(F)\equiv i - \deg(a) \pmod{\ell}\\ (F,ar)=1 }}G_\ell(a^{\ell-2}r,F) u^{\deg(F)}=\sum_{\substack{F\in \mathcal{M}\\ \deg(F)\equiv i - \deg(a)\pmod{\ell}\\ (F,ar)=1 }}G_\ell(a^{\ell-2}r_1\cdots r_{\ell-1}^{\ell-1},F) u^{\deg(F)}\\
 &=\sum_{\substack{F\in \mathcal{M}\\ \deg(F)\equiv i -\deg(a) \pmod{\ell}\\ (F,ar_1\cdots r_{\ell-1})=1 }}G_\ell(a^{\ell-2}r_1\cdots r_{\ell-1}^{\ell-1},F)u^{\deg(F)} \sum_{\substack{E\in \mathcal{M}\\ E\mid (F,r_\ell^*) }}\mu(E)\nonumber\\
  &=\sum_{\substack{E\in \mathcal{M}\\ E \mid r_\ell^* }}\mu(E) u^{\deg(E)} \sum_{\substack{F\in \mathcal{M}\\ \deg(F)\equiv i-\deg(a) -\deg(E) \pmod{\ell}\\ (FE,ar_1\cdots r_{\ell-1})=1\\ (F, E)=1 }}G_\ell(a^{\ell-2}r_1\cdots r_{\ell-1}^{\ell-1},FE)u^{\deg(F)}\nonumber \\
    &=\sum_{\substack{E\in \mathcal{M}\\ E \mid r_\ell^* }}\mu(E)G_\ell(a^{\ell-2}r_1\cdots r_{\ell-1}^{\ell-1},E)u^{\deg(E)} \sum_{\substack{F\in \mathcal{M}\\ \deg(F)\equiv i-\deg(a)-\deg(E) \pmod{\ell}\\ (F,aEr_1\cdots r_{\ell-1})=1 }}G_\ell(E^{\ell-2}a^{\ell-2}r_1\cdots r_{\ell-1}^{\ell-1},F)u^{\deg(F)}, \nonumber
\end{align*}
where we used Lemma \ref{gauss}(i) again since $(F, E)=1$.
We removed the condition $(E, a r_1 \cdots r_{\ell-1})=1$  in the last line since it is guaranteed by $E \mid r_\ell^*$ and $(a,r)=1$.\\

Hence we have that 
\begin{align*}
  \psi_r^{(i)}(r, a,u)
   =& G_\ell(r,a) u^{\deg(a)}\sum_{\substack{E\in \mathcal{M}\\ E \mid r_\ell^* }}\mu(E)G_\ell(a^{\ell-2}r_1\cdots r_{\ell-1}^{\ell-1},E)u^{\deg(E)}\\&\times \psi^{(i-\deg(a)-\deg(E) )}_{aEr_1\cdots r_{\ell-1}}(E^{\ell-2}a^{\ell-2}r_1\cdots r_{\ell-1}^{\ell-1};u).
\end{align*}
Applying Lemma \ref{lem:argh}, we have 
\begin{align*}
  \psi_r^{(i)}(r, a,u) =& G_\ell(r,a) u^{\deg(a)}\sum_{\substack{E\in \mathcal{M}\\ E \mid r_\ell^* }}\mu(E)G_\ell(a^{\ell-2}r_1\cdots r_{\ell-1}^{\ell-1},E)u^{\deg(E)}\\&\times
  \prod_{\pi \mid r_{\ell-1}} \left(1 - q^{(\ell-1)\deg(\pi)}u^{\ell \deg(\pi)}\right)^{-1}\psi^{(i-\deg(a)-\deg(E) )}_{aEr_1\cdots r_{\ell-2}}(E^{\ell-2}a^{\ell-2}r_1\cdots r_{\ell-1}^{\ell-1};u)\\
   =& G_\ell(r,a) u^{\deg(a)}\sum_{\substack{E\in \mathcal{M}\\ E \mid r_\ell^* }}\mu(E)G_\ell(a^{\ell-2}r_1\cdots r_{\ell-1}^{\ell-1},E)u^{\deg(E)}\\&\times
  \prod_{\pi \mid r_{\ell-1}} \left(1 - q^{(\ell-1)\deg(\pi)}u^{\ell \deg(\pi)}\right)^{-1}\prod_{\pi \mid aEr_{\ell-2}}
\left(1 - q^{(\ell-1)\deg(\pi)} \legl{-1}{\pi}^{\ell-1}u^{\ell \deg(\pi)}\right)^{-1}\nonumber \\
&\times  \sum_{d_{\ell-2} \mid aEr_{\ell-2}} \mu(d_{\ell-2} ) q^{(\ell-2) \deg(d)} G_{\ell,\ell-1} \left( \frac{a^{\ell-2}E^{\ell-2}r_1 r_2^2 \cdots r_{\ell-1}^{\ell-1}}{d_{\ell-2} ^{\ell-2}} , d \right) u^{(\ell-1) \deg(d_{\ell-2})}\\&  \times \psi^{(i-\deg(a)-\deg(E)-(\ell-1)\deg(d_{\ell-2})}_{r_1 r_2 \cdots r_{\ell-3}} \left( \frac{a^{\ell-2}E^{\ell-2}r_1 r_2^2 \cdots r_{\ell-1}^{\ell-1}}{d_{\ell-2} ^{\ell-2}}, u \right)\\
  =& G_\ell(r,a) u^{\deg(a)}\sum_{\substack{E\in \mathcal{M}\\ E \mid r_\ell^* }}\mu(E)G_\ell(a^{\ell-2}r_1\cdots r_{\ell-1}^{\ell-1},E)u^{\deg(E)}\\&\times
  \prod_{\pi \mid r_{\ell-1}} \left(1 - q^{(\ell-1)\deg(\pi)}u^{\ell \deg(\pi)}\right)^{-1}\prod_{\pi \mid aEr_{\ell-2}}
\left(1 - q^{(\ell-1)\deg(\pi)} \legl{-1}{\pi}^{\ell-1}u^{\ell \deg(\pi)}\right)^{-1}\nonumber \\
&\times \prod_{\pi \mid r_{\ell-3}} 
\left(1 - q^{(\ell-1)\deg(\pi)} \legl{-1}{\pi}^{\ell-2}u^{\ell\deg(\pi)}\right)^{-1}\nonumber \\
&\times  \sum_{d_{\ell-2} \mid aEr_{\ell-2}} \mu(d_{\ell-2} ) q^{(\ell-2) \deg(d_{\ell-2})} G_{\ell,\ell-1} \left( \frac{a^{\ell-2}E^{\ell-2}r_1 r_2^2 \cdots r_{\ell-1}^{\ell-1}}{d_{\ell-2} ^{\ell-2}} , d_{\ell-2} \right) u^{(\ell-1) \deg(d_{\ell-2})}\\
& \times  \sum_{d_{\ell-3} \mid r_{\ell-3}} \mu(d_{\ell-3} ) q^{(\ell-3) \deg(d_{\ell-3})} G_{\ell,\ell-2} \left( \frac{a^{\ell-2}E^{\ell-2}r_1 r_2^2 \cdots r_{\ell-1}^{\ell-1}}{d_{\ell-3} ^{\ell-3}d_{\ell-2} ^{\ell-2}} , d_{\ell-3} \right) u^{(\ell-2) \deg(d_{\ell-3})}\\
&  \times \psi^{(i-\deg(a)-\deg(E)-(\ell-1)\deg(d_{\ell-2})-(\ell-2)\deg(d_{\ell-3})}_{r_1 r_2 \cdots r_{\ell-4}} \left( \frac{a^{\ell-2}E^{\ell-2}r_1 r_2^2 \cdots r_{\ell-1}^{\ell-1}}{d_{\ell-3} ^{\ell-4}d_{\ell-2} ^{\ell-2}}, u \right).
\end{align*}
Continuing in this way, we get 
\begin{align*}
  \psi_r^{(i)}(r, a,u) =&
 G_\ell(r,a) u^{\deg(a)}\sum_{\substack{E\in \mathcal{M}\\ E \mid r_\ell^* }}\mu(E)G_\ell(a^{\ell-2}r_1\cdots r_{\ell-1}^{\ell-1},E)u^{\deg(E)}\\&\times
  \prod_{\pi \mid aEr_1\cdots r_{\ell-1}}
\left(1 - q^{(\ell-1)\deg(\pi)} \legl{-1}{\pi}^{\nu_\pi(aEr_1\cdots r_{\ell-2}r_{\ell-1})+1}u^{\ell \deg(\pi)}\right)^{-1}\nonumber \\
&\times  \sum_{d_{\ell-2} \mid aEr_{\ell-2}} \mu(d_{\ell-2} ) q^{(\ell-2) \deg(d_{\ell-2})} G_{\ell,\ell-1} \left( \frac{a^{\ell-2}E^{\ell-2}r_1 r_2^2 \cdots r_{\ell-1}^{\ell-1}}{d_{\ell-2} ^{\ell-2}} , d_{\ell-2} \right) u^{(\ell-1) \deg(d_{\ell-2})}\\
& \times  \sum_{d_{\ell-3} \mid r_{\ell-3}} \mu(d_{\ell-3} ) q^{(\ell-3) \deg(d_{\ell-3})} G_{\ell,\ell-2} \left( \frac{a^{\ell-2}E^{\ell-2}r_1 r_2^2 \cdots r_{\ell-1}^{\ell-1}}{d_{\ell-3} ^{\ell-3}d_{\ell-2} ^{\ell-2}} , d_{\ell-3} \right) u^{(\ell-2) \deg(d_{\ell-3})}\\
& \times \dots \\
& \times  \sum_{d_{1} \mid r_{1}} \mu(d_{1} ) q^{\deg(d_1)} G_{\ell,2} \left( \frac{a^{\ell-2}E^{\ell-2}r_1 r_2^2 \cdots r_{\ell-1}^{\ell-1}}{d_{1}\cdots d_{\ell-2} ^{\ell-2}} , d_{1} \right) u^{2 \deg(d_1)}\\
&  \times \psi^{(i-\deg(a)-\deg(E)-(\ell-1)\deg(d_{\ell-2})-(\ell-2)\deg(d_{\ell-3})-\cdots -2d_1)} \left( \frac{a^{\ell-2}E^{\ell-2}r_1 r_2^2 \cdots r_{\ell-1}^{\ell-1}}{\prod_{j=1}^{\ell-2}d_j^{2j+2-\ell}}, u \right).
\end{align*}
Writing 
\[G_{\ell,j+1} \left( \frac{a^{\ell-2}E^{\ell-2}r_1 r_2^2 \cdots r_{\ell-1}^{\ell-1}}{d_{j}^j\cdots d_{\ell-2} ^{\ell-2}} , d_{j} \right)
=\overline{\legl{d_1\cdots d_{j-1}^{j-1}}{d_j}^{j+1}}G_{\ell,j+1} \left( \frac{a^{\ell-2}E^{\ell-2}r_1 r_2^2 \cdots r_{\ell-1}^{\ell-1}}{d_{1}d_2^2\cdots d_{\ell-2} ^{\ell-2}} , d_{j} \right),\]
we can further simplify
\begin{align*}
  \psi_r^{(i)}(r, a,u) =&
 G_\ell(r,a) u^{\deg(a)}\sum_{\substack{E\in \mathcal{M}\\ E \mid r_\ell^* }}\mu(E)G_\ell(a^{\ell-2}r_1\cdots r_{\ell-1}^{\ell-1},E)u^{\deg(E)}\\&\times
  \prod_{\pi \mid aEr_1\cdots r_{\ell-1}}
\left(1 - q^{(\ell-1)\deg(\pi)} \legl{-1}{\pi}^{\nu_\pi(aEr_1\cdots r_{\ell-2}r_{\ell-1})+1}u^{\ell \deg(\pi)}\right)^{-1}\nonumber \\
&\times  \sum_{\substack{d_{\ell-2} \mid aEr_{\ell-2}\\ d_1\mid r_1, \dots, d_{\ell-3}\mid r_{\ell-3}}} \mu(d_1d_2\cdots d_{\ell-2} ) q^{\deg(d_1d_2^2\cdots d_{\ell-2}^{\ell-2})}  u^{\deg(d_1^2d_2^3\cdots d_{\ell-2}^{\ell-1}  )} \prod_{1\leq \ell< j \leq n-2} \overline{\legl{d_\ell}{d_j}^{\ell(j+1)}}\\
& \times \prod_{j=1}^{\ell-2} G_{\ell,j+1} \left( \frac{a^{\ell-2}E^{\ell-2}r_1 r_2^2 \cdots r_{\ell-1}^{\ell-1}}{d_{1}d_2\cdots d_{\ell-2} ^{\ell-2}} , d_{j} \right)\\
&  \times \psi^{(i-\deg(a)-\deg(E)-\deg(d_1^2d_2^3\cdots d_{\ell-2}^{\ell-1} ))} \left( \frac{a^{\ell-2}E^{\ell-2}r_1 r_2^2 \cdots r_{\ell-1}^{\ell-1}}{\prod_{j=1}^{\ell-2}d_j^{2j+2-\ell}}, u \right).
\end{align*}
  \end{proof}

\begin{corollary}
\label{cor_convexity}
 Let $a, r \in \F_q[t]$, where $a$ is monic  square-free and $(r,a)=1$. We write $r = r_1 r_2^2 \cdots r_{\ell-1}^{\ell-1}r_\ell^\ell$, where the $r_j$ are square-free and pairwise coprime for $j=1,\dots,\ell-1$.
Let $u=q^{-s}$ and $\sigma=\re(s)$, where $1 \leq \sigma < \frac32$.
Then,
$$
 |  \psi_r^{(i)}(r, a,q^{-s}) |  \ll \begin{cases}  |a r|^{\varepsilon} |a|^{\ell \left( \frac34 - \frac{\sigma}{2} \right) - 1} |r|^{(\ell+1) \left( \frac34 - \frac{\sigma}{2} \right)  - \frac{3}{2}} &
 \mbox{for $ \ell  > \frac{6}{3 - 2 \sigma},$}\\
  |a r|^{\varepsilon} |a|^{\ell \left( \frac34 - \frac{\sigma}{2} \right) - 1} |r|^{ \left( \frac34 - \frac{\sigma}{2} \right)}  & \mbox{otherwise.} \end{cases}
$$
\end{corollary}

\begin{proof} Recall that $r_\ell^*$ denotes the product of the primes dividing $r_\ell$ but not $r_1\cdots r_{\ell-1}$. Since $a, r_1, \dots, r_{\ell-1}, r_\ell^*$ are square-free, we can use Corollary \ref{size-GS} to bound the Gauss sums in the formula of Theorem \ref{thm:gettingridofa}, and we get
\begin{align*}
|  \psi_r^{(i)}(r, a,q^{-s}) | &\ll |a|^{\frac12 - \sigma}
\sum_{\substack{E\in \mathcal{M}\\ E \mid r_\ell^* }}  |E|^{\frac12 - \sigma}
 \sum_{\substack{d_{\ell-2} \mid aEr_{\ell-2}\\ d_1\mid r_1, \dots, d_{\ell-3}\mid r_{\ell-3}}}
\prod_{j=1}^{\ell-2} |d_j|^{\frac12+j-(j+1) \sigma} \\
& \quad\quad \times \max_{\substack{0\leq i \leq \ell-1}} \left| \psi^{(i)} \left( \frac{a^{\ell-2} E^{\ell-2}r_1 r_2^2 \cdots r_{\ell-1}^{\ell-1}}{\prod_{j=1}^{\ell-2}d_j^{2j+2-\ell}}, q^{-s} \right) \right|.
\end{align*}
The number of terms in the $E$-sum is bounded by $|r_\ell|^\varepsilon$, and the number of terms in the $(d_1, \dots, d_{\ell-2})$-sum is bounded by $|a r_1 \cdots r_{\ell-2} r_\ell|^\varepsilon.$ Bounding each sum by the number of terms times the maximal value of the terms, we get
\begin{align*}
|  \psi_r^{(i)}(r, a,q^{-s}) | &\ll  |ar|^\varepsilon |a|^{\frac12 - \sigma}
 \max_{\substack{d_{\ell-2} \mid a r_\ell^* r_{\ell-2}\\ d_1\mid r_1, \dots, d_{\ell-3}\mid r_{\ell-3}}}
\prod_{j=1}^{\ell-2} |d_j|^{\frac12+j-(j+1) \sigma} \\
 & \quad \quad \times \max_{\substack{E \mid r_\ell^* \\0\leq i \leq \ell-1}} \left| \psi^{(i)} \left( \frac{a^{\ell-2}E^{\ell-2}r_1 r_2^2 \cdots r_{\ell-1}^{\ell-1}}{\prod_{j=1}^{\ell-2}d_j^{2j+2-\ell}}, q^{-s} \right) \right|.
 \end{align*}
Applying Theorem \ref{residue},
 \begin{align*}
 |  \psi_r^{(i)}(r, a,q^{-s}) | 
 &\ll  |ar|^\varepsilon |a|^{\frac12 - \sigma}
 \max_{\substack{d_{\ell-2} \mid a r_\ell^* r_{\ell-2}\\ d_1\mid r_1, \dots, d_{\ell-3}\mid r_{\ell-3}\\ E \mid r_\ell^*}}
 \prod_{j=1}^{\ell-2} |d_j|^{\frac12+j-(j+1) \sigma}  \left| \frac{a^{\ell-2}E^{\ell-2}r_1 r_2^2 \cdots r_{\ell-1}^{\ell-1}}{\prod_{j=1}^{\ell-2}d_j^{2j+2-\ell}} \right|^{\frac12 (\frac32 - \sigma)}\\
 &\ll  |ar|^\varepsilon |a|^{\ell \left( \frac34 - \frac{\sigma}{2} \right) -1 } \left|r_1 r_2^2 \cdots r_{\ell-1}^{\ell-1} r_\ell^{\ell-2}\right|^{\frac12 (\frac32 - \sigma)} \\
 & \;\;\;\; \times  \max_{\substack{d_{\ell-2} \mid a r_\ell^* r_{\ell-2}\\ d_1\mid r_1, \dots, d_{\ell-3}\mid r_{\ell-3}}}
|d_1 \cdots d_{\ell-2}|^{\ell \left( \frac34 - \frac{\sigma}{2} \right) }  \; {\prod_{j=1}^{\ell-2} |d_j|^{-\frac{j}{2}-1}} .
\end{align*}
We have
$$
\max_{\substack{d_{\ell-2} \mid a r_\ell^* r_{\ell-2}\\ d_1\mid r_1, \dots, d_{\ell-3}\mid r_{\ell-3}}}
|d_1 \cdots d_{\ell-2}|^{\ell \left( \frac34 - \frac{\sigma}{2} \right) }  \; {\prod_{j=1}^{\ell-2} |d_j|^{-\frac{j}{2}-1}}  \leq \begin{cases} |r|^{\ell \left( \frac34 - \frac{\sigma}{2} \right) - \frac32} & \ell \left( \tfrac34 - \tfrac{\sigma}{2} \right) - \tfrac32 > 0, \\ 1 & \ell \left( \tfrac34 - \tfrac{\sigma}{2} \right) - \tfrac32 \leq 0. \end{cases}
$$
Replacing in the last equation, this completes the proof.
\end{proof}

\section{Lindel\"{o}f on average-type bounds}
\label{section_lov}

We will prove the following Lindel\"{o}f on average bound for $\psi^{(i)}$.

\begin{proposition}
\label{lov}
Let $|z|=q^{-\sigma}$ with $1 < \sigma \leq 4/3$ and $i \in \{0,1,\ldots, \ell-1\}$. Further assume that  $|z^{\ell}-q^{-\ell-1}| > \delta$ for some $\delta>0$. For $h \in \mathbb{F}_q[t]$ and $j \geq 1$, we have
\begin{align*}
  \sum_{V \in \mathcal{M}_n} \Big| \psi^{(i)} (hV^j, z ) \Big|^2 & \ll |h|^{\varepsilon} q^{\varepsilon n} \Bigg( q^n + \Big(q^{n(\frac{1}{3}+\frac{3j}{2}-j\sigma)} |h|^{3/2-\sigma} +  q^{n(\frac{2}{3}+\frac{4j}{3}-j\sigma)} |h|^{4/3-\sigma}\Big) \\
  & \times \max \Big\{1, q^{jn(\sigma-7/6)} |h|^{\sigma-7/6} \Big\}\Bigg) .
  \end{align*}
\end{proposition}

\begin{proof} The first steps are as in the proof from \cite[Lemma 6.5]{ddds}.
In Theorem \ref{afe}, let $T_1$ denote the first term, $T_2$ the second, and $T_3$ the third. Note that it is enough to show the bound for
$$\sum_{V \in \mathcal{M}_n} |T_k(hV^j,z)|^2,$$
for $k=1,2,3$. We will bound the term corresponding to $k=1$ and only sketch the proof for $k=2,3$, since they are similar. Using the approximate functional equation from Theorem \ref{afe}, we have to bound
\begin{align}
\label{tobound}
 \sum_{V \in \mathcal{M}_n} \Bigg| \sum_{\deg(F) \leq [\frac{\deg(hV^j)}{2}]} G_{\ell}(hV^j,F)b_1(\deg(F),z) z^{\deg(F)} \Bigg|^2,
 \end{align} where we interpret the coefficient $b_1(\deg(F),z)$ as being $0$ if $\deg(F) \not \equiv i \pmod \ell$. We write $F=W F'$ where $W|(hV)^{\infty}$ and $(F', hV)=1$. We have
$$ G_{\ell} (hV^j, WF') = G_{\ell} (hV^j,W) G_{\ell} (hV^j,F') \Big( \frac{F'}{W} \Big)_{\ell}^2.$$
Note that in order for $G_{\ell}(hV^j,W)$ to be non-zero, we need $W|(hV^j)^{2}$. Using the Cauchy--Schwarz inequality and equation \eqref{eq:GS}, we have
\begin{align}
\Bigg|& \sum_{\deg(F)  \leq [\frac{\deg(hV^j)}{2}]} G_{\ell}(hV^j,F)b_1(\deg(F),z) z^{\deg(F)} \Bigg|^2 \leq \sum_{W|(hV^j)^2} \frac{|G_{\ell}(hV^j,W)|^2}{|W|^{2\sigma}} \label{sumw} \\
& \times  \sum_{W|(hV^j)^2} \Bigg| \sum_{\substack{\deg(F') \leq [\deg(hV^j)/2]-\deg(W) \\ (F', hV)=1}} G_{\ell}(1,F') \Big(  \frac{F'}{W}\Big)_{\ell}^2 \overline{\Big( \frac{F'}{hV^j}\Big)_{\ell} } z^{\deg(F')} b_1(\deg(WF'),z) \Bigg|^2. \nonumber
\end{align}
We note that the first sum over $W$ is bounded by $|hV|^{\varepsilon}$, which follows from the fact that $\sigma > 1$ and $|G_{\ell}(hV^j, W)| \leq |W|$. We then get that
\begin{align*}
\eqref{tobound} \ll |h|^{\varepsilon} q^{\varepsilon n} \sum_{V \in \mathcal{M}_n}  \sum_{W|(hV^j)^2} \Bigg| \sum_{\substack{\deg(F') \leq [\deg(hV^j)/2]-\deg(W) \\ (F', hV)=1}} G_{\ell}(1,F') \Big(  \frac{F'}{W}\Big)_{\ell}^2 \overline{\Big( \frac{F'}{hV^j}\Big)_{\ell} } z^{\deg(F')} b_1(\deg(WF'),z) \Bigg|^2.
\end{align*}
 We write $W=ab^2$ with $a$ square-free. Then it follows that $ab|hV^j$, and $\deg(ab^2) \leq [(jn+\deg(h))/2]$ since $W | F$.
 Let $Q=ab/(h,ab)$. It follows that $Q|V^{j}$ and we write $V=\text{rad}(Q)v$, where $\text{rad}(Q)$ denotes the product of the primes dividing $Q$. Then from the previous equation, we get that
 \begin{align*}
 \eqref{tobound} & \ll |h|^{\varepsilon} q^{\varepsilon n} \sum_{\substack{a \in \mathcal{H} \\ \deg(ab^2) \leq [(jn+\deg(h))/2]}} \sum_{v \in \mathcal{M}_n - \deg(\text{rad}(Q))} \\
 & \times  \Bigg| \sum_{\substack{F'\in \mathcal{H} \\ \deg(F') \leq [(\deg(h)+jn)/2]-\deg(ab^2) \\ (F',hQv)=1}} G_{\ell}(1,F') \Big(  \frac{F'}{ab^2}\Big)_{\ell}^2 \overline{\Big( \frac{F'}{h \text{rad}(Q)v}\Big)_{\ell} } z^{\deg(F')} b_1(\deg(ab^2F'),z) \Bigg|^2.
 \end{align*}
 In the sum over $F'$, note that we can restrict to $F'$ square-free, otherwise, $G_{\ell}(1,F') =0$. Let 
 $$a_{T_1}(F') = G_{\ell}(1,F') \Big(  \frac{F'}{ab^2}\Big)_{\ell}^2 \overline{\Big( \frac{F'}{h\text{rad}(Q)}\Big)_{\ell} } z^{\deg(F')} b_1(\deg(ab^2F'),z).$$
 We get that there exists some $k \leq [ (\deg(h)+jn)/2]-\deg(ab^2)$ such that
 \begin{align*}
 \sum_{v \in \mathcal{M}_n - \deg(\text{rad}(Q))}& \Bigg|  \sum_{\substack{F' \in \mathcal{H} \\ \deg(F') \leq [(\deg(h)+jn)/2]-\deg(ab^2) \\ (F',hQv)=1}} a_{T_1}(F') \overline{\Big( \frac{F'}{v}\Big)_{\ell}} \Bigg|^2\\ &\ll (\deg(h)+jn)^2 \sum_{v \in \mathcal{M}_n - \deg(\text{rad}(Q))} \Bigg| \sum_{\substack{F' \in \mathcal{H}_k \\ (F',hQv)=1}} a_{T_1}(F') \overline{\Big( \frac{F'}{v}\Big)_{\ell}} \Bigg|^2.
 \end{align*}
 Using the Large Sieve Inequality as in Theorem \ref{largesieve2} and the facts that $|a_{T_1}(F')| \ll |F'|^{1/2-\sigma}$ and $k \leq \deg(h)/2+jn/2-\deg(ab^2)$, we get that
 \begin{align}
& \sum_{v \in \mathcal{M}_{n - \deg(\text{rad}(Q))}}  \Bigg| \sum_{\substack{F' \in \mathcal{H} \\  \deg(F') \leq [(\deg(h)+jn)/2]-\deg(ab^2) \nonumber  \\ (F',hQv)=1}} a_{T_1}(F') \overline{\Big( \frac{F'}{v}\Big)_{\ell}} \Bigg|^2\\
 &  \ll |h|^{\varepsilon} q^{\varepsilon n} \Bigg(q^{n-\deg(\text{rad}(Q))}+q^{n(1/3+j/2)} \frac{|h|^{1/2}}{|\text{rad}(Q)|^{1/3} |ab^2|}+q^{n(2/3+j/3)} \frac{|h|^{1/3}}{|\text{rad}(Q)|^{2/3}|ab^2|^{2/3}} \Bigg)   \nonumber \\
 & \times q^{jn(1-\sigma)} |h|^{1-\sigma} |ab^2|^{2\sigma-2} \nonumber \\
 & \ll |h|^{\varepsilon} q^{\varepsilon n} \Bigg(q^{n(1+j-j\sigma)-\deg(\text{rad}(Q))} |h|^{1-\sigma} |ab^2|^{2\sigma-2} + q^{n (1/3+3j/2-j \sigma)} \frac{|h|^{3/2-\sigma}}{|\text{rad}(Q)|^{1/3} |ab^2|^{3-2\sigma}} \nonumber\\
 & +q^{n(2/3+4j/3-j \sigma)} \frac{|h|^{4/3-\sigma}}{|\text{rad}(Q)|^{2/3}|ab^2|^{8/3-2\sigma}}\Bigg). \label{sumv}
\end{align} 
 Now we introduce the sums over $a,b$ and keep in mind that $Q=ab/(ab,h)$. We have that
 $$\text{rad}(Q) = \frac{\text{rad}(b) a }{(a,b) \text{rad}(ab,h)}.$$
  To bound the contribution of first term in \eqref{sumv}, we write
  \begin{align*} \sum_{\deg(ab^2) \leq [(jn+\deg(h))/2]}& \frac{1}{|\text{rad}(Q)| |ab^2|^{2-2\sigma}} =  \sum_{\deg(a) \leq  [(jn+\deg(h))/2]} \frac{1}{|a|^{3-2\sigma}}\\
 & \times \sum_{\deg(b) \leq (jn+\deg(h))/4-\deg(a)/2} \frac{|(a,b)| |\text{rad}(ab,h)|}{|\text{rad}(b)| |b|^{4-4\sigma}}.
 \end{align*}
 The generating series of the sum over $b$ is
 \begin{align*}
& \sum_{b \in \mathcal{M}}  \frac{|(a,b)| |\text{rad}(ab,h)|}{|\text{rad}(b)| |b|^{4-4\sigma}} u^{\deg(b)} = \prod_{P \nmid ah} \Big(  1+ \frac{u^{\deg(P)}}{|P|^{5-4\sigma} ( 1- \frac{u^{\deg(P)}}{|P|^{4-4\sigma}})}\Big) \\
 & \times \prod_{\substack{P|h \\ P \nmid a}} \Big( 1-\frac{u^{\deg(P)}}{|P|^{4-4\sigma}}\Big)^{-1} \prod_{P|(a,h)} \Big(|P|\sum_{j=0}^{\infty} \frac{ u^{j\deg(P)} }{|P|^{j(4-4\sigma)}} \Big)  \prod_{\substack{P | a \\ P \nmid h}}  \Big( 1-\frac{u^{\deg(P)}}{|P|^{4-4\sigma}}\Big)^{-1},
 \end{align*}
which converges for $|u|<q^{4-4\sigma}$. By using Perron's formula for the sum over $b$ and choosing the contour of integration $|u|=q^{4-4\sigma-\varepsilon}$, we get that
 $$  \sum_{\deg(b) \leq (jn+\deg(h))/4-\deg(a)/2} \frac{|(a,b)| |\text{rad}(ab,h)|}{|\text{rad}(b)| |b|^{4-4\sigma}} \ll q^{(4\sigma-4)( \frac{jn+\deg(h)}{4}-\frac{\deg(a)}{2})} |(a,h)| |ah|^{\varepsilon}.$$
 Introducing the sum over $a$, we have that 
 \begin{align*}
  \sum_{\deg(a) \leq  [(jn+\deg(h))/2]} \frac{|(a,h)|}{|a|}
  \ll \sum_{d|h} |d| \sum_{\substack{\deg(a)  [(jn+\deg(h))/2] \\ (a,h)=d}} \frac{1}{|a|}
  \ll |h|^{\varepsilon} q^{jn \varepsilon}.
 \end{align*}
 Combining the equations above, it follows that the contribution of the first term of  \eqref{sumv} is bounded by
 \begin{equation} \sum_{\substack{a \in \mathcal{H} \\ \deg(ab^2) \leq [(jn+\deg(h))/2]}}q^{n(1+j-j\sigma)-\deg(\text{rad}(Q))} |h|^{1-\sigma} |ab^2|^{2\sigma-2} \ll q^{n+\varepsilon n} |h|^{\varepsilon}  .
   \label{first_term}
 \end{equation}
To bound the contribution of second and third terms in \eqref{sumv}, we need to estimate sums of the form
\begin{align}
\sum_{\deg(ab^2)\leq (jn+\deg(h))/2} & \frac{|(a,b)|^{\Delta}|\text{rad}(ab,h)|^{\Delta}}{| \text{rad}(b) a|^{\Delta}|ab^2|^\delta}= \sum_{\deg(a) \leq (jn+\deg(h))/2} \frac{1}{|a|^{\Delta+\delta}} \nonumber \\
& \times \sum_{\deg(b) \leq (jn+\deg(h))/4-\deg(a)/2} \frac{|(a,b)|^{\Delta}|\text{rad}(ab,h)|^{\Delta}}{|\text{rad}(b)|^{\Delta} |b|^{2\delta}},
\label{toestimate}
\end{align}
where $\delta, \Delta\geq 0$ and $\Delta<1$. The generating series of the sum over $b$ becomes
\begin{align*}
\sum_{b \in \mathcal{M}} & \frac{|(a,b)|^{\Delta}|\text{rad}(ab,h)|^{\Delta}}{|\text{rad}(b)|^{\Delta} |b|^{2\delta}} u^{\deg(b)} = \prod_{P \nmid ah} \Big(  1+ \frac{u^{\deg(P)}}{|P|^{\Delta+2\delta} (1- \frac{u^{\deg(P)}}{|P|^{2\delta}})}\Big) \prod_{\substack{P | h \\ P \nmid a}} \Big(1 - \frac{u^{\deg(P)}}{|P|^{2\delta}} \Big)^{-1} \\
& \times \prod_{P|(a,h)} |P|^{\Delta} \Big(1 - \frac{u^{\deg(P)}}{|P|^{2\delta}} \Big)^{-1} \prod_{\substack{P |a \\ P\nmid h}} \Big(1 - \frac{u^{\deg(P)}}{|P|^{2\delta}} \Big)^{-1} \\
&= |(a,h)|^{\Delta} \prod_{P \nmid ah} \Big(  1+ \frac{u^{\deg(P)}}{|P|^{\Delta+2\delta} (1- \frac{u^{\deg(P)}}{|P|^{2\delta}})}\Big) \prod_{P|ah} \Big(1 - \frac{u^{\deg(P)}}{|P|^{2\delta}} \Big)^{-1}  \\
& := |(a,h)|^{\Delta} \mathcal{Z}\Big(  \frac{u}{q^{\Delta+2\delta}}\Big) \mathcal{F}(u), 
\end{align*}
where $\mathcal{F}(u)$ converges absolutely for $|u|<q^{2\delta}$.
If we denote by $M = (jn+\deg(h))/4-\deg(a)/2$, then Perron's formula gives
\begin{align*}
\sum_{\deg(b) \leq (jn+\deg(h))/4-\deg(a)/2} \frac{|(a,b)|^{\Delta}|\text{rad}(ab,h)|^{\Delta}}{|\text{rad}(b)|^{\Delta} |b|^{2\delta}} = \frac{1}{2 \pi i} \oint \frac{ |(a,h)|^{\Delta} \mathcal{F}(u)}{(1-u)  (1-q^{1-\Delta-2\delta}u)u^{M+1}} \, du,
\end{align*}
where we are integrating along a small circle around the origin. If $\Delta+2\delta-1<0$, then we get that
\begin{equation}
\label{sumb}
\sum_{\deg(b) \leq (jn+\deg(h))/4-\deg(a)/2} \frac{|(a,b)|^{\Delta}|\text{rad}(ab,h)|^{\Delta}}{|\text{rad}(b)|^{\Delta} |b|^{2\delta}}  \ll q^{(1-\Delta-2\delta)M} |(a,h)|^{\Delta} |ah|^{\varepsilon},
\end{equation} while if $\Delta+2\delta-1>0$, then
$$\sum_{\deg(b) \leq (jn+\deg(h))/4-\deg(a)/2} \frac{|(a,b)|^{\Delta}|\text{rad}(ab,h)|^{\Delta}}{|\text{rad}(b)|^{\Delta} |b|^{2\delta}}  \ll q^{\varepsilon n} |(a,h)|^{\Delta} |ah|^{\varepsilon}. $$
Introducing the sum over $a$, if $\Delta+2\delta-1>0$, we have
\begin{align*}
%\label{suma}
\sum_{\deg(a) \leq (jn+\deg(h))/2} \frac{|(a,h)^{\Delta}|}{|a|^{\Delta+\delta}} \ll 
\begin{cases}
|h|^{\varepsilon} q^{\varepsilon n} & \mbox{ if } \Delta+\delta \geq 1, \\
|h|^{\varepsilon} q^{(jn+\deg(h))(1-\Delta-\delta)/2} & \mbox{ if } \Delta+\delta <1,
\end{cases}
\end{align*}
and hence in this case,
$$ \eqref{toestimate} \ll |h|^\varepsilon q^{\varepsilon n}   \max \Big\{1, q^{(1-\Delta-\delta)jn/2} |h|^{(1-\Delta-\delta)/2} \Big\}.$$
If $\Delta+2\delta-1<0$, then introducing the sum over $a$, we have
\begin{align*}
\sum_{\deg(a) \leq (jn+\deg(h))/2} \frac{|(a,h)^{\Delta}|}{|a|^{1/2+\Delta/2}} \ll q^{(jn+\deg(h))(1-\Delta)/4} |h|^{\varepsilon}.
\end{align*}
Using this and with \eqref{sumb}, we get that in the case $\Delta+2\delta-1<0$, we have
$$ \eqref{toestimate} \ll |h|^{\varepsilon}  q^{(jn+\deg(h))(1-\Delta-\delta)/2} .$$
Combining this with the case $\Delta+2\delta-1>0$, it follows that
\begin{align}
\eqref{toestimate} \ll |h|^\varepsilon q^{\varepsilon n}   \max \Big\{1, q^{(1-\Delta-\delta)jn/2} |h|^{(1-\Delta-\delta)/2} \Big\}.\label{Deltadelta}
\end{align}
Using this to bound the second and third terms in \eqref{sumv} (with $\Delta=\frac{1}{3}$ and $\delta={3-2\sigma}$ for the second term and $\Delta=\frac{2}{3}$ and $\delta=\frac{8}{3}-2\sigma$ for the third term), together with \eqref{first_term}, we get that
 \begin{align*}
 \eqref{tobound} \ll |h|^{\varepsilon} q^{\varepsilon n} \Big( q^n +\Big(q^{n(1/3+3j/2-j\sigma)}|h|^{3/2-\sigma}+q^{n(2/3+4j/3-j\sigma)}|h|^{4/3-\sigma}\Big)  \max \Big\{1, q^{(jn+\deg(h))(\sigma-7/6)} \Big\} \Big).
 \end{align*}

 To bound the contributions from $T_2$ and $T_3$, the argument we used to bound the contribution from $T_1$ goes through similarly. For $T_2$, we replace $z \mapsto 1/(q^2z)$, and we multiply the final bound by $|qz|^{\deg(hV^j)}$. The only slight variation is in bounding the sum over $W$ in \eqref{sumw}. Namely, using Lemma \ref{gauss} (in particular, the fact that $|G_{\ell}(h,W)|$ is multiplicative as a function of $W$), we have
 \begin{align*}
 \sum_{W|(hV^j)^2} \frac{|G_{\ell}(hV^j,W)|^2}{|W|^{2(2-\sigma)}}  = \prod_{P|hV} \Big( 1+  O \Big(\sum_{i \geq 2} \frac{ |P|^{2i}}{|P|^{2i(2-\sigma)}} \Big)\Big) \ll |hV|^{\varepsilon}.
 \end{align*}
 Proceeding as in the case of $T_1$, we get that
 \begin{align*}
\sum_{V \in \mathcal{M}_n} |T_2(hV^j,z)|^2 \ll& |h|^{\varepsilon} q^{\varepsilon n} \Big( q^n +q^{n(1/3+3j/2-j\sigma)}|h|^{3/2-\sigma}+q^{n(2/3+4j/3-j\sigma)}|h|^{4/3-\sigma}\Big),
 \end{align*}
 where again we used \eqref{first_term} and \eqref{Deltadelta}.
 For $T_3$, we again proceed similarly as before, but now we replace $a_{T_1}(F')$ by $a_{T_3}(F')$, where
 $$a_{T_3}(F') = G_{\ell}(1,F') \Big(  \frac{F'}{ab^2}\Big)_{\ell}^2 \overline{\Big( \frac{F'}{h \text{rad}(Q)}\Big)_{\ell} } (q^2z)^{-\deg(F')} b_3(\deg(ab^2F'),z).$$
 Note that \[|a_{T_3}(F')| \ll |F'|^{\sigma -3/2}\] and $k  \leq \deg(h)/2+jn/2-\deg(ab^2)$. We then get that
 \begin{align*}
\sum_{V \in \mathcal{M}_n} |T_3(hV,z)|^2 &\ll |h|^{\varepsilon} q^{\varepsilon n} \Big( q^n +q^{n(1/3+3j/2-j\sigma)}|h|^{3/2-\sigma}+q^{n(2/3+4j/3-j\sigma)}|h|^{4/3-\sigma}\Big).
 \end{align*}
 Combining the bounds for the contributions of $T_1, T_2,$ and $T_3$ finishes the proof.
 \end{proof}
The following Corollary follows.
\begin{corollary}
\label{cor_lov}
For $i \in \{0,\ldots,\ell-1\}$, $j \geq 2$ and $z$ as in Proposition \ref{lov}, we have
\begin{align*}
 \sum_{V \in \mathcal{M}_n} \Big|  \psi^{(i)}(hV^j,z)\Big| \ll |h|^{\varepsilon} q^{\varepsilon n} \Big(q^n + \max \{ q^{n(\frac{2}{3} + \frac{3j}{4} - \frac{j \sigma}{2})} |h|^{\frac{1}{2} (\frac{3}{2}-\sigma)}, q^{n (\frac{2}{3}+\frac{j}{6})} |h|^{\frac{1}{6}}\} \Big).
 \end{align*}
\end{corollary}  
 \begin{proof}
 The proof easily follows from Cauchy--Schwarz and Proposition \ref{lov}.
 \end{proof}
 We will also prove the following Proposition.
 
 \begin{proposition}
Let $\ell \geq 5$, $i \in \{0,\ldots, \ell-1\}$, and $z$ as in Proposition  \ref{lov}. We have
\begin{align*}
\sum_{V \in \mathcal{M}_n} & |\psi^{(i)}_V(V,b,z)|   \ll |b|^{1/2-\sigma+\varepsilon} q^{\varepsilon n} \Big(q^n +q^{ n \max\{ \frac{\ell(9-6\sigma)-6\sigma-1}{12} , \frac{7-2\sigma}{4} \}} |b|^{(\ell-2)\frac{1}{2} (\frac{3}{2}-\sigma)}\\
& +q^{n \max\{ \frac{\ell}{6}-2\sigma+\frac{5}{3}, \frac{7}{6} \}} |b|^{\frac{\ell-2}{6}} \sum_{n_1+\ldots+\ell n_{\ell}=n}  \prod_{i \geq \max\{2, \frac{ \frac{\ell+1}{6} -\sigma}{\sigma-\frac{2}{3}} \}} q^{n_i (1-\frac{i}{2} ( \frac{\ell}{6} +\frac{15}{6}-3\sigma))}  \Big).
\end{align*}

 \end{proposition}

 \begin{proof}
 Using Theorem \ref{thm:gettingridofa}, we have 
 \begin{align*}
  \psi_V^{(i)}(V, b,z) =&
 G_\ell(r,b) z^{\deg(b)}\sum_{\substack{E\in \mathcal{M}\\ E \mid V_\ell^* }}\mu(E)G_\ell(b^{\ell-2}V_1\cdots V_{\ell-1}^{\ell-1},E)z^{\deg(E)}\\&\times
  \prod_{\pi \mid bEV_1\cdots V_{\ell-1}}
\left(1 - q^{(\ell-1)\deg(\pi)} \legl{-1}{\pi}^{\nu_\pi(aEr_1\cdots V_{\ell-2}V_{\ell-1})+1}z^{\ell \deg(\pi)}\right)^{-1}\nonumber \\
&\times  \sum_{\substack{d_{\ell-2} \mid bEV_{\ell-2}\\ d_1\mid V_1, \dots d_{\ell-3}\mid V_{\ell-3}}} \mu(d_1d_2\cdots d_{\ell-2} ) q^{\deg(d_1d_2^2\cdots d_{\ell-2}^{\ell-2})}  z^{\deg(d_1^2d_2^3\cdots d_{\ell-2}^{\ell-1}  )} \prod_{1\leq k< j \leq \ell-2} \overline{\legl{d_k}{d_j}^{k(j+1)}}\\
& \times \prod_{j=1}^{\ell-2} G_{\ell,j+1} \left( \frac{b^{\ell-2}E^{\ell-2}V_1 V_2^2 \cdots V_{\ell-1}^{\ell-1}}{d_{1}d_2\cdots d_{\ell-2} ^{\ell-2}} , d_{j} \right)\\
&  \times \psi^{(i-\deg(b)-\deg(E)- \deg(d_1^2d_2^3\cdots d_{\ell-2}^{\ell-1} ))} \left( \frac{b^{\ell-2}E^{\ell-2}V_1 V_2^2 \cdots V_{\ell-1}^{\ell-1}}{\prod_{j=1}^{\ell-2}d_j^{2j+2-\ell}}, z \right),
\end{align*}
where we recall that $V=V_1 \ldots V_{\ell-1}^{\ell-1} V_{\ell}^{\ell}$ with $V_1, \ldots, V_{\ell-1}$ square-free and pairwise coprime, and where $V_{\ell}^{*}$ denotes the product of the primes dividing $V_{\ell}$, but not $V_1 \ldots V_{\ell-1}$. It then follows that for any $\varepsilon>0$
\begin{align*}
\sum_{V \in \mathcal{M}_n} & |\psi^{(i)}_V(V,b,z)| \ll |b|^{1/2-\sigma+\varepsilon} \sum_{n_1+2n_2+\cdots+\ell n_\ell=n} \sum_{\substack{i=1 \\ \deg(V_i)=n_i} }^{\ell} \sum_{E|V_{\ell}^*} |E|^{1/2-\sigma+\varepsilon}  \nonumber  \\
& \times \sum_{\substack{d_{\ell-2} \mid bEV_{\ell-2}\\ d_i|V_i, i \leq \ell-3}} |d_i|^{i+1/2-(i+1)\sigma} |d_{\ell-2}|^{\ell-3/2-(\ell-1)\sigma} \sum_{j=0}^{\ell-1} \Big| \psi^{(j)}\left( \frac{b^{\ell-2}E^{\ell-2}V_1 V_2^2 \cdots V_{\ell-1}^{\ell-1}}{\prod_{j=1}^{\ell-2}d_j^{2j+2-\ell}}, z \right) \Big|. %\label{sumv0}
\end{align*}
For simplicity, denote
$$ \alpha = \frac{b^{\ell-2}E^{\ell-2} V_2^2 \cdots V_{\ell-1}^{\ell-1}}{\prod_{j=2}^{\ell-2}d_j^{2j+2-\ell}}.$$
Focusing on the sum over $V_1$ in the expression above and letting $V_1=d_1a$, we have
\begin{align}
\label{sumalov}
\sum_{\deg(V_1)=n_1} \sum_{d_1|V_1}|d_1|^{\frac{3}{2}-2\sigma} |\psi^{(j)}(\alpha V_1 d_1^{\ell-4},z) | = \sum_{\deg(a) \leq n_1} q^{(n_1-\deg(a))(\frac{3}{2}-2\sigma)} \sum_{\deg(d_1)=n_1-\deg(a)} |\psi^{(j)}(\alpha a d_1^{\ell-3} ,z)|.
\end{align}
Now we use Corollary \ref{cor_lov} to get that 
\begin{align*}
\sum_{\deg(d_1)=n_1-\deg(a)}&  |\psi^{(j)}(\alpha a d_1^{\ell-3} ,z)| \ll |\alpha a|^{\varepsilon} q^{\varepsilon n_1} \Big( q^{n_1-\deg(a)}  \\
&+ \max \{ q^{(n_1-\deg(a))(\frac{3\ell}{4} - \frac{19}{12} - \frac{(\ell-3)\sigma}{2})} |\alpha a|^{\frac{1}{2}( \frac{3}{2} -\sigma)} , q^{(n_1-\deg(a))(\frac{\ell}{6}+\frac{1}{6})} |\alpha a|^{\frac{1}{6}}\} \Big) .
\end{align*}
Introducing the sum over $a$, we have that
\begin{align*}
\eqref{sumalov} &\ll q^{n_1(\frac{3}{2}-2\sigma)+\varepsilon n} |\alpha|^{\varepsilon} \sum_{\deg(a) \leq n_1} \frac{1}{|a|^{\frac{3}{2}-2\sigma}} \Big( q^{n_1-\deg(a)}  \\
&+ \max \{ q^{(n_1-\deg(a))(\frac{3\ell}{4} - \frac{19}{12} - \frac{(\ell-3)\sigma}{2})} |\alpha a|^{\frac{1}{2}( \frac{3}{2} -\sigma)} , q^{(n_1-\deg(a))(\frac{\ell}{6}+\frac{1}{6})} |\alpha a|^{\frac{1}{6}}\} \Big) \\
& \ll |\alpha|^{\varepsilon} \Big( q^{n_1} + \max\{ q^{n_1 \max\{ \frac{\ell(9-6\sigma)-6\sigma-1}{12} , \frac{7-2\sigma}{4} \}} |\alpha|^{\frac{1}{2}(\frac{3}{2}-\sigma)}, q^{n_1 \max\{ \frac{\ell}{6}-2\sigma+\frac{5}{3}, \frac{7}{6} \}} |\alpha|^{\frac{1}{6}}\} \Big).
\end{align*}
Now we introduce the sum over $V_i$ for $2 \leq i  \leq \ell-3$. For $\theta \in \{ \frac{1}{2}\left(\frac{3}{2}-\sigma\right), \frac{1}{6}\}$  we have that
$$
\sum_{\deg(V_i) =n_i}   \sum_{d_i|V_i} |d_i|^{i+\frac{1}{2}- (i+1)\sigma} \Big| \frac{V_i^i}{d_i^{2i+2-\ell}} \Big|^{\theta}   \ll
\begin{cases}
q^{n_i(1+i\theta)} & \mbox{ if } \theta(\ell-2-2i)+i+\frac{1}{2} -(i+1)\sigma \leq 0 \\
q^{n_i(\theta(\ell-2-i)+i+\frac{3}{2}-(i+1)\sigma)} & \mbox{ otherwise.}
\end{cases}
$$
We similarly have that
\begin{align*}
\sum_{\deg(V_{\ell-2})=n_{\ell-2}} \sum_{d_{\ell-2}|EbV_{\ell-2}} |d_{\ell-2}|^{\ell-\frac{3}{2}-(\ell-1)\sigma} \Big|\frac{ V_{\ell-2}^{\ell-2}}{d_{\ell-2}^{\ell-2}} \Big|^{\theta} \ll q^{n_{\ell-2}(1+\theta (\ell-2))},
\end{align*}
\begin{align*}
\sum_{\deg(V_{\ell-1})=n_{\ell-1}} |V_{\ell-1}|^{(\ell-1)\theta}\ll q^{n_{\ell-1}(1+\theta (\ell-1))},
\end{align*}
\begin{align*}
\sum_{\deg(V_{\ell}) = n_{\ell}} \sum_{E | V_{\ell}^{*}} |E|^{\frac{1}{2}-\sigma+\varepsilon+(\ell-2)\theta} \ll q^{n_{\ell}(\frac{3}{2}-\sigma+(\ell-2)\theta)}.
\end{align*}
Putting all of these together, we get that
\begin{align}
\eqref{sumv} & \ll |b|^{1/2-\sigma+\varepsilon} \sum_{n_1+2n_2+\cdots+\ell n_{\ell}=n} \Big( q^n +\max\{ q^{n_1 \max\{ \frac{\ell(9-6\sigma)-6\sigma-1}{12} , \frac{7-2\sigma}{4} \}} |b^{\ell-2}|^{\frac{1}{2}\left(\frac{3}{2}-\sigma\right)}  \nonumber \\
& \times  \prod_{\substack{2 \leq i < { \frac{(3-2\sigma)(\ell-2)}{2}+1-2\sigma }}} q^{n_i (\frac{\ell(3-2\sigma)}{4}+\frac{i(1-2\sigma)}{4} )}  \prod_{\substack{i \geq
{ \frac{(3-2\sigma)(\ell-2)}{2}+1-2\sigma }}} q^{n_i (1+\frac{i}{2}\left(\frac{3}{2}-\sigma\right))}  \nonumber \\
&\times q^{n_1 \max\{ \frac{\ell}{6}-2\sigma+\frac{5}{3}, \frac{7}{6} \}} |b^{\ell-2}|^{\frac{1}{6}} \prod_{2 \leq i< \frac{ \frac{\ell+1}{6} -\sigma}{\sigma-\frac{2}{3}}} q^{n_i ( \frac{\ell+7+5i}{6}-(i+1)\sigma)} \prod_{i \geq  \frac{ \frac{\ell+1}{6} -\sigma}{\sigma-\frac{2}{3}}} q^{n_i (1+\frac{i}{6})} \} \Big).\label{mixedbound}
\end{align}
Now note that 
\begin{align*}
 \prod_{\substack{2 \leq i < { \frac{(3-2\sigma)(\ell-2)}{2}+1-2\sigma }}} q^{n_i (\frac{\ell(3-2\sigma)}{4}+\frac{i(1-2\sigma)}{4} )} \leq q^{\left(n-n_1-\sum_{i \geq  \frac{(3-2\sigma)(\ell-2)}{2}+1-2\sigma } in_i\right) \left(  \frac{ \ell(3-2\sigma)}{8} - \frac{2\sigma-1}{4}\right)},
\end{align*}
so \begin{align*}
&  q^{n_1 \max\{ \frac{\ell(9-6\sigma)-6\sigma-1}{12} , \frac{7-2\sigma}{4} \}}   \prod_{\substack{2 \leq i < { \frac{(3-2\sigma)(\ell-2)}{2}+1-2\sigma }}} q^{n_i (\frac{\ell(3-2\sigma)}{4}+\frac{i(1-2\sigma)}{4} )}  \prod_{\substack{i \geq
{ \frac{(3-2\sigma)(\ell-2)}{2}+1-2\sigma }}} q^{n_i (1+\frac{i}{2}\left(\frac{3}{2}-\sigma\right))} \\
& \leq q^{n  (  \frac{ \ell(3-2\sigma)}{8} - \frac{2\sigma-1}{4})} q^{n_1 \Big(  \max\{ \frac{\ell(9-6\sigma)-6\sigma-1}{12} , \frac{7-2\sigma}{4} \} - ( \frac{\ell(3-2\sigma)}{8}-\frac{2\sigma-1}{4})\Big)} \prod_{\substack{i \geq
{ \frac{(3-2\sigma)(\ell-2)}{2}+1-2\sigma }}} q^{n_i \Big(\frac{i\ell(2\sigma-3)}{8}+\frac{i}{2}+1\Big)} \\
& \ll 
q^{ n \max\{ \frac{\ell(9-6\sigma)-6\sigma-1}{12} , \frac{7-2\sigma}{4} \}}.
\end{align*}

We similarly get that
\begin{align*}
& q^{n_1 \max\{ \frac{\ell}{6}-2\sigma+\frac{5}{3}, \frac{7}{6} \}}  \prod_{2 \leq i< \frac{ \frac{\ell+1}{6} -\sigma}{\sigma-\frac{2}{3}}} q^{n_i ( \frac{\ell+7+5i}{6}-(i+1)\sigma)} \prod_{i \geq  \frac{ \frac{\ell+1}{6} -\sigma}{\sigma-\frac{2}{3}}} q^{n_i (1+\frac{i}{6})} \\
& \ll q^{n \max\{ \frac{\ell}{6}-2\sigma+\frac{5}{3}, \frac{7}{6} \}} \prod_{i \geq  \frac{ \frac{\ell+1}{6} -\sigma}{\sigma-\frac{2}{3}}} q^{n_i (1-\frac{i}{2} ( \frac{\ell}{6} +\frac{15}{6}-3\sigma))}.
\end{align*}

Using the two bounds above and \eqref{mixedbound}, we get that
\begin{align*}
\eqref{sumv} & \ll |b|^{1/2-\sigma+\varepsilon} q^{\varepsilon n} \Big(q^n +q^{ n \max\{ \frac{\ell(9-6\sigma)-6\sigma-1}{12} , \frac{7-2\sigma}{4} \}} |b|^{(\ell-2)\frac{1}{2} (\frac{3}{2}-\sigma)}\\
& +q^{n \max\{ \frac{\ell}{6}-2\sigma+\frac{5}{3}, \frac{7}{6} \}} |b|^{\frac{\ell-2}{6}} \sum_{n_1+2n_2+\cdots+\ell n_{\ell}=n}  \prod_{i \geq  \frac{ \frac{\ell+1}{6} -\sigma}{\sigma-\frac{2}{3}}} q^{n_i (1-\frac{i}{2} ( \frac{\ell}{6} +\frac{15}{6}-3\sigma))}  \Big) \Big),
\end{align*}
which finishes the proof.
 \end{proof}

\begin{corollary}
\label{lov_sigma}
When $\sigma=1+1/\ell+\varepsilon$ and $\ell \geq 11$, we have that
\begin{align*}
&\sum_{V \in \mathcal{M}_n}  |\psi^{(i)}_V(V,b,z)| \ll   |b|^{\frac{\ell}{4}-\frac{3}{2}+\varepsilon}   q^{n ( \frac{\ell}{4} - \frac{13}{12} - \frac{1}{2\ell} )+\varepsilon n}.
\end{align*}
\end{corollary}

 \begin{remark}
 By comparison,  the convexity bound of Corollary \ref{cor_convexity} yields for  $\sigma=1+1/\ell+\varepsilon$  and  $\ell\geq 8$,
 \begin{align*}
\sum_{V \in \mathcal{M}_n} & |\psi^{(i)}_V(V,b,z)| \ll |b|^{\frac{\ell}{4}-\frac{3}{2}+\varepsilon}   q^{n(\frac{\ell}{4}-\frac{3}{4}-\frac{1}{2\ell})+\varepsilon n}.
\end{align*}
                       \end{remark}

\section{One-level density for characters of order $\ell$ - setup}
\label{section_setup}
Recall that we are considering the family $\mathcal{F}_{\ell}(d)$ of order $\ell$ Dirichlet characters, given in equation \eqref{family}.
We also recall the expression \eqref{eq:rootsprod} for the $L$--function associated to $\chi_c$, where $c$ is square-free of degree $d$.

We are interested in evaluating the one-level density of zeros, which is defined in \eqref{1ld_def} and \eqref{sigma}. We rewrite it here as 
\begin{align} \label{OLD}
\left\langle\Sigma(\phi,\chi_c)\right\rangle_{\mathcal{F}_\ell(d)} = \frac{1}{|\mathcal{F}_\ell(d)|}\sum_{\chi \in \mathcal{F}_\ell(d)}\Sigma(\phi,\chi_c).\end{align}

\subsection{The explicit formula}We will follow the computation of the explicit formula in \cite{Bui-Florea}.
By computing the logarithmic derivative $\frac{\cL'}{\cL}$ in two different ways from \eqref{eq:Eulerprod} and \eqref{eq:rootsprod}, we get, for $n>0$, 
\[-q^\frac{n}{2} \sum_{j=1}^{d-1} e\left(-n \theta_{\chi_c,j}\right)=\sum_{f\in \mathcal{M}_n} \Lambda(f) \chi_c(f),\]
where $e(x) = e^{2 \pi i x}$.
The above can be written as 
\[-\sum_{j=1}^{d-1} e\left(n \theta_{\chi_c,j}\right)=\sum_{f\in \mathcal{M}_{|n|}} \frac{\Lambda(f) \overline{\chi_c(f)}}{|f|^\frac{1}{2}},\]
which is valid for $n$ both positive and negative.

We recall that 
\begin{align*}
\Phi(\theta)=\sum_{k\in \Z}\phi((d-1)(k-\theta)) =: \sum_{k\in \Z} f(k)             . \end{align*}
We compute
\begin{align*}
\widehat{f}(n)
=\int_{-\infty}^\infty \phi((d-1)(x-\theta))e(nx)dx
= \frac{\widehat{\phi}\left(\frac{n}{d-1}\right) e\left(n\theta\right)}{d-1}, 
\end{align*}
and by Poisson summation
\begin{align*}
\Phi(\theta)=\frac{1}{d-1} \sum_{n\in \Z} \widehat{\phi}\left(\frac{n}{d-1}\right) e\left(n\theta\right) = \frac{1}{d-1} \sum_{|n| \leq v (d-1) }\widehat{\phi}\left(\frac{n}{d-1}\right) e\left(n\theta\right),
             \end{align*}
             where the second equality follows from the fact that the support of $\widehat{\phi}$ is  in $(-v, v)$.
             Let
             \begin{align*} %\label{def-N}
             N := v (d-1). \end{align*}
Summing over all the $\theta_{\chi_c,j}$ for $j=1,\ldots,d-1$,  we have that
 \begin{align*}
\Sigma(\phi, \chi_c) = \sum_{j=1}^{d-1} \Phi(\theta_{\chi_c,j}) =& \widehat{\phi}(0)+ \frac{1}{d-1} \sum_{0<|n|\leq N} \widehat{\phi}\left(\frac{n}{d-1}\right) \sum_{j=1}^{d-1} e(n\theta_{\chi_c,j})\\
=&  \widehat{\phi}(0)- \frac{1}{d-1}\sum_{1 \leq n \leq N} \widehat{\phi}\left(\frac{n}{d-1}\right) \sum_{f\in \mathcal{M}_{n}} \frac{\Lambda(f) (\chi_c(f)+\overline{\chi_c(f)})}{|f|^\frac{1}{2}}, 
\end{align*}
and replacing in \eqref{OLD}, this gives
$$ \left\langle\Sigma(\phi,\chi_c)\right\rangle_{\mathcal{F}_\ell(d)} =  \widehat{\phi}(0) - \frac{1}{(d-1)|\mathcal{H}_d|}\sum_{1 \leq n\leq N} \widehat{\phi}\left(\frac{n}{d-1}\right) \sum_{c \in \mathcal{H}_d} \sum_{f\in \mathcal{M}_{n}} \frac{\Lambda(f) (\chi_c(f)+\overline{\chi_c(f)})}{|f|^\frac{1}{2}}.$$

Now note that the contribution from $f = \pi^k$ with $k \geq 2$  and $k \not \equiv 0 \pmod{\ell}$ is bounded up to a constant by $q^{-d/2+\varepsilon d}$. Indeed, note that using the Lindel\"of hypothesis, after removing the square-free condition on $c$, we have that
$$\sum_{c \in \mathcal{H}_d} \chi_c(f) \ll q^{d/2+\varepsilon d},$$
and then we trivially bound the contribution from the primes.

The contribution from $f=\pi^k$ with $k \equiv 0 \pmod{\ell}$ is bounded up to a constant by $1/d$. 
So it is enough to consider the contribution from $f = \pi$ for $\pi$ a prime. We rewrite
\begin{align}
\label{sigma1avg}
\left\langle\Sigma(\phi,\chi_c)\right\rangle_{\mathcal{F}_\ell(d)} =  \widehat{\phi}(0)- S+O\left(\frac{1}{d}\right),
\end{align}
where $S$ corresponds to the contribution from $f = \pi$, for $\pi$ a prime.
Namely, we have
\begin{align}
\label{expression_s}
S= \frac{1}{(d-1)|\mathcal{H}_d|} & \sum_{1 \leq n\leq N} \widehat{\phi}\left(\frac{n}{d-1}\right) \sum_{ \pi \in \mathcal{P}_{n}} \frac{\deg(\pi)}{\sqrt{|\pi|}} \sum_{c \in \mathcal{H}_d}( \chi_c(\pi)+ \overline{\chi_c(\pi)}) .
\end{align}

\begin{remark} \label{remark-sum-over-n}
 Note that by using the Lindel\"{o}f bound on the sum over $c$, we get the trivial bound
 \begin{equation}
 S \ll q^{\frac{N}{2} - \frac{d}{2}+\varepsilon d}.\label{trivial}
 \end{equation}
Breaking the sum over $n$ in \eqref{expression_s} in the two sums $1 \leq n \leq d(1-3\varepsilon)$ and
$d(1-3\varepsilon) < n \leq N$ for any $\varepsilon >0$, and using \eqref{trivial} to bound the first sum, we have
 $$
 S = \frac{1}{(d-1)|\mathcal{H}_d|}  \sum_{d(1-\varepsilon) \leq n\leq N} \widehat{\phi}\left(\frac{n}{d-1}\right) \sum_{ \pi \in \mathcal{P}_{n}} \frac{\deg(\pi)}{\sqrt{|\pi|}} \sum_{c \in \mathcal{H}_d}( \chi_c(\pi)+ \overline{\chi_c(\pi)})  + O \left(q^{-\varepsilon d} \right),
 $$
 so we can consider that ${d(1-\varepsilon) \leq n\leq N}$ when we need $n$ large (after a relabeling of $\varepsilon$). We will use this in Section \ref{proof} to choose the value of some parameters.
 \end{remark}

We write $S=S_1+\overline{S_1}$, where $S_1$ corresponds to the sum involving the character $\chi_c(\pi)$. We have that 

 \begin{align*}
S_1 &= \frac{1}{(d-1)|\mathcal{H}_d|}  \sum_{1 \leq n\leq N} \widehat{\phi}\left(\frac{n}{d-1}\right) \sum_{ \pi \in \mathcal{P}_{n}} \frac{\deg(\pi)}{\sqrt{|\pi|}} \sum_{c \in \mathcal{H}_d} \chi_c(\pi) \\
 & =  \frac{1}{(d-1)|\mathcal{H}_d|}  \sum_{1 \leq n\leq N} \widehat{\phi}\left(\frac{n}{d-1}\right) \sum_{ \pi \in \mathcal{P}_{n}} \frac{\deg(\pi)}{\sqrt{|\pi|}} \sum_{D \in \mathcal{M}_{\leq d/2}} \mu(D) \chi_{D^2}(\pi) \sum_{c \in \mathcal{M}_{d-2\deg(D)}} \chi_c(\pi).
 \end{align*}

 We further write 
 $$S_1=S_{1, \leq y}+S_{1,>y},$$ where $S_{1,\leq y}$ corresponds to the terms with $\deg(D) \leq y$ and where $S_{1,>y}$ corresponds to the terms with $\deg(D)>y$. 
 For $S_{1,>y}$, we use the Weil bound \eqref{weil} for the sum over primes, and we have that 
 $$ \sum_{\pi \in \mathcal{P}_n}  \chi_{\pi}(cD^2) \ll q^{n/2} \frac{d}{n}.$$
 Trivially bounding the sums over $c$ and $D$, we get that 
 \begin{equation}S_{1,>y} \ll d q^{-y},
 \label{s1>y}
 \end{equation}

We will now bound the term $S_{1, \leq y}$. 
We write 
\begin{equation}
\label{s1<y}
S_{1,\leq y} = S_{11}+S_{12},
\end{equation} where $S_{11}$ corresponds to the sum over $\deg(\pi) \equiv 0 \pmod{\ell}$, and where $S_{12}$ corresponds to the sum over $\deg(\pi) \not\equiv 0\pmod{\ell}$.  We will focus on $S_{12}$,   since $S_{11}$ is similar.  Using Poisson summation (Lemma \ref{poisson}) on the sum over $c$, we have
\begin{align*}
S_{12} &=  \frac{q^{d+1/2}}{(d-1)|\mathcal{H}_d|}  \sum_{\substack{ 1 \leq n\leq N \\ n \not \equiv 0 \pmod \ell}} \widehat{\phi}\left(\frac{n}{d-1}\right) \sum_{ \pi \in \mathcal{P}_{n}} \frac{\deg(\pi) \overline{\varepsilon(\chi_{\pi})}}{|\pi|^{3/2}} \sum_{D \in \mathcal{M}_{\leq y}} \frac{ \mu(D) \chi_{D^2}(\pi)}{|D|^2} \\
& \times \sum_{V \in \mathcal{M}_{n-d+2\deg(D)-1}} G_\ell(V,\pi).
\end{align*} 
We interchange the sums over $\pi$ and $V$. We keep in mind that $\varepsilon(\chi_{\pi})$ only depends on the degree of $\pi$.  For simplicity,  we denote $\varepsilon(\chi_{\pi})=\varepsilon(\deg(\pi))$.  Then we get that
\begin{align*}
S_{12} &= \frac{q^{d+1/2}}{(d-1)|\mathcal{H}_d|}  \sum_{\substack{ 1 \leq n\leq N \\ n \not \equiv 0 \pmod \ell}} \widehat{\phi}\left(\frac{n}{d-1}\right) \frac{n \overline{\varepsilon(n)}}{q^{3n/2}}\sum_{D \in \mathcal{M}_{\leq y}} \frac{ \mu(D)}{|D|^2}   \sum_{V \in \mathcal{M}_{n-d+2\deg(D)-1}} \\
& \times  \sum_{\substack{ \pi \in \mathcal{P}_{n} \\ (\pi, VD)=1}} G_\ell(VD^{2\ell-2},\pi).
\end{align*}
\subsection{Vaughan's identity}
Following Heath-Brown and Patterson \cite{hbp},  we define $\Sigma_j(n,R,U)$ by
$$\Sigma_j(n,R, U) = \sum_{\substack{a,b,c}} \Lambda(a) \mu(b) {G}_\ell (R, abc),$$
where 
\begin{align*}
\text{if   } &j=0 &\deg(abc)=n, (R, abc)=1, \, \, \,  \,  &\deg(bc) \leq U, \\
&j=1 &\deg(abc)=n, (R, abc)=1, \, \, \, \,   &\deg(b) \leq U, \\
&j=2' &\deg(abc)=n, (R, abc)=1,\, \, \, \,   &\deg(ab) \leq U, \\
&j=2'' &\deg(abc)=n, (R, abc)=1,\, \, \, \,   &\deg(a) \leq U, \deg(b) \leq U, \deg(ab) >U \\
&j=3 &\deg(abc)=n, (R, abc)=1,\, \, \, \,   &\deg(a) > U,  \deg(bc)>U, \deg(b) \leq U,\\
&j=4 &\deg(abc)=n, (R, abc)=1,\, \, \, \,   &\deg(bc) \leq U,  \deg(a) \leq U, 
\end{align*} where $U=U(n)$.
Note that we have 
\begin{equation}
\label{sumsigma}
\Sigma_0(n,R,U) + \Sigma_{2'}(n,R,U)+ \Sigma_{2''}(n,R,U)+\Sigma_3(n,R,U) = \Sigma_1(n,R,U)+\Sigma_4(n,R,U).
\end{equation}
If we assume that  $U=U(n) <n/2$, then $\Sigma_4(n,R,U)=0$.

Now similarly as in \cite{hbp}, we define
$$H(n, R) := \sum_{\substack{c \in \mathcal{M}_n\\ (c, R)=1}} \Lambda(c) G_\ell(R, c) = \sum_{\substack{\pi \in \mathcal{P}_n\\ (\pi, R)=1}} G_\ell(R, \pi) \deg(\pi)$$
by Lemma \ref{gauss}(i).
Note that we have
\begin{equation}
H(n, R) = \Sigma_0(n,R, U).
\label{hn}
\end{equation}
Indeed,  if we denote $d=bc$, then we rewrite
$$\Sigma_0(n,R,U) = \sum_{\substack{\deg(ad)= n \\ (ad, R)=1 \\\deg(d) \leq U}} \Lambda(a) G_\ell(R, ad) \sum_{b|d} \mu(b).$$
The sum over $b$ above is $0$ unless $d=1$,  which proves \eqref{hn}.

Now for $\alpha \in \mathbb{F}_q[t]$ monic,  let 
\begin{equation}
\label{fn}
F(n, R, \alpha)= \sum_{\substack{c \in \mathcal{M}_n \\ (c,R)=1 \\ \alpha|c}} G_\ell(R, c).
\end{equation}
We have the following lemma, which is the analogue of Proposition 1 in \cite{hbp}.
\begin{lemma}
\label{lemma_sigma}
We have the following.
\begin{enumerate}
\item $ \displaystyle \Sigma_1(n,R,U)= \sum_{\substack{\deg(b) \leq U\\ (b,R)=1}} \mu(b)  (n-\deg(b)) F(n,R,b)$,
\item $ \displaystyle \Sigma_{2'}(n,R, U) =  \sum_{\substack{\deg(b) \leq U \\ (b,R)=1}} \mu^2(b) h(b) F(n,R, b)$,
\end{enumerate}
where $|h(b)| \leq \deg(b)$.
\end{lemma}
\begin{proof}
Let $d = ac$.  Then
$$\Sigma_1(n,R,U) = \sum_{\substack{\deg(bd) = n \\ (bd, R)=1\\ \deg(b) \leq U}}\mu(b) G_\ell(R, bd) \sum_{a|d} \Lambda(a).$$
Note that $\sum_{a|d} \Lambda(a) = \deg(d)$. Now let $e = bd$.  Then
\begin{align*}
\Sigma_1(n,U)& =\sum_{\substack{\deg(b) \leq U\\ (b,R)=1}} \mu(b) \sum_{\substack{ \deg(e) = n \\ (e, R)=1 \\ b|e}} G_\ell(R, e) (\deg(e)-\deg(b))=\sum_{\substack{\deg(b) \leq U\\ (b,R)=1}} \mu(b)  (n-\deg(b)) F(n,R,b).
\end{align*}
To prove the second identity,  let $d=ab$,  and let
$$h(d) = \sum_{\substack{a|d}} \Lambda(a) \mu(d/a).$$
Note that $|h(d)| \leq \deg(d)$.

We have that
$$\Sigma_{2'}(n,R, U) = \sum_{\substack{\deg(cd) =n\\(cd,R)=1 \\ \deg(d) \leq U}} G_\ell(R, cd) h(d) = \sum_{\substack{\deg(d) \leq U \\ (d,R)=1}}\mu^2(d)  h(d) F(n,R,d),$$
where we have used one more time that if $(cd,R)=1$, then $G_\ell(R, cd)=0$ unless $cd$ is square-free. This finishes the proof.
\end{proof}

Now using \eqref{sumsigma}, with $U<n/2$, we rewrite 
\begin{align}
\label{S12}
S_{12} & =  \frac{q^{d+1/2}}{|\mathcal{H}_d|}  \sum_{\substack{ 1 \leq n\leq N \\ n \not \equiv 0 \pmod \ell}} \widehat{\phi}\left(n\right) \frac{ \overline{\varepsilon(n)}}{q^{3n/2}}\sum_{D \in \mathcal{M}_{\leq y}} \frac{ \mu(D)}{|D|^2}   \sum_{V \in \mathcal{M}_{n-d+2\deg(D)-1}}\Sigma_0(n,VD^{2\ell-2},U) \\
&= \frac{q^{d+1/2}}{|\mathcal{H}_d|}  \sum_{\substack{ 1 \leq n\leq N \\ n \not \equiv 0 \pmod \ell}} \widehat{\phi}\left(n\right) \frac{ \overline{\varepsilon(n)}}{q^{3n/2}}\sum_{D \in \mathcal{M}_{\leq y}} \frac{ \mu(D)}{|D|^2}   \sum_{V \in \mathcal{M}_{n-d+2\deg(D)-1}} \Big( \Sigma_1(n,VD^{2\ell-2},U) \nonumber \\
&-\Sigma_{2'}(n,VD^{2\ell-2},U)-\Sigma_{2''}(n,VD^{2\ell-2},U)-\Sigma_3(n,VD^{2\ell-2},U)\Big) .\nonumber 
\end{align}
We will denote the contribution from $\Sigma_1$ and $\Sigma_{2'}$ by Type I terms, and the contribution from $\Sigma_{2''}$ and $\Sigma_3$ by Type II terms. To bound the Type I terms, we will use our results on the generating series of Gauss sums, and to bound the contribution from Type II sums, we use the large sieve inequality in Theorem \ref{largesieve}.
\section{Type II sums}
\label{typeII_section}
Here, we will bound the contribution from $\Sigma_{2''}$ in \eqref{S12}. Bounding the contribution from $\Sigma_3$ is similar. We denote the contribution from $\Sigma_{2''}$ by $S_{12}(\Sigma_{2''})$ and we rewrite it as
\begin{align}
\label{tb3}
S_{12}(\Sigma_{2''}) &= \frac{q^{d+1/2}}{ (d-1)|\mathcal{H}_d|}  \sum_{\substack{ 1 \leq n\leq N \\ n \not \equiv 0 \pmod \ell}} \widehat{\phi}\left( \frac{n}{d-1}\right) \frac{ \overline{\varepsilon(n)}}{q^{3n/2}}\sum_{D \in \mathcal{M}_{\leq y}} \frac{ \mu(D)}{|D|^2}   \sum_{V \in \mathcal{M}_{n-d+2\deg(D)-1}} \\
& \times \sum_{\substack{\deg(abc)=n \\ (abc, VD)=1 \\ \deg(a), \deg(b) \leq U \\ \deg(ab)>U}} \Lambda(a) \mu(b) G(VD^{2\ell-2}, abc). \nonumber
\end{align}
We recall that $U=U(n)$ is a function of $n$, and $U(n) < n/2$. 
Let $ab=\alpha$ and
$$A(\alpha) = \sum_{\substack{ab=\alpha \\ \deg(a), \deg(b)\leq U\\ (\alpha, VD^{2\ell-2})=1}} \Lambda(a) \mu(b) G(VD^{2\ell-2},\alpha).$$
Using Lemma \ref{gauss}, we have that
\begin{align}
\label{in4}
\Sigma_{2''}(n,VD^{2\ell-2},U) = \sum_{\substack{\deg(abc)=n \\ (abc, VD)=1 \\ \deg(a), \deg(b)\leq U \\ \deg(ab)>U}} \Lambda(a) \mu(b) G(VD^{2\ell-2}, abc) = \sum_{\substack{\alpha c \in \mathcal{H}_n \\ U < \deg(\alpha) \leq 2U}} A(\alpha) G(VD^{2\ell-2},c) \Big(\frac{\alpha}{c} \Big)_{\ell}^2.
\end{align}
Note that in the above equation, we removed the coprimality condition by restricting to square-free polynomials $\alpha, c$ (see Lemma \ref{gauss}.) Using Cauchy--Schwarz, we now have that
\begin{align} 
\Sigma_{2''}(n,VD^{2\ell-2},U) &= \sum_{U<m \leq 2U} \sum_{\alpha \in \mathcal{H}_m} \sum_{c \in \mathcal{H}_{n-m}} A(\alpha) G(VD^{2\ell-2},c) \Big(\frac{\alpha}{c} \Big)_{\ell}^2 \nonumber  \\
& \leq \sum_{U < m \leq 2U} \Big(  \sum_{\alpha \in \mathcal{H}_m} |A(\alpha)|^2 \Big)^{1/2} \Big(\sum_{\alpha \in \mathcal{H}_m} \Big| \sum_{c \in \mathcal{H}_{n-m}} G(VD^{2\ell-2},c) \Big(\frac{\alpha}{c} \Big)_{\ell}^2 \Big|^2 \Big)^{1/2}.\label{interm3}
\end{align}
Using Lemma \ref{gauss}, note that $|A(\alpha) | \ll |\alpha|^{1/2+\varepsilon}$. Then
$$\sum_{\alpha \in \mathcal{H}_m} |A(\alpha)|^2  \ll q^{2m+\varepsilon m}.$$
Now using the Large Sieve Inequality as in Theorem \ref{largesieve}, and keeping in mind that $|G(VD^{2\ell-2},c)|=|c|^{1/2}$ (see Lemma \ref{gauss}), we have that
\begin{align*}
\sum_{\alpha \in \mathcal{H}_m}  \Big| \sum_{c \in \mathcal{H}_{n-m}} G(VD^{2\ell-2},c) \Big(\frac{\alpha}{c} \Big)_{\ell}^2 \Big|^2  & \ll q^{\varepsilon(m+n)} (q^m+q^{n-m}+q^{2n/3}) \sum_{c \in \mathcal{H}_{n-m}} |G(VD^{2\ell-2},c)|^2 \\
& \ll q^{\varepsilon(m+n)} \Big(q^{2n-m} + q^{3(n-m)}+q^{\frac{8n}{3}-2m} \Big).
\end{align*}

Combining equations \eqref{in4}, \eqref{interm3} and the two bounds above, we get that 
\begin{align}
\label{typeII}
\Sigma_{2''}(n,VD^{2\ell-2},U)  \ll q^{\varepsilon n } \Big( q^{n+U} + q^{\frac{3n}{2}-\frac{U}{2}} + q^{\frac{4n}{3}}\Big).
\end{align}
Now we pick 
\begin{align*} %\label{cond}
 U=U(n) \leq n/3,
 \end{align*} and the middle term of \eqref{typeII}
 dominates. 
We also assume that $U=U(n)$ is a (positive) non-decreasing function of $n$, such that
\begin{align} \label{sum-over-n}
 \sum_{1 \leq n \leq N} q^{U(n)} \leq q^{U(N)} q^{\varepsilon d}
\end{align}
where we recall that $N=v(d-1)$. We will often use that fact in the computations of the next sections.

Replacing \eqref{typeII} in  \eqref{tb3}, trivially bounding the sums over $V$ and $D$, and using \eqref{sum-over-n}, we get that 
\begin{align*}
S_{12}(\Sigma_{2''}) \ll q^{N-\frac{U(N)}{2}-d+y+N\varepsilon},
\end{align*}
Now choosing $y = 100\varepsilon d$, and after a  relabeling of the $\varepsilon$, we get for that for $U(n) \leq \tfrac{n}{3},$
\begin{align}
\label{s12sigma2''}
S_{12}(\Sigma_{2''}) \ll q^{N-\frac{U(N)}{2}-d+d \varepsilon}.
\end{align}

\section{Type I sums}
\label{typeI}
Here, we will bound the contribution from the Type I terms in equation \eqref{S12}, coming from the terms involving $\Sigma_1$ and $\Sigma_{2'}$. We denote these terms by $S_{12}(\Sigma_1)$ and $S_{12}(\Sigma_{2'})$ respectively.

We only bound $S_{12}(\Sigma_{2'})$, the treatment of $S_{12}(\Sigma_{1})$ being similar. Using Lemma \ref{lemma_sigma}, we rewrite
\begin{align*}
S_{12}(\Sigma_{2'}) &= \frac{q^{d+1/2}}{(d-1)|\mathcal{H}_d|}  \sum_{\substack{ 1 \leq n\leq N \\ n \not \equiv 0 \pmod \ell}} \widehat{\phi}\left(\frac{n}{d-1}\right) \frac{ \overline{\varepsilon(n)}}{q^{3n/2}}\sum_{D \in \mathcal{M}_{\leq y}} \frac{ \mu(D)}{|D|^2}   \sum_{V \in \mathcal{M}_{n-d+2\deg(D)-1}} \\
& \times   \sum_{\substack{\deg(b) \leq U \\ (b,VD)=1}} \mu^2(b) h(b)  F(n,VD^{2\ell-2},b),
\end{align*}
where recall that $F(n,VD^{2\ell-2},b)$ is given in equation \eqref{fn}. 

Using Perron's formula, we rewrite 
\begin{align}
S_{12}(\Sigma_{2'}) &= \frac{q^{d+1/2}}{(d-1)|\mathcal{H}_d|}  \sum_{\substack{ 1 \leq n\leq N \\ n \not \equiv 0 \pmod \ell}} \widehat{\phi}\left(\frac{n}{d-1}\right) \frac{ \overline{\varepsilon(n)}}{q^{3n/2}}\sum_{D \in \mathcal{M}_{\leq y}} \frac{ \mu(D)}{|D|^2}   \sum_{V \in \mathcal{M}_{n-d+2\deg(D)-1}}  \sum_{\substack{\deg(b) \leq U \\ (b,VD)=1}} \mu^2(b)  h(b) \nonumber  \\
& \times \frac{1}{2 \pi i} \oint \frac{\sum_{i=0}^{\ell-1}\psi_{VD}^{(i)}(VD^{2\ell-2},b, z)}{z^{n+1}} \, dz, \label{initial_formula}
 \end{align}
where we are integrating along a small circle around the origin.

  Using Theorem \ref{thm:gettingridofa}, we have 
\begin{align}
S_{12}(\Sigma_{2'}) 
&=  \frac{q^{d+1/2}}{(d-1)|\mathcal{H}_d|}  \sum_{\substack{ 1 \leq n\leq N \\ n \not \equiv 0 \pmod \ell}} \widehat{\phi}\left(\frac{n}{d-1}\right) \frac{ \overline{\varepsilon(n)}}{q^{3n/2}}\sum_{D \in \mathcal{M}_{\leq y}} \frac{ \mu(D)}{|D|^2}   \sum_{V \in \mathcal{M}_{n-d+2\deg(D)-1}} \sum_{\substack{\deg(b) \leq U \\ (b,VD)=1}}  \mu^2(b) h(b)\nonumber  \\
& \times \frac{1}{2 \pi i} \oint_{|z|=q^{-1-1/\ell-\varepsilon}} \sum_{i=0}^{\ell-1} \frac{1}{z^{n+1}} G_\ell(VD^{2\ell-2},b) z^{\deg(b)}\sum_{\substack{E\in \mathcal{M}\\ E \mid r_\ell^* }}\mu(E)G_\ell(b^{\ell-2}r_1\cdots r_{\ell-1}^{\ell-1},E)z^{\deg(E)} \nonumber \\&\times
  \prod_{\pi \mid bEr_1\cdots r_{\ell-1}}
\left(1 - q^{(\ell-1)\deg(\pi)} \legl{-1}{\pi}^{\nu_\pi(bEr_1\cdots r_{\ell-2}r_{\ell-1})+1}z^{\ell \deg(\pi)}\right)^{-1}\nonumber \\
&\times  \sum_{\substack{d_{\ell-2} \mid bEr_{\ell-2}\\ d_1\mid r_1, \dots, d_{\ell-3}\mid r_{\ell-3}}} \mu(d_1d_2\cdots d_{\ell-2} ) q^{\deg(d_1d_2^2\cdots d_{\ell-2}^{\ell-2})}  z^{\deg(d_1^2d_2^3\cdots d_{\ell-2}^{\ell-1}  )} \prod_{1\leq k< j \leq \ell-2} \overline{\legl{d_k}{d_j}^{k(j+1)}} \nonumber \\
& \times \prod_{j=1}^{\ell-2} G_{\ell,j+1} \left( \frac{b^{\ell-2}E^{\ell-2}r_1 r_2^2 \cdots r_{\ell-1}^{\ell-1}}{d_{1}d_2\cdots d_{\ell-2} ^{\ell-2}} , d_{j} \right)\nonumber \\
&  \times \psi^{(i-\deg(b)-\deg(E)- \deg(d_1^2d_2^3\cdots d_{\ell-2}^{\ell-1} ))} \left( \frac{b^{\ell-2}E^{\ell-2}r_1 r_2^2 \cdots r_{\ell-1}^{\ell-1}}{\prod_{j=1}^{\ell-2}d_j^{2j+2-\ell}}, z \right)\, dz,\label{long_expression}
 \end{align} 
 where, as before, $VD^{2\ell-2} = r_1 \ldots r_{\ell}^{\ell}$ and $r_{\ell}^{*}$ is the product of the primes dividing $r_{\ell}$ 
but not $r_1 \ldots r_{\ell-1}$. In the above, we have shifted the contour of integration to  $|z|=q^{-1-1/\ell-\varepsilon}$.
 We will bound the term $S_{12}(\Sigma_{2'})$ differently depending on the size of $\ell$.

 \subsection{The cases $\ell=3,4$}
Let $\alpha = r_{\ell}^{\ell} \prod_{j=1}^{\ell-2} d_j^{2j+2-\ell}/ E^{\ell-2}$, which is a polynomial for $\ell=3,4$. By trivially bounding the sums over $E, \pi, d_j$ we get that
\begin{align*}
S_{12}(\Sigma_{2'}) & \ll\max_{\substack{|z|=q^{-1-1/\ell-\varepsilon}}}\frac{1}{d}\sum_{n \leq N} \frac{q^{n (1+1/\ell+\varepsilon)}}{q^{3n/2}} \sum_{D \in \mathcal{M}_{\leq y}} \frac{1}{|D|^2}  \sum_{V \in \mathcal{M}_{n-d+2\deg(D)-1}}  \sum_{\substack{\deg(b) \leq U \\ (b,VD)=1}} \frac{\mu^2(b)}{|b|^{1/2+1/\ell-\varepsilon}}  \nonumber \\
& \times \sum_{\substack{\alpha \in \mathcal{M}\\\alpha| b^{\ell-2}VD^{2\ell-2}}} |VD^{2\ell-2}|^{\varepsilon} \sum_{i=0}^{\ell-1} \Big| \psi^{(i)} \Big(b^{\ell-2} \frac{VD^{2\ell-2}}{\alpha},z \Big) \Big|. %\label{another_estimate}
\end{align*}
Since $\alpha| b^{\ell-2} VD^{2\ell-2}$, we write $\alpha = \alpha_1 \alpha_2 \alpha_3$, where $\alpha_1 | V$, $\alpha_2 |D^{2\ell-2}$, and $\alpha_3|b^{\ell-2}$.  Since $(b,VD)=1$ note that we have $(\alpha_3,\alpha_2)=1$. 
We write $V = \alpha_1 V_1$, and then we get that
\begin{align*}
S_{12}(\Sigma_{2'}) &\ll  \max_{\substack{|z|=q^{-1-1/\ell-\varepsilon}}} \frac{1}{d} \sum_{n \leq N} \frac{q^{n (1+1/\ell+\varepsilon)}}{q^{3n/2}} \sum_{D \in \mathcal{M}_{\leq y}} \frac{1}{|D|^2} \sum_{\deg(b) \leq U} \frac{\mu^2(b)}{|b|^{1/2+1/\ell-\varepsilon}} \sum_{\deg(\alpha_1) \leq n-d+2\deg(D)-1}  \sum_{\alpha_2 | D^{2\ell-2}}  \nonumber \\
& \times  \sum_{\substack{\alpha_3| b^{\ell-2} \\ (\alpha_3,\alpha_2)=1}} \sum_{V_1 \in \mathcal{M}_{n-d+2\deg(D)-1-\deg(\alpha_1)}} \sum_{i=0}^{\ell-1}  \Big| \psi^{(i)} \Big(  V_1 \frac{D^{2\ell-2}b^{\ell-2}}{\alpha_2 \alpha_3},z \Big) \Big|,
\end{align*} 
Replacing $\alpha_2 \mapsto D^{2\ell-2}/\alpha_2$ and $\alpha_3 \mapsto b^{\ell-2}/\alpha_3$, we get that
\begin{align}
S_{12}(\Sigma_{2'}) &\ll \max_{\substack{|z|=q^{-1-1/\ell-\varepsilon}}}\frac{1}{d} \sum_{n \leq N} \frac{q^{n (1+1/\ell+\varepsilon)}}{q^{3n/2}} \sum_{D \in \mathcal{M}_{\leq y}} \frac{1}{|D|^2} \sum_{\deg(b) \leq U} \frac{\mu^2(b)}{|b|^{1/2+1/\ell-\varepsilon}} \sum_{\deg(\alpha_1) \leq n-d+2\deg(D)-1}  \sum_{\alpha_2 | D^{2\ell-2}}  \nonumber \\
& \times  \sum_{\substack{\alpha_3| b^{\ell-2} \\ (\alpha_3,\alpha_2)=1}} \sum_{V_1 \in \mathcal{M}_{n-d+2\deg(D)-1-\deg(\alpha_1)}}  \sum_{i=0}^{\ell-1} \Big| \psi^{(i)} (  V_1 \alpha_2 \alpha_3,z ) \Big|. \label{ssigma2}
\end{align}

Using Cauchy--Schwarz and Proposition \ref{lov} with $j=1$ and  $\sigma=1+1/\ell+\varepsilon$ in the form
$$ \sum_{V \in \mathcal{M}_n} \Big| \psi^{(i)}(hV, z)\Big|^2 \ll  q^{n+\varepsilon n} |h|^{1/3+\varepsilon},$$
we get that 
\begin{align*}
S_{12}(\Sigma_{2'}) & \ll \frac{1}{d} \sum_{n \leq N} \frac{q^{n (1+1/\ell+\varepsilon)}}{q^{3n/2}} \sum_{D \in \mathcal{M}_{\leq y}} \frac{1}{|D|^2} \sum_{\deg(b) \leq U} \frac{\mu^2(b)}{|b|^{1/2+1/\ell}} \sum_{\deg(\alpha_1) \leq n-d+2\deg(D)-1} \sum_{\alpha_2|D^{2\ell-2}} \nonumber  \\
& \times \sum_{\substack{\alpha_3|b^{\ell-2} \\ (\alpha_3,\alpha_2)=1}}  q^{(n-d+2\deg(D)-\deg(\alpha_1))(1+\varepsilon) } |\alpha_2\alpha_3|^{1/6}.
\end{align*}
Trivially bounding all the other terms and using \eqref{sum-over-n} , it follows that
\begin{equation*}
S_{12}(\Sigma_{2'}) \ll q^{N(1/2+1/\ell+\varepsilon)-d+ \frac{U(N)(\ell-2)}{2}( 1/3+1/\ell) +  y(1+(\ell-1) /3)}.
\end{equation*}
Recall that $y=100\varepsilon d$, and after a relabeling of the $\varepsilon$, it follows that
\begin{equation}
S_{12}(\Sigma_{2'}) \ll q^{N(1/2+1/\ell)-d+ \frac{U(N)(\ell-2)}{2}( 1/3+1/\ell) + d \varepsilon},
\label{ssigma3}
\end{equation}
and for $\ell=4$, we get the bound
\begin{equation*}
S_{12}(\Sigma_{2'}) \ll q^{N(3/4)-d+ U(N) ( 1/3+1/\ell) + d \varepsilon}.
%\label{ell=4}
\end{equation*}

For $\ell=3$, we can improve the bound above. We shift the contour of integration in \eqref{long_expression} to  $|z|=q^{-1}$ and encounter the pole at $z^3=q^{-4}$. We write 
\begin{equation}
\label{s2_mt}
S_{12}(\Sigma_{2'}) = M_2+E,
\end{equation} where $M_2$ corresponds to the residue at $z^3=q^{-4}$, and where $E$ is the expression in \eqref{long_expression} with the contour of integration $|z|=q^{-1}$.

 Similarly to equation \eqref{ssigma2} with $\ell=3$, we have that
\begin{align*}
E & \ll \max_{\substack{|z|=q^{-1}}}  \frac{1}{d} \sum_{n \leq N} \frac{q^{n +n \varepsilon}}{q^{3n/2}} \sum_{D \in \mathcal{M}_{\leq y}} \frac{1}{|D|^2} \sum_{\deg(b) \leq U} \frac{\mu^2(b)}{|b|^{1/2-\varepsilon}} \sum_{\deg(\alpha_1) \leq n-d+2\deg(D)-1}  \sum_{\alpha_2 | D^{4}}  \nonumber \\
& \times \sum_{\substack{\alpha_3|b \\ (\alpha_3,\alpha_2)=1}} \sum_{V_1 \in \mathcal{M}_{n-d+2\deg(D)-1-\deg(\alpha_1)}}\sum_{i=0}^{\ell-1}  \Big| \psi^{(i)} (  V_1 \alpha_2 \alpha_3,z ) \Big|. %\label{ssigma3'}
\end{align*}
Since $y=100 \varepsilon d$, and using Corollary \ref{cor_lov}, we have
\begin{align*}
E & \ll \sum_{n \leq N} \frac{q^{n(1+\varepsilon)}}{q^{3n/2}} \sum_{D \in \mathcal{M}_{\leq y}} \frac{1}{|D|^2} \sum_{\deg(b) \leq U} \frac{1}{|b|^{1/2-\varepsilon}} \sum_{\deg(\alpha_1) \leq n-d+2\deg(D)-1}  \sum_{\alpha_2 | D^{4}}  \nonumber \\
& \times \sum_{\substack{\alpha_3|b \\ (\alpha_3,\alpha_2)=1}} q^{(n-d+2\deg(D)-1-\deg(\alpha_1))(1+\varepsilon)}   |\alpha_2 \alpha_3|^{1/4}  \nonumber \\
& \ll q^{\frac{N}{2}-d+\frac{3U(N)}{4}+\varepsilon d}
%\label{ssigma_3}
\end{align*}
Using the above and equation \eqref{s2_mt}, we get that
\begin{equation}
S_{12}(\Sigma_{2'}) = M_2 + O \Big( q^{\frac{N}{2}-d+\frac{3U(N)}{4}+\varepsilon d}\Big).
\label{s22'}
\end{equation}
To bound the contribution from $M_2$, we proceed as before, but use Theorem \ref{residue} to bound the residue at $z^3 = q^{-4}$, which gives that
\begin{align*}
M_2  & \ll \sum_{n \leq N} \frac{q^{4n/3+n \varepsilon}}{q^{3n/2}} \sum_{D \in \mathcal{M}_{\leq y}} \frac{1}{|D|^2} \sum_{\substack{\deg(b) \leq U  \\(b,D)=1}} \frac{\mu(b)^2}{|b|^{5/6-\varepsilon}} \sum_{\substack{V \in \mathcal{M}_{n-d+2\deg(D)-1} \\ (V,b)=1}} \sum_{E|r_3^{*}} \frac{\mu(E)^2}{|E|^{5/6-\varepsilon}} \\
& \times \sum_{d|bEr_1} \frac{1}{|d|^{7/6-\varepsilon}}  \Big| \frac{bEr_1}{d} \Big|^{-1/6},
\end{align*}
where we write $VD^4=r_1r_2^2r_3^3$ with $r_1, r_2$ square-free and coprime, and where $r_3^{*}$ denotes the product of the primes dividing $r_3$ but not $r_1r_2$. Introducing the sum over $V$, we have that
\begin{align*}
\sum_{\substack{V \in \mathcal{M}_{n-d+2\deg(D)-1-\deg(\alpha_1)} \\ (V,b)=1}} \frac{1}{|r_1|^{1/6}} \ll q^{5/6(n-d+2\deg(D))} |D|^{\varepsilon}.
\end{align*}
Trivially bounding everything else, and since $y=100\varepsilon d$, we get that
$$M_2 \ll q^{5N/6-d+\varepsilon d}.$$

Using the equation above and \eqref{s22'}, it follows that 
\begin{equation}
\label{s12sigma'again}
S_{12}(\Sigma_{2'}) \ll  q^{\frac{N}{2}-d+\frac{3U(N)}{4}+ d \varepsilon}+q^{\frac{5N}{6}-d+\varepsilon d}, \;\; \text{for $\ell = 3$.}
\end{equation}

 \subsection{The cases $5 \leq \ell \leq 10$}
 We proceed as in the previous cases to bound the term \eqref{initial_formula}, but use the convexity bound in Corollary \ref{cor_convexity} rather than Corollary \ref{cor_lov}. Taking $|z|=q^{-1-1/\ell-\varepsilon}$ and $y=100 \varepsilon d$, using the pointwise bound in Corollary \ref{cor_convexity}, trivially bounding everything else and after a relabeling of the $\varepsilon$, we get that for $5 \leq \ell \leq 8$, we have
 \begin{align}
 \label{mediumell}
 S_{12}(\Sigma_{2'}) \ll q^{N ( \frac{3}{4}+\frac{1}{2\ell})-d(\frac{5}{4}-\frac{1}{2\ell})+ \frac{U(N)(\ell-2)}{4}+\varepsilon d},
 \end{align}
 while for $9 \leq \ell \leq 10$, we have 
 \begin{align}
 \label{mediumell2}
 S_{12}(\Sigma_{2'}) \ll q^{N (\frac{\ell}{4}-\frac{5}{4}+\frac{1}{2\ell}) - d (\frac{\ell}{4} - \frac{3}{4}-\frac{1}{2\ell})+\frac{U(N)(\ell-2)}{4}+\varepsilon d}.
 \end{align}

 \subsection{The case $\ell \geq 11$}
Using Corollary \ref{lov_sigma} rather than Corollary \ref{cor_convexity} in \eqref{initial_formula} and letting $A=VD^{2\ell-2}$, we get that
 \begin{align*}
 S_{12}(\Sigma_{2'}) \ll\max_{\substack{|z|=q^{-1-1/\ell-\varepsilon}}} \sum_{n \leq N} q^{n(\frac{1}{\ell}-\frac{1}{2}+\varepsilon)} \sum_{\deg(A) \leq n-d+2\ell y } |A|^{\varepsilon} \sum_{\substack{\deg(b) \leq U \\ (b,A)=1}} \mu^2(b) h(b) \sum_{i=0}^{\ell-1} |\psi_A^{(i)}(A,b,z)|,
 \end{align*}
 where we bounded $d(A) \ll |A|^{\varepsilon}$.
  Now we can use Corollary \ref{lov_sigma}, and since $y = 100 \varepsilon d$ and after a relabeling of $\varepsilon$, we get that
 \begin{equation} 
\label{bd_bigell}
S_{12}(\Sigma_{2'})\ll  q^{N(\frac{\ell}{4}-\frac{19}{12}+ \frac{1}{2\ell}) - d(\frac{\ell}{4}-\frac{13}{12}-\frac{1}{2\ell}) + \frac{U(N)(\ell-2)}{4}+\varepsilon d}.
\end{equation}

\section{The proofs of Theorem \ref{main_thm} and Corollary \ref{cor_nonvanishing}}
\label{proof}
Now we are finally ready to prove Theorem \ref{main_thm}. 
\subsection{The cases $\ell=3,4$}
When $\ell=4$, using equations \eqref{S12}, \eqref{s12sigma2''} and \eqref{ssigma3}, and since bounding the other Type I and Type II terms is similar to the bounds we have, we get that 
\begin{align*}
%\label{s12_bound}
S_{12} \ll  q^{N(1/2+1/\ell)-d+ \frac{ U(N) (\ell-2)}{2}( 1/3+1/\ell) + d \varepsilon} + q^{N - \frac{U(N)}{2}-d+d\varepsilon}, \;\; \text{for } U(n) \leq \tfrac{n}{3}.
\end{align*}
Specializing the above to $\ell=4$, we pick
$U(n)= \frac{3n}{13},$ and then we get that
$$S_{12} \ll q^{\frac{23N}{26}-d+d \varepsilon}.$$
Similarly bounding the term $S_{11}$, and using equations \eqref{s1<y} and \eqref{s1>y} and keeping in mind that $S=S_1 + \overline{S_1}$ (where $S$ is given by \eqref{expression_s}), we have that
$$S \ll q^{\frac{23N}{26}-d+d \varepsilon} +q^{-d\varepsilon}.$$
Using equation \eqref{sigma1avg}, we get that 
\begin{align*}
%\label{final_formula}
\left\langle\Sigma(\phi,\chi_c)\right\rangle_{\mathcal{F}_\ell(g)} =  \widehat{\phi}(0) + O \Big( q^{\frac{23N}{26}-d+d \varepsilon} +q^{-d\varepsilon} + 1/d \Big).
\end{align*}
It follows that as long as $N = v(d-1) < d \Big(\frac{26}{23}- \varepsilon \Big)$, the error terms above are negligible, and Theorem \ref{main_thm} follows in the case $\ell=4$ with $v < \tfrac{26}{23}$.

When $\ell=3$, using equations \eqref{s12sigma2''} and \eqref{s12sigma'again} and similarly to the case $\ell=4$, we get that
$$S  \ll  q^{\frac{N}{2}-d+\frac{3U(N)}{4}+ d \varepsilon}+q^{\frac{5N}{6}-d+\varepsilon d}+ q^{N-\frac{U(N)}{2}-d+d\varepsilon}+ q^{-\varepsilon d},    \;\; \text{for } U(n) \leq \tfrac{n}{3}.$$
We pick $U(n)=\tfrac{n}{3}$, and we get that
$$S \ll q^{\frac{5N}{6}-d+d \varepsilon} + q^{-d\varepsilon}.$$
Hence, as long as $N = v(d-1) < d(\frac{6}{5}-\varepsilon)$, we get that the error terms are negligible, and Theorem \ref{main_thm} follows in the case $\ell=3$ with $v < \tfrac{6}{5}$.

\subsection{The cases $5 \leq \ell \leq 10$}
When $5 \leq \ell \leq 8$, using equations \eqref{s12sigma2''} and \eqref{mediumell}, we have that
$$S \ll q^{N ( \frac{3}{4}+\frac{1}{2\ell})-d(\frac{5}{4}-\frac{1}{2\ell})+ \frac{U(N) (\ell-2)}{4}+\varepsilon d}+ q^{N-\frac{U(N)}{2}-d+d\varepsilon}+ q^{-\varepsilon d}, \;\; \text{for $U(n) \leq \frac{n}{3}$}. $$
We pick
$$U(n) = \frac{ (n+d)(\ell-2)}{\ell^2},$$
which satisfies $U(n)  \leq n/3$ as long as $d \leq \frac{n(\ell^2-3\ell+6)}{3(\ell-2)}$, i.e. $n$ is large enough. This is justified by Remark \ref{remark-sum-over-n}: we could have replaced the sum $1 \leq n \leq N$ in $S$ by the sum $d(1-\varepsilon) < n \leq N$ for $\varepsilon$ small enough such that $d \leq \frac{n(\ell^2-3\ell+6)}{3(\ell-2)}$ when $n > d(1-\varepsilon)$.
With this choice of $U(n)$, 
$$S \ll q^{ \frac{N(2\ell^2-\ell+2)-d(2\ell^2+\ell-2)}{2\ell^2}+d\varepsilon } +q^{-d\varepsilon },$$
and when $N = v(d-1) < d ( \frac{2\ell^2+\ell-2}{2\ell^2-\ell+2}-\varepsilon)$, which is true when $v < \frac{2 \ell^2 + \ell -2}{2\ell^2 - \ell + 2}$, Theorem \ref{main_thm} follows for $5 \leq \ell \leq 8$.

When $9 \leq \ell \leq 10$, using equations \eqref{s12sigma2''} and \eqref{mediumell2}, we have that 
$$ S \ll q^{N (\frac{\ell}{4}-\frac{5}{4}+\frac{1}{2\ell}) - d (\frac{\ell}{4} - \frac{3}{4}-\frac{1}{2\ell})+\frac{U(N)(\ell-2)}{4}+d\varepsilon }+ q^{N-\frac{U(N)}{2}-d+d\varepsilon}+ q^{-d\varepsilon }.$$
We pick
$$U(n) = \frac{1}{\ell^2} \Big[ d(\ell^2-7\ell-2) - n (\ell^2-9\ell+2)\Big],$$
which satisfies $U(n)\leq n/3$ as long as $d \leq n \frac{4\ell^2-27\ell+6}{3\ell^2-21\ell-6}$, and we can restrict $n$ to this range by Remark \ref{remark-sum-over-n}. With this choice of $U(n)$, we get
$$S \ll q^{ \frac{ N(3\ell^2-9\ell+2)- d(3\ell^2-7\ell-2)}{2\ell^2}+d\varepsilon}+ q^{-d \varepsilon }.$$
Hence, for $N = v(d-1) < d (\frac{3\ell^2-7\ell-2}{3\ell^2-9\ell+2}-\varepsilon )$, Theorem \ref{main_thm} follows for this case.

\subsection{The case $\ell \geq 11$}

When $\ell \geq 11$, using equations \eqref{s12sigma2''} and \eqref{bd_bigell}, we get that
$$ S \ll q^{N(\frac{\ell}{4}-\frac{19}{12}+ \frac{1}{2\ell}) - d(\frac{\ell}{4}-\frac{13}{12}-\frac{1}{2\ell}) + \frac{U(N)(\ell-2)}{4}+d \varepsilon} +q^{N-\frac{U(N)}{2}-d+d\varepsilon}+q^{-d\varepsilon }.$$
We choose
$$U(n) = \frac{1}{3\ell^2} \Big[  d (3\ell^2-25 \ell-6) - n ( 3\ell^2 - 31\ell+6)\Big], $$
which satisfies $U(n) \leq n/3$ as long as $d \leq n  \left( \frac{4\ell^2-31\ell+6}{3\ell^2-25\ell-6} \right) $. As before, we can restrict $n$ to this range by Remark \ref{remark-sum-over-n}. With this choice of $U(n)$, we get
$$S \ll q^{ \frac{N(9\ell^2-31\ell+6)-d(9\ell^2-25\ell-6)}{6\ell^2}+d\varepsilon }+q^{-d \varepsilon}.$$
Hence, as long as $N = v(d-1) < d(\frac{9\ell^2-25\ell-6}{9\ell^2-31\ell+6}-\varepsilon)$, Theorem \ref{main_thm} follows.

\subsection{Proof of Corollary \ref{cor_nonvanishing}}

\begin{proof}[Proof of Corollary \ref{cor_nonvanishing}]
Recall that
\[v=\begin{cases}\frac{6}{5} & \ell=3,\\
\frac{26}{23} & \ell=4,\\
1+\frac{2(\ell-2)}{2\ell^2-\ell+2} & 5\leq \ell \leq 8,\\
1+\frac{2(\ell-2)}{3\ell^2-9\ell+2} & \ell=9, 10,\\
1 + \frac{6 (\ell - 2)}{9 \ell^2 - 31 \ell + 6} & 11\leq  \ell.\end{cases}
\]

From Theorem \ref{main_thm}, we have that
\begin{equation}
\left\langle\Sigma(\phi,\chi_c)\right\rangle_{\mathcal{F}_\ell(d)} = \int_{-\infty}^{\infty} \widehat{\phi}(y) \widehat{W}_{U}(y) \, dy + o(1),
\label{1ld}
\end{equation}
as $d \to \infty$, where $\phi$ is any even Schwartz function with the support of $\widehat{\phi}$ in $(-v,v)$, and where
\begin{equation}
\widehat{W}_{U}(y) = \delta_0(y).
\label{delta0}
\end{equation}
Let
$$p_{m}(d) = \frac{1}{|\mathcal{F}_{\ell}(d)|} \Big| \Big\{  \chi \in \mathcal{F}_{\ell}(d) : \text{ord}_{s=1/2} L(1/2,\chi) = m\Big\} \Big|.$$
Note that $p_0(d)$ gives the proportion of non-vanishing of $L(1/2,\chi)$ in the family.
We have that
\begin{equation}
\label{sumpm}
\sum_{m=0}^{\infty} p_m(d) =1.
\end{equation}
Now we choose the test function
\begin{equation*}
\phi_{v}(x) = \Big(  \frac{\sin(v \pi x)}{v \pi x}\Big)^2,
%\label{test_function}
\end{equation*} and we have that $\widehat{\phi}_{v}(y)$ is supported on $(-v,v)$ and
$$\widehat{\phi}_{v}(y) =
\begin{cases}
\frac{v-|y|}{v^2} & \mbox{ if } |y| \leq v, \\
0 & \mbox{ otherwise.}
\end{cases}
$$

From \eqref{1ld} and since $\phi_{v}(x) \geq 0$, $\phi_{ v}(0)=1$, we have that
$$\sum_{m=1}^{\infty} m p_m(d) \leq \left\langle\Sigma(\phi,\chi_c)\right\rangle_{\mathcal{F}_\ell(d)}+ o(1), $$ and combining this with \eqref{sumpm}, it follows that
\begin{equation*}
p_0(d) \geq 1 - \left\langle\Sigma(\phi,\chi_c)\right\rangle_{\mathcal{F}_\ell(d)} +o(1).
\end{equation*}
Since $\widehat{\phi}_{v}(0)=1/v$, and using \eqref{delta0}, we get that
$$p_0(d) \geq 1- \frac{1}{v}+o(1),$$
which finishes the proof of Corollary \ref{cor_nonvanishing}.
\end{proof}

\section{The proof of Theorem \ref{theorem_cancellation}} \label{proof-thm-cancellation}
The proof follows from the proof of Theorem \ref{main_thm}. Namely, using equations \eqref{sumsigma} and \eqref{hn}, we have that
\begin{align*}
\sum_{\pi \in \mathcal{P}_n} G_{\ell}(R,\pi) \deg(\pi) =  \Sigma_1(n,R,U)-\Sigma_3(n,R,U)-\Sigma_{2'}(n,R,U)-\Sigma_{2''}(n,R,U).
\end{align*}
From equation \eqref{typeII}, we see that the Type II terms are bounded by 
\begin{equation}
\label{bdagain}
\Sigma_{2''}(n,R,U) \ll q^{\frac{3n}{2}-\frac{U}{2}+\varepsilon n } ,
\end{equation}
and similarly 
\begin{equation}
\label{bdagain2}
\Sigma_{3}(n,R,U) \ll   q^{\frac{3n}{2}-\frac{U}{2}+\varepsilon n } ,
\end{equation}
where we used the fact that $U \leq n/3$. Note that the bound does not depend on the shift $R$.

For the Type I terms, from Lemma \ref{lemma_sigma} and \eqref{fn}, we have that
\begin{align*}
\Sigma_{2'}(n,R,U) = \sum_{\substack{ \deg(b) \leq U \\ (b,R)=1}} \mu^2(b) h(b) \frac{1}{2  \pi i} \oint \sum_{i=0}^{\ell-1} \frac{ \psi_R^{(i)}(R,b,z)}{z^{n+1}} \, dz.
\end{align*}
Choosing the contour of integration to be $|z|=q^{-1-1/\ell-\varepsilon}$, we have that
\begin{align}  
\label{typeone}
\Sigma_{2'}(n,R,U) \ll \max_{\substack{|z|=q^{-1-1/\ell-\varepsilon}}} q^{n(1+1/\ell+\varepsilon)} \sum_{\deg(b) \leq U} |b|^{\varepsilon} \sum_{i=0}^{\ell-1} |\psi_R^{(i)}(R,b,z)|.
\end{align}
 Using Corollary \ref{cor_convexity}, for $\ell \geq 9$, we get that
$$ \Sigma_{2'}(n,R,U) \ll q^{n(1+1/\ell+\varepsilon)+ \frac{U(\ell-2)}{4}} |R|^{\frac{\ell}{4}-\frac{7}{4}-\frac{1}{2\ell}+\varepsilon},$$
and similarly
$$ \Sigma_{1}(n,R,U) \ll q^{n(1+1/\ell+\varepsilon)+ \frac{U(\ell-2)}{4}} |R|^{\frac{\ell}{4}-\frac{7}{4}-\frac{1}{2\ell}+\varepsilon}.$$
Using the two equations above and \eqref{bdagain} and \eqref{bdagain2}, we have that 
\begin{equation*}
\sum_{\pi \in \mathcal{P}_n} G_{\ell}(R,\pi) \ll q^{\frac{3n}{2}-\frac{U}{2}+\varepsilon n }+ q^{n(1+1/\ell+\varepsilon)+ \frac{U(\ell-2)}{4}} |R|^{\frac{\ell}{4}-\frac{7}{4}-\frac{1}{2\ell}+\varepsilon}.
\end{equation*}
Note that when 
$$ \deg(R) < \frac{2n(\ell-2)}{\ell^2-7\ell-2},$$ we pick
$$U = \frac{2n(\ell-2) - \deg(R)(\ell^2-7\ell-2)}{\ell^2}, $$ and we get in this case that
\begin{equation} 
\label{firstbound}
\sum_{\pi \in \mathcal{P}_n} G_{\ell}(R,\pi) \deg(\pi) \ll q^{\frac{n(3\ell^2-2\ell+4)+\deg(R)(\ell^2-7\ell-2)}{2\ell^2}+\varepsilon n }.
\end{equation}
When $3 < \ell \leq 8$, we use Corollary \ref{cor_convexity} and from equation \eqref{typeone}, we have that 
$$\Sigma_{2'}(n,R,U) \ll q^{n(1+1/\ell+\varepsilon)+ \frac{U(\ell-2)}{4}} |R|^{\frac{1}{4}-\frac{1}{2\ell}+\varepsilon},
$$ and similarly 
$$\Sigma_1(n,R,U) \ll q^{n(1+1/\ell+\varepsilon)+ \frac{U(\ell-2)}{4}}|R|^{\frac{1}{4}-\frac{1}{2\ell}+\varepsilon}.$$
Using these and \eqref{bdagain} and \eqref{bdagain2}, we have that
$$ \sum_{\pi \in \mathcal{P}_n} G_{\ell}(R,\pi) \ll q^{\frac{3n}{2}-\frac{U}{2}+\varepsilon n } + q^{n(1+1/\ell+\varepsilon)+ \frac{U(\ell-2)}{4}}|R|^{\frac{1}{4}-\frac{1}{2\ell}+\varepsilon}.$$
For $\deg(R) <2n$, choosing 
$$U = \frac{(\ell-2)(2n-\deg(R))}{\ell^2},$$ 
we have that 
\begin{equation} 
\label{secondbound}
\sum_{\pi \in \mathcal{P}_n} G_{\ell}(R,\pi) \deg(\pi) \ll q^{\frac{n(3\ell^2-2\ell+4)+\deg(R)(\ell-2)}{2\ell^2}+\varepsilon n}.
\end{equation}

When $\ell=3$, we shift the contour to $|z|=q^{-4/3-\varepsilon}$ and proceed as in Section \ref{typeI}.
Similarly to equation \eqref{long_expression}, we have that 
\begin{align*}
\Sigma_{2'}(n,R,u) &\ll \max_{\substack{|z|=q^{-4/3-\varepsilon}}} q^{4n/3+\varepsilon n} \sum_{\substack{\deg(b) \leq U \\ (b,R)=1}} \frac{\mu(b)^2}{|b|^{5/6-\varepsilon}}  \sum_{E|r_3^{*}} \frac{1}{|E|^{5/6-\varepsilon}}\sum_{d|bEr_1} \frac{1}{|d|^{7/6-\varepsilon}}  \sum_{i=0}^2 \Big| \psi^{(i)} \Big(  \frac{bEr_1r_2^2}{d},z \Big) \Big|,
\end{align*}
where in the above we write $R=r_1r_2^2r_3^3$ where $r_1, r_2$ are square-free and coprime, and where $r_3^{*}$ denotes, as before, the product of the primes dividing $r_3$, but not $r_1 r_2$. 
Now we use Lemma \ref{new_lemma}, and since $(b,R)=1$ and $b$ is square-free, we have that 
\begin{align*}
\Sigma_{2'}(n,R,u) &\ll q^{4n/3+\varepsilon n}|R|^{\varepsilon} \sum_{\substack{\deg(b) \leq U \\ (b,R)=1}} \frac{\mu(b)^2}{|b|^{5/6-\varepsilon}}  \sum_{E|r_3^{*}} \frac{1}{|E|^{5/6-\varepsilon}}\sum_{d|bEr_1} \frac{1}{|d|^{7/6-\varepsilon}}  \sum_{i=0}^2  \Big| \frac{bEr_1}{d} \Big|^{-1/6} \\
& \ll q^{4n/3+\varepsilon n}|R|^{\varepsilon}.
\end{align*}
Combining this with \eqref{bdagain} and \eqref{bdagain2} and choosing $U=n/3$, we get that

$$ \sum_{\pi \in \mathcal{P}_n} G_{\ell}(R,\pi) \ll q^{4n/3+\varepsilon n} |R|^{\varepsilon}.$$
Together with \eqref{firstbound} and \eqref{secondbound}, Theorem \ref{theorem_cancellation} follows.

\section{The large sieve}
\label{ls}
We will first prove Theorem \ref{largesieve}. We follow the proof in \cite{bgl} and \cite{HB00}. Let
$$\Sigma_1 = 
\sum_{M \in \cH_m} \left| \sum_{N \in \cH_n} \lambda(N) \legl{M}{N} \right|^2 ,$$
and let
$$\mathcal{B}_1(m,n) = \sup \left\{ \Sigma_1 : \sum |\lambda(N)|^2=1 \right\}.$$

We note that the $\Sigma_i$ terms that appear in this section are not to be confused with the terms appearing from Vaughan's identity.

Our goal is to show that
$$\mathcal{B}_1(m,n) \ll q^{\varepsilon(m+n)} (q^m+q^n+q^{\frac{2(m+n)}{3}}).$$
Note that we have
\begin{align*}
\Sigma_1 \leq \sum_{M \in \mathcal{M}_m}  \left| \sum_{N \in \cH_n} \lambda(N) \legl{M}{N} \right|^2 \leq q^m  \sum_{N \in \mathcal{H}_n} |\lambda(N)|^2 + \sum_{\substack{N_1, N_2 \in \mathcal{H}_n \\ N_1 \neq N_2}} \lambda(N_1) \overline{\lambda(N_2)} \sum_{M \in \mathcal{M}_m} \Big(  \frac{M}{N_1} \Big)_{\ell} \overline{ \Big(  \frac{M}{N_2} \Big)_{\ell}}.
\end{align*}
We use the the P\'olya--Vinogradov bound \eqref{pv} to bound the sum over $M$, and we obtain that
$$\sum_{M \in \mathcal{M}_m} \Big(  \frac{M}{N_1} \Big)_{\ell} \overline{ \Big(  \frac{M}{N_2} \Big)_{\ell}} \ll q^{n(1+\varepsilon)}.$$
Further applying Cauchy--Schwarz on the sum over $N_1, N_2$, we get that
$$\Sigma_1 \leq q^{\varepsilon n} (q^m+q^{2n}) \sum_{N \in \mathcal{H}_n} |\lambda(N)|^2,$$
and therefore,
\begin{equation}
\label{initial_bound}
\mathcal{B}_1(m,n) \ll q^{\varepsilon n} (q^m+q^{2n}).
\end{equation}
By duality (see Lemma $4$ in \cite{FloreaSound} and page 71 in \cite{Montgomery}),
we have that
\begin{equation*}
\mathcal{B}_1(m,n) = \mathcal{B}_1(n,m),
%\label{duality}
\end{equation*}
Similarly to \cite[Lemma $5$]{FloreaSound} and \cite[Lemma $9$]{HB95}, we have the following.
\begin{lemma}
\label{increasing}
There is a sufficiently large absolute constant $C$ as follows.  If $q^{m_2-m_1} \geq C(n+m_1)$, then $\mathcal{B}_1(m_1,n) \ll \mathcal{B}_1(m_2,n)$. If $q^{n_2-n_1} \geq C(n_1+m)$, then $\mathcal{B}_1(m,n_1) \ll \mathcal{B}_1(m,n_2).$
\end{lemma}
\begin{proof}
Let $P$ be a prime with $\deg(P) =m_2-m_1$. We have
\begin{align}
\sum_{M \in \mathcal{H}_{m_1}}  \Big| \sum_{N \in \mathcal{H}_n} \lambda(N) \legl{M}{N} \Big|^2 \leq & 2 \sum_{\substack{M \in \mathcal{H}_{m_1} \\ P \nmid M}} \Big| \sum_{\substack{N \in \mathcal{H}_n \\ P\mid N}} \lambda(N) \legl{M}{N} \Big|^2 + 2\sum_{\substack{M \in \mathcal{H}_{m_1}\\ P \nmid M}} \Big| \sum_{\substack{N \in \mathcal{H}_n \\ P \nmid N}} \lambda(N)  \legl{M}{N} \Big|^2  \nonumber \\
&+  \sum_{\substack{M \in \mathcal{H}_{m_1} \\ P \mid M}} \Big| \sum_{N \in \mathcal{H}_n } \lambda(N) \legl{M}{N} \Big|^2.
\label{inc1}
\end{align}
For the second sum on the right-hand side above, note that
\begin{align*}
\sum_{\substack{M \in \mathcal{H}_{m_1}\\P\nmid M}} \Big| \sum_{\substack{N \in \mathcal{H}_n \\ P \nmid N}} \lambda(N) \legl{M}{N} \Big|^2 & = \sum_{\substack{M \in \mathcal{H}_{m_1}\\ P\nmid M}} \Big| \sum_{\substack{N \in \mathcal{H}_n \\ P \nmid N}} \widetilde{\lambda}(N) \legl{PM}{N} \Big|^2 \leq \sum_{M \in \mathcal{H}_{m_2}} \Big| \sum_{\substack{N \in \mathcal{H}_n \\ P \nmid N}} \widetilde{\lambda}(N) \legl{M}{N}  \Big|^2 \\
& \leq \mathcal{B}_1(m_2,n) \sum_{N \in \mathcal{H}_n} |\lambda(N)|^2,
\end{align*}
where $ \widetilde{\lambda}(N) = \lambda(N)  \legl{P}{N}^{-1}$.

We also have that the first sum on the right-hand side of \eqref{inc1} is bounded by
\begin{equation*}
\sum_{\substack{M \in \mathcal{H}_{m_1}}} \Big| \sum_{\substack{N \in \mathcal{H}_n \\ P\mid N}} \lambda(N) \legl{M}{N} \Big|^2 \leq \mathcal{B}_1(m_1,n) \sum_{\substack{N \in \mathcal{H}_n \\ P\mid N}} |\lambda(N)|^2.
\end{equation*}
Combining the above and equation \eqref{inc1} gives that
\begin{align*}
\sum_{M \in \mathcal{H}_{m_1}}  \Big| \sum_{N \in \mathcal{H}_n} \lambda(N) \legl{M}{N} \Big|^2  \leq&   2 \mathcal{B}_1(m_1,n) \sum_{\substack{N \in \mathcal{H}_n \\ P\mid N}} |\lambda(N)|^2 + 2 \mathcal{B}_1(m_2,n) \sum_{N \in \mathcal{H}_n} |\lambda(N)|^2\\&+  \sum_{\substack{M \in \mathcal{H}_{m_1} \\ P \mid  M}} \Big| \sum_{N \in \mathcal{H}_n } \lambda(N) \legl{M}{N}  \Big|^2.
\end{align*}
Now we sum over all $P \in \mathcal{P}_{m_2-m_1}$. Since $$| \mathcal{P}_{m_2-m_1}| = \frac{q^{m_2-m_1}}{m_2-m_1} + O(q^{\frac{m_2-m_1}{2}}),$$ we get that
\begin{align*}
\frac{q^{m_2-m_1}}{m_2-m_1} \sum_{M \in \mathcal{H}_{m_1}} & \Big| \sum_{N \in \mathcal{H}_n} \lambda(N) \legl{M}{N} \Big|^2 \leq 2 \mathcal{B}_1(m_1,n) \sum_{\substack{P \in \mathcal{P}_{m_2-m_1}}}\sum_{\substack{N \in \mathcal{H}_n \\ P\mid N}} |\lambda (N)|^2\\&  + 2\frac{q^{m_2-m_1}}{m_2-m_1} \mathcal{B}_1(m_2,n) \sum_{N \in \mathcal{H}_n} |\lambda(N)|^2  + \sum_{\substack{P \in \mathcal{P}_{m_2-m_1}}}\sum_{\substack{M \in \mathcal{H}_{m_1} \\ P \mid M}} \Big| \sum_{N \in \mathcal{H}_n } \lambda(N) \legl{M}{N} \Big|^2.
\end{align*}

Note that for a given $N \in \mathcal{H}_n$, there are at most $\frac{n}{m_2-m_1}$ primes of degree $m_2-m_1$ dividing $N$, and similarly for a given $M \in \mathcal{H}_{m_1}$, there are at most $\frac{m_1}{m_2-m_1}$ primes of degree $m_2-m_1$ dividing $M$. Hence we get that
$$\frac{q^{m_2-m_1}}{m_2-m_1}  \mathcal{B}_1(m_1,n) \leq \frac{2n+m_1}{m_2-m_1} \mathcal{B}_1(m_1,n) +\frac{2 q^{m_2-m_1}}{m_2-m_1} \mathcal{B}_1(m_2,n).$$
We take any $C> 2$. If  $q^{m_2-m_1} \geq C (n+m_1)$, we have that $$\mathcal{B}_1(m_1,n) \ll \mathcal{B}_1(m_2,n).$$
The second part of the statement follows by duality.
\end{proof}

Let
$$ \Sigma_2 = \sum_{M \in \mathcal{M}_m}  \left| \sum_{N \in \cH_n} \lambda(N) \legl{M}{N} \right|^2$$
and let
$$\mathcal{B}_2(m,n) = \sup \left\{  \Sigma_2 : \sum |\lambda(N)|^2=1\right\}.$$
We have that 
\begin{equation} \mathcal{B}_1(m,n) \leq \mathcal{B}_2(m,n).
\label{b12}
\end{equation}
We also let
$$\Sigma_3 = \Sigma_3(m,n) = \sum_{\substack{N_1, N_2 \in \mathcal{H}_n \\ (N_1,N_2)=1}} \lambda(N_1) \overline{\lambda(N_2)} \sum_{M \in \mathcal{M}_m}  \Big(  \frac{M}{N_1} \Big)_{\ell} \overline{ \Big(  \frac{M}{N_2} \Big)_{\ell}},$$
and $$\mathcal{B}_3(m,n) = \sup \left\{  \Sigma_3 : \sum |\lambda(N)|^2=1\right\}.$$
We will prove the following Lemmas. 
\begin{lemma}
\label{lem6}
There exist $m_1, \ldots, m_{\ell}$ such that $\sum_{j=1}^{\ell} j m_j = m$ and
$$\mathcal{B}_2(m,n) \ll m^{\ell-1}  \min \Big\{q^{\sum_{ j \neq k} m_j} \mathcal{B}_1(m_k,n) : k \leq \ell-1  \Big\}.$$
\end{lemma} 
\begin{proof}
We write $M= M_1 \cdots M_{\ell-1}^{\ell-1} M_{\ell}^{\ell}$, where $M_1, \ldots, M_{\ell-1}$ are monic, square-free and pairwise coprime. Then
\begin{align*}
\Sigma_2 &= \sum_{m_1+\cdots +\ell m_{\ell}=m} \sum_{\substack{M_j \in \cH_{m_j},\, j\leq \ell-1 \\ M_{\ell} \in \cM_{m_{\ell}}}} \Big| \sum_{N \in \cH_n} \lambda(N) \Big( \frac{M_1 \cdots M_{\ell}^{\ell}}{N} \Big)_{\ell} \Big|^2 \\
& \ll m^{\ell-1} \sum_{\substack{M_j \in \cH_{m_j},\, j\leq \ell-1 \\ M_{\ell} \in \cM_{m_{\ell}}}} \Big| \sum_{N \in \cH_n} \lambda(N) \Big( \frac{M_1 \cdots M_{\ell}^{\ell}}{N} \Big)_{\ell} \Big|^2,
\end{align*}
for some tuple $(m_1,\ldots m_{\ell})$ such that $m_1+\ldots+\ell m_{\ell}=m$. Fixing $k\leq \ell-1$, we have that
\begin{align*}
\Sigma_2 &\ll m^{\ell-1} \sum_{\substack{j=1 \\ j \neq k}}^{\ell-1} \sum_{M_j \in \cH_{m_j}} \sum_{M_{\ell} \in \cM_{m_{\ell}}} \sum_{M_k \in \cH_{m_k}} \Big| \sum_{N \in \cH_n} \lambda(N) \Big( \frac{M_1 \ldots M_{\ell}^{\ell}}{N} \Big)_{\ell} \Big|^2 \\
& \ll m^{\ell-1} q^{m_{\ell} + \sum_{j \leq \ell-1, j \neq k} m_j} \mathcal{B}_1(m_k,n) \sum_{N \in \cH_n} |\lambda(N)|^2.
\end{align*}
Using the fact that $\ell m_{\ell} = m- \sum_{j=1}^{\ell-1} j m_j$, the conclusion follows.
\end{proof}
We also have the following. 
\begin{lemma}
\label{lem7}
There exist $\Delta_1, \Delta_2$ with $\Delta_1 \leq \Delta_2$ such that
$$\mathcal{B}_2(m,n) \ll q^{\varepsilon n} \mathcal{B}_3(m-\Delta_1, n-\Delta_2).$$
\end{lemma}
\begin{proof}
We write
\begin{align*}
\Sigma_2 = \sum_{\delta \in \mathcal{H}_{\leq n}} \sum_{M \in \cM_m} \sum_{\substack{N_1,N_2\in \mathcal{H}_n\\ (N_1,N_2)=\delta}} \lambda(N_1) \overline{\lambda(N_2)} \Big(  \frac{M}{N_1} \Big)_{\ell} \overline{ \Big(  \frac{M}{N_2} \Big)_{\ell}}.
\end{align*}
Writing $N_i = \delta R_i$, we have that
\begin{align*}
\Sigma_2 &= \sum_{\delta \in \mathcal{H}_{\leq n}} \sum_{\substack{M \in \cM_m \\ (M,\delta)=1}} \sum_{\substack{R_1,R_2\in \mathcal{H}_{n-\deg(\delta)}\\ (R_1,R_2)=1}} \lambda( \delta R_1) \overline{\lambda(\delta R_2)} \Big(  \frac{M}{R_1} \Big)_{\ell} \overline{ \Big(  \frac{M}{R_2} \Big)_{\ell}} \\
&= \sum_{\delta \in \mathcal{H}_{\leq n}} \sum_{D | \delta} \mu(D) \sum_{S \in \mathcal{M}_{m-\deg(D)}} \sum_{\substack{R_1,R_2\in \mathcal{H}_{n-\deg(\delta)}\\ (R_1,R_2)=1}}  \lambda( \delta R_1) \overline{\lambda(\delta R_2)} \Big(  \frac{D}{R_1} \Big)_{\ell} \overline{ \Big(  \frac{D}{R_2} \Big)_{\ell}} \Big(  \frac{S}{R_1} \Big)_{\ell} \overline{ \Big(  \frac{S}{R_2} \Big)_{\ell}} \\
& \leq \sum_{\delta \in \mathcal{H}_{\leq n}} \sum_{D | \delta} \mathcal{B}_3(m-\deg(D), n - \deg(\delta))\sum_{R\in \mathcal{H}_{n-\deg(\delta)}} |\lambda(\delta R)|^2.
\end{align*}
Let $$\widetilde{\mathcal{B}_3}(m,n) = \max \{ \mathcal{B}_3(m - \Delta_1, n - \Delta_2): \Delta_1 \leq \Delta_2 \}.$$
From the above it follows that
\begin{align*}
\Sigma_2  & \leq \widetilde{\mathcal{B}_3}(m,n) \sum_{\delta \in \mathcal{H}_{\leq n}} \sum_{D|\delta} \sum_{R\in \mathcal{H}_{n-\deg(\delta)}} |\lambda(\delta R)|^2 \leq \widetilde{\mathcal{B}_3}(m,n) \sum_{N\in \mathcal{H}_n} d(N)^2 |\lambda(N)|^2  \\
& \ll q^{\varepsilon n} \widetilde{\mathcal{B}_3}(m,n)\sum_{N\in \mathcal{H}_n}  |\lambda(N)|^2.
\end{align*}
The conclusion now follows.
\end{proof}
We will also prove the following.
\begin{lemma}
\label{lem8}
For any $\varepsilon>0$, we have
$$\mathcal{B}_3(m,n) \ll q^{m-n+\varepsilon n} \max \{ \mathcal{B}_2(k,n) : k \leq 2n-m-1 \}.$$
\end{lemma}
\begin{remark}
 If $m \geq 2n$, the above statement is still true, as $\mathcal{B}_3(m,n)=0$ by orthogonality of characters.
\end{remark}
\begin{proof}
We rewrite
\begin{align*}
\Sigma_3 &= \sum_{\substack{N_1,N_2\in \mathcal{H}_n\\(N_1,N_2)=1}} \lambda(N_1) \overline{ \lambda(N_2)} \sum_{M \in \mathcal{M}_m}  \Big(  \frac{M}{N_1} \Big)_{\ell} \overline{ \Big(  \frac{M}{N_2} \Big)_{\ell}},
\end{align*}
and use Poisson summation for the sum over $M$.

We write $\Sigma_3 = \Sigma_{3,0} + \Sigma_{3,\not\equiv 0}$ according to whether $n \equiv 0 \pmod{\ell}$ or $n \not \equiv 0 \pmod{\ell}$ and we bound $\Sigma_{3,\not\equiv 0}$.
The case of  $\Sigma_{3,0}$ is similar. Let
$$  \chi(M) =  \Big(  \frac{M}{N_1} \Big)_{\ell} \overline{ \Big(  \frac{M}{N_2} \Big)_{\ell}} \mbox{ and recall that } \chi_{N_i}(M)= \Big(  \frac{M}{N_i} \Big)_{\ell}.$$
Using Lemma \ref{poisson}, we have that
\begin{align*}
\sum_{M \in \mathcal{M}_m} \chi(M) = q^{m-2n+1/2} \overline{\varepsilon(\chi)} \sum_{V \in \mathcal{M}_{\leq 2n-m-1}} G_{\ell}(V,\chi).
\end{align*}
Since $N_1$ and $N_2$ are square-free and coprime, and since $\legl{N_1}{N_2}\overline{\legl{N_2}{N_1}}=1$, we have
$$G_{\ell}(V,\chi) = G_{\ell}(V,\chi_{N_1}) G_{\ell}(V,\overline{\chi_{N_2}})=G_{\ell}(1,\chi_{N_1}) G_{\ell}(1,\overline{\chi_{N_2}}) \overline{ \Big( \frac{V}{N_1} \Big)}_{\ell} \Big(  \frac{V}{N_2} \Big)_{\ell}.$$
We have that
\begin{align*}
 \Sigma_{3,\not\equiv 0}= & q^{m-2n-1/2}\overline{\varepsilon(\chi)}\sum_{V \in \mathcal{M}_{\leq 2n-m-1}} \sum_{\substack{N_1,N_2\in \mathcal{H}_n\\(N_1,N_2)=1}} \lambda(N_1) \overline{\lambda(N_2)} G_\ell(1,\chi_{N_1}) G_\ell(1,\overline{\chi_{N_2}}) \overline{ \Big( \frac{V}{N_1} \Big)}_{\ell} \Big(  \frac{V}{N_2} \Big)_{\ell}\\
 =& q^{m-2n-1/2}\overline{\varepsilon(\chi)}\sum_{V \in \mathcal{M}_{\leq 2n-m-1}}\sum_{\substack{D \in \mathcal{M}_{\leq n}}} \mu(D)  \sum_{\substack{N_1,N_2\in \mathcal{H}_n\\D\mid N_1\\ D\mid N_2}} \lambda(N_1) \overline{\lambda(N_2)} G_\ell(1,\chi_{N_1}) G_\ell(1,\overline{\chi_{N_2}}) \overline{ \Big( \frac{V}{N_1} \Big)}_{\ell} \Big(  \frac{V}{N_2} \Big)_{\ell}\\
 \ll & q^{m-2n-1/2}  \sum_{\substack{D \in \mathcal{M}_{\leq n}}} \sum_{j=0}^{2n-m-1} \sum_{V \in \mathcal{M}_{j}} \Bigg| \sum_{\substack{N_1,N_2\in \mathcal{H}_n\\D\mid N_1\\ D\mid N_2}} \lambda(N_1) \overline{\lambda(N_2)} G_\ell(1,\chi_{N_1}) G_\ell(1,\overline{\chi_{N_2}}) \overline{ \Big( \frac{V}{N_1} \Big)}_{\ell} \Big(  \frac{V}{N_2} \Big)_{\ell}\Bigg|.
\end{align*}
Using Cauchy--Schwarz on the sum over $V$, we get that
\begin{align*}
\sum_{V \in \mathcal{M}_j} & \Big| \sum_{\substack{N_1,N_2 \in \mathcal{H}_n\\ D|N_1 \\ D|N_2}}  \lambda(N_1) \overline{\lambda(N_2)} G(1,\chi_{N_1}) G(1,\overline{\chi_{N_2}}) \overline{ \Big( \frac{V}{N_1} \Big)}_{\ell} \Big(  \frac{V}{N_2} \Big)_{\ell}  \Big| \\
& \leq \Big( \sum_{V \in \mathcal{M}_j} \Big|  \sum_{\substack{N_1\in \mathcal{H}_n\\D|N_1}}    \lambda(N_1)  G(1,\chi_{N_1}) \overline{ \Big( \frac{V}{N_1} \Big)}_{\ell} \Big|^2 \Big)^{1/2} \Big( \sum_{V \in \mathcal{M}_j} \Big|  \sum_{\substack{N_2\in \mathcal{H}_n\\D|N_2}}   \lambda(N_2)  G(1,\overline{\chi_{N_2}})  \Big( \frac{V}{N_2} \Big)_{\ell} \Big|^2 \Big)^{1/2}\\
& \leq \mathcal{B}_2(j,n) \sum_{\substack{N\in \mathcal{H}_n\\D|N}} q^n |\lambda(N)|^2,
\end{align*}
where in the last line we used the fact that $|G(1,\chi_N)|=|G(1,\overline{\chi_N})| = q^{n/2}.$
Combining the previous two equations, it follows that
\begin{align*}
S_{3,1} &\ll q^{m-n} \sum_{D\in \mathcal{M}_{\leq n}} \sum_{\substack{N\in \mathcal{H}_n\\D|N}} |\lambda(N)|^2 \max \{ \mathcal{B}_2(j,n) : j \leq 2n-m-2\} \\
& =q^{m-n}\max \{ \mathcal{B}_2(j,n) : j \leq 2n-m-2\}  \sum_{\substack{N\in \mathcal{H}_n}} |\lambda(N)|^2 d(N)  \\
& \ll q^{m-n+\varepsilon n} \max \{ \mathcal{B}_2(j,n) : j \leq 2n-m-2\}\sum_{\substack{N\in \mathcal{H}_n}} |\lambda(N)|^2.
\end{align*}
A similar expression holds for $S_{3,2}$, which finishes the proof of the statement.
\end{proof}
Now we will prove the following recursive estimate.
\begin{lemma}
\label{recursive}
Suppose that $4/3<\alpha \leq 2$ and
$$\mathcal{B}_1(m,n) \ll q^{\varepsilon(m+n)} \Big( q^m +q^{\alpha n} + q^{\frac{2}{3}(m+n)}\Big),$$ for any $\varepsilon>0$. Then
$$\mathcal{B}_1(m,n) \ll q^{\varepsilon(m+n)} \Big(q^m + q^{(2-\frac{2}{3\alpha-1})n} + q^{\frac{2}{3}(m+n)} \Big).$$
\end{lemma}
\begin{proof}
Using duality and the condition in the statement, we have that
$$ \mathcal{B}_1(m,n) \ll q^{\varepsilon(m+n)} \Big( q^{\alpha m} + q^{n} + q^{\frac{2}{3}(m+n)} \Big).$$
We use this in Lemma \ref{lem6} as follows. If $m_1 \geq m_2$, then we use Lemma \ref{lem6} with $k=1$, and we have that
\begin{align} 
\label{bd1}
\mathcal{B}_2(m,n) \ll q^{\varepsilon(m+n)+m_2+\cdots+m_{\ell}} \Big(q^{\alpha m_1} + q^n + q^{\frac{2}{3}(m_1+n)} \Big).
\end{align}
Now we use the fact that for $m_1 \geq m_2$, 
\begin{equation} q^{m_2+\cdots+m_{\ell}} \ll \min \{ q^{\frac{m}{3}} , q^{\frac{2}{3}(m-m_1)} \}.
\label{bd2}
\end{equation}
Indeed, to get the inequality above, note that
$$ m_2+\cdots + m_{\ell} = \frac{1}{3} (3m_2+\cdots + 3m_{\ell}) \leq \frac{1}{3} (m + m_2-m_1) \leq \frac{m}{3},$$
and
$$ m_2+\cdots+m_{\ell} = \frac{1}{2} (2m_2+\cdots +2m_{\ell}) \leq \frac{m-m_1}{2} \leq \frac{2(m-m_1)}{3}.$$
Combining equations \eqref{bd1} and \eqref{bd2}, we get that 
\begin{equation}
\label{b2b}
\mathcal{B}_2(m,n) \ll q^{\varepsilon(m+n)} (q^{\alpha m}+ q^{\frac{m}{3}+ n}).
\end{equation}
If $m_1 < m_2$, then we use Lemma \ref{lem6} with $k=2$ and use the fact that
\begin{equation}
\label{fact2}
q^{m_1+m_3+\cdots+m_{\ell}} \leq \min \{ q^{\frac{m}{3}}, q^{\frac{2}{3}(m-m_2)} \}.
\end{equation}
Indeed, to see the above, note that 
$$ m_1+ m_3+\cdots+m_{\ell} = \frac{1}{3} (3m_1+3m_3+\cdots+3m_{\ell}) \leq \frac{1}{3} (m+2m_1-2m_2).$$
We then get that
\begin{align*}
\mathcal{B}_2(m,n) &\ll q^{\varepsilon(m+n)} \min \{ q^{\frac{m}{3}}, q^{\frac{2}{3}(m-m_2)} \} \Big(q^{\alpha m_2} + q^n + q^{\frac{2}{3}(m_2+n)} \Big) \\
& \ll q^{\varepsilon(m+n)} \Big( q^{\alpha m} + q^{\frac{m}{3}+n}\Big).
\end{align*}
Combining this with \eqref{b2b}, we have that
\begin{equation}
\mathcal{B}_2(m,n) \ll q^{\varepsilon(m+n)} (q^{\alpha m}+ q^{\frac{m}{3}+ n}).
\label{bound_b2}
\end{equation}
Using Lemma \ref{lem8} and equation \eqref{bound_b2}, we get that
\begin{equation*}
\mathcal{B}_3(m,n) \ll q^{\varepsilon(m+n)} \Big(q^{n(2\alpha-1)+m(1-\alpha)} + q^{\frac{2}{3}(m+n)} \Big).
\end{equation*}
Now we use the above and Lemma \ref{lem7} and it follows that there exist $\Delta_1 \leq \Delta_2$ such that 
\begin{align*}
\mathcal{B}_2(m,n) &\ll q^{\varepsilon n} \mathcal{B}_3(m-\Delta_1, n - \Delta_2) \ll q^{\varepsilon(m+n)} \Big(  q^{(n-\Delta_2)(2\alpha-1)+(m-\Delta_1)(1-\alpha)} + q^{\frac{2}{3} (m+n-\Delta_1-\Delta_2)}\Big) \\
& \ll q^{\varepsilon(m+n)} \Big(q^{m(1-\alpha)+n(2\alpha-1)} + q^{\frac{2}{3} (m+n)} \Big),
\end{align*}
where in the last line we used the fact that $\Delta_1 \leq \Delta_2$. Note that the above holds if $n \geq 1$. When $n=0$, then we have the bound
$$ \mathcal{B}_2(m,0) \ll q^{m}.$$
Putting the two bounds above together, we get that
\begin{align}
\label{bd_interm}
\mathcal{B}_2(m,n) \ll q^{\varepsilon(m+n)} \Big(q^m+ q^{m(1-\alpha)+n(2\alpha-1)} + q^{\frac{2}{3} (m+n)} \Big).
\end{align}
Now using Lemma \ref{increasing} and \eqref{b12}, we have that there exists some large enough absolute constant $C$  such that for $m' \geq m+\log_q(C(m+n))$
\begin{align*}
\mathcal{B}_1(m,n) &\ll \mathcal{B}_1(m',n) \leq \mathcal{B}_2(m',n) \\
& \ll q^{\varepsilon(m'+n)} \Big(q^{m'}+ q^{m'(1-\alpha)+n(2\alpha-1)} + q^{\frac{2}{3} (m'+n)} \Big),
\end{align*}
where in the last line we used equation \eqref{bd_interm}. Now we pick
$m' =  \Big[ \max \{ m, \frac{6\alpha-5}{3\alpha-1} n\}+\log_q(C(m+n)) \Big]+1$. Note that if $\max \{ m, \frac{6\alpha-5}{3\alpha-1} n\}= m$ then the second term above is bounded by the third. If $\max \{ m, \frac{6\alpha-5}{3\alpha-1} n\}= \frac{6\alpha-5}{3\alpha-1}n$, then
$$\mathcal{B}_1(m,n) \ll q^{\varepsilon(m+n)} q^{ \frac{6\alpha-4}{3\alpha-1}n},$$
where we have used that $C(m+n)$ is negligible compared to $q^{\varepsilon (m+n)}$.

Combining these, we get that
$$ \mathcal{B}_1(m,n) \ll q^{\varepsilon(m+n)} \Big(q^m+q^{ \frac{6\alpha-4}{3\alpha-1}n}+q^{\frac{2}{3} (m+n)} \Big).$$
\end{proof}

\begin{proof}[Proof of Theorem \ref{largesieve}]

Now we are finally ready to prove Theorem \ref{largesieve}. Note that the condition in Lemma \ref{recursive} is satisfied with $\alpha = 2$ by equation \eqref{initial_bound}. Repeatedly using Lemma \ref{recursive}, we deduce that the bound holds for the fixed point of $\alpha \mapsto 2 - 2/(3 \alpha -1)$, which is $\alpha = 4/3$, and then
$$\mathcal{B}_1(m,n) \ll q^{\varepsilon(m+n)} \Big( q^m +q^{\frac{4n}{3}} + q^{\frac{2}{3}(m+n)}\Big).$$
By duality,
\begin{align*}
 \mathcal{B}_1(m,n) &\ll q^{\varepsilon(m+n)} \min \left\{q^m+q^{ \frac{4n}{3}}+q^{\frac{2}{3} (m+n)}, q^n+q^{ \frac{4m}{3}}+q^{\frac{2}{3} (m+n)} \right\} \\
 & \ll q^{\varepsilon(m+n)} \Big(q^m+q^n+q^{\frac{2}{3}(m+n)} \Big),
\end{align*}
which finishes the proof.
\end{proof}

\begin{proof}[Proof of Theorem \ref{largesieve2}]
We use Lemma \ref{lem6} and Theorem \ref{largesieve}. If $m_1 \geq m_2$, then we use Lemma \ref{lem6} with $k=1$. We have that
\begin{align*}
\mathcal{B}_2(m,n) & \ll q^{\varepsilon(m+n)} q^{m_2+\cdots+m_{\ell}} \mathcal{B}_1(m_1,n) \\
& \ll  q^{\varepsilon(m+n)} q^{m_2+\cdots+m_{\ell}} \Big(q^{m_1}+q^n+q^{\frac{2}{3}(m_1+n)} \Big).
\end{align*}
We further use \eqref{bd2} in the equation above, which finishes the proof in this case.

If $m_1<m_2$, then we use Lemma \ref{lem6} with $k=2$. We get that
\begin{align*}
\mathcal{B}_2(m,n) & \ll q^{\varepsilon(m+n)} q^{m_1+m_3+\cdots+m_{\ell}} \mathcal{B}_1(m_2,n) \\
& \ll  q^{\varepsilon(m+n)} q^{m_1+m_3+\cdots+m_{\ell}} \Big(q^{m_2}+q^n+q^{\frac{2}{3}(m_2+n)} \Big).
\end{align*}
Using the fact \eqref{fact2} again finishes the proof.
\end{proof}

 \bibliographystyle{amsalpha}

\bibliography{Bibliography}

\providecommand{\bysame}{\leavevmode\hbox to3em{\hrulefill}\thinspace}
\providecommand{\MR}{\relax\ifhmode\unskip\space\fi MR }
\providecommand{\MRhref}[2]{%
  \href{http://www.ams.org/mathscinet-getitem?mr=#1}{#2}
}
\providecommand{\href}[2]{#2}
\begin{thebibliography}{CLDLL22}

\bibitem[BF18]{Bui-Florea}
H.~M. Bui and Alexandra Florea, \emph{Zeros of quadratic {D}irichlet
  {$L$}-functions in the hyperelliptic ensemble}, Trans. Amer. Math. Soc.
  \textbf{370} (2018), no.~11, 8013--8045. \MR{3852456}

\bibitem[BGL14]{bgl}
Valentin Blomer, Leo Goldmakher, and Beno\^it Louvel, \emph{{$L$}-functions
  with {$n$}-th-order twists}, Int. Math. Res. Not. IMRN (2014), no.~7,
  1925--1955. \MR{3190355}

\bibitem[Bui12]{bui}
H.~M. Bui, \emph{Non-vanishing of {D}irichlet {$L$}-functions at the central
  point}, Int. J. Number Theory \textbf{8} (2012), no.~8, 1855--1881.
  \MR{2978845}

\bibitem[BY10]{BY}
Stephan Baier and Matthew~P. Young, \emph{Mean values with cubic characters},
  J. Number Theory \textbf{130} (2010), no.~4, 879--903. \MR{2600408}

\bibitem[CLDLL22]{ldlw}
Antoine Comeau-Lapointe, Chantal David, Matilde Lalin, and Wanlin Li, \emph{On
  the vanishing of twisted {$L$}-functions of elliptic curves over rational
  function fields}, Res. Number Theory \textbf{8} (2022), no.~4, Paper No. 76,
  28. \MR{4491491}

\bibitem[DdFDS]{ddds}
Chantal David, Alexandre de~Faveri, Alexander Dunn, and Joshua Stucky,
  \emph{{Non-vanishing for cubic Hecke {L}-functions}}, arXiv e-prints,
  arXiv:2410.03048.

\bibitem[DDHL]{DDHL}
Chantal David, Alexander Dunn, Alia Hamieh, and Hua Lin, \emph{Quartic gauss
  sums over primes and metaplectic theta functions}, arXiv e-prints,
  arXiv:2306.11875.

\bibitem[DFL21]{DFL2}
Chantal David, Alexandra Florea, and Matilde Lalin, \emph{Nonvanishing for
  cubic {$L$}-functions}, Forum Math. Sigma \textbf{9} (2021), Paper No. e69,
  58. \MR{4323990}

\bibitem[DFL22]{DFL}
Chantal David, Alexandra Florea, and Matilde Lal\'in, \emph{The mean values of
  cubic {$L$}-functions over function fields}, Algebra Number Theory
  \textbf{16} (2022), no.~5, 1259--1326. \MR{4471042}

\bibitem[DG22]{dg}
Chantal David and Ahmet~M. G\"ulo\u{g}lu, \emph{One-level density and
  non-vanishing for cubic {$L$}-functions over the {E}isenstein field}, Int.
  Math. Res. Not. IMRN (2022), no.~23, 18833--18873. \MR{4519156}

\bibitem[DR24]{DR}
Alexander Dunn and Maksym Radziwi\l\l, \emph{Bias in cubic {G}auss sums:
  {P}atterson's conjecture}, Ann. of Math. (2) \textbf{200} (2024), no.~3,
  967--1057. \MR{4816436}

\bibitem[ELS20]{ELS}
Jordan~S. Ellenberg, Wanlin Li, and Mark Shusterman, \emph{Nonvanishing of
  hyperelliptic zeta functions over finite fields}, Algebra Number Theory
  \textbf{14} (2020), no.~7, 1895--1909. \MR{4150253}

\bibitem[Flo17]{Florea1}
Alexandra~M. Florea, \emph{Improving the error term in the mean value of
  {$L(\frac12,\chi)$} in the hyperelliptic ensemble}, Int. Math. Res. Not. IMRN
  (2017), no.~20, 6119--6148. \MR{3712193}

\bibitem[FS]{FloreaSound}
Alexandra Florea and Kannan Soundararajan, \emph{The large sieve over function
  fields}, Unpublished note, available at
  https://sites.google.com/view/alexandraflorea/publications.

\bibitem[GY24]{G-Y}
Ahmet~M. G\"{u}lo\u{g}lu and Hamza Yesilyurt, \emph{Mollified moments of cubic
  {D}irichlet {$L$}-functions over the {E}isenstein field}, J. Math. Anal.
  Appl. \textbf{533} (2024), no.~2, Paper No. 128014, 49. \MR{4677712}

\bibitem[GZ20]{gao_zhao}
Peng Gao and Liangyi Zhao, \emph{One-level density of low-lying zeros of
  quadratic and quartic {H}ecke {$L$}-functions}, Canad. J. Math. \textbf{72}
  (2020), no.~2, 427--454. \MR{4081698}

\bibitem[GZ22]{gao_zhao2}
\bysame, \emph{Bounds for moments of cubic and quartic {D}irichlet
  {$L$}-functions}, Indag. Math. (N.S.) \textbf{33} (2022), no.~6, 1263--1296.
  \MR{4498233}

\bibitem[Hay66]{hayes}
D.~R. Hayes, \emph{The expression of a polynomial as a sum of three
  irreducibles}, Acta Arith. \textbf{11} (1966), 461--488. \MR{0201422}

\bibitem[HB95]{HB95}
D.~R. Heath-Brown, \emph{A mean value estimate for real character sums}, Acta
  Arith. \textbf{72} (1995), no.~3, 235--275. \MR{1347489}

\bibitem[HB00]{HB00}
\bysame, \emph{Kummer's conjecture for cubic {G}auss sums}, Israel J. Math.
  \textbf{120} (2000), no.~part A, 97--124. \MR{1815372}

\bibitem[HBP79]{hbp}
D.~R. Heath-Brown and S.~J. Patterson, \emph{The distribution of {K}ummer sums
  at prime arguments}, J. Reine Angew. Math. \textbf{310} (1979), 111--130.
  \MR{546667}

\bibitem[Hof92]{hoffstein}
Jeffrey Hoffstein, \emph{Theta functions on the {$n$}-fold metaplectic cover of
  {${\rm SL}(2)$}---the function field case}, Invent. Math. \textbf{107}
  (1992), no.~1, 61--86. \MR{1135464}

\bibitem[Hsu99]{hsu}
Chih-Nung Hsu, \emph{Estimates for coefficients of {$L$}-functions for function
  fields}, Finite Fields Appl. \textbf{5} (1999), no.~1, 76--88. \MR{1667104}

\bibitem[IS99]{is}
H.~Iwaniec and P.~Sarnak, \emph{Dirichlet {$L$}-functions at the central
  point}, Number theory in progress, {V}ol. 2 ({Z}akopane-{K}o\'scielisko,
  1997), de Gruyter, Berlin, 1999, pp.~941--952. \MR{1689553}

\bibitem[KMN22]{kmn}
Rizwanur Khan, Djordje Mili\'cevi\'c, and Hieu~T. Ngo, \emph{Nonvanishing of
  {D}irichlet {$L$}-functions, {II}}, Math. Z. \textbf{300} (2022), no.~2,
  1603--1613. \MR{4363789}

\bibitem[KS99]{KS}
Nicholas~M. Katz and Peter Sarnak, \emph{Random matrices, {F}robenius
  eigenvalues, and monodromy}, American Mathematical Society Colloquium
  Publications, vol.~45, American Mathematical Society, Providence, RI, 1999.
  \MR{1659828}

\bibitem[Mon71]{Montgomery}
Hugh~L. Montgomery, \emph{Topics in multiplicative number theory}, Lecture
  Notes in Mathematics, Vol. 227, Springer-Verlag, Berlin-New York, 1971.
  \MR{337847}

\bibitem[OS99]{OS}
A.~E. \"{O}zl\"{u}k and C.~Snyder, \emph{On the distribution of the nontrivial
  zeros of quadratic {$L$}-functions close to the real axis}, Acta Arith.
  \textbf{91} (1999), no.~3, 209--228. \MR{1735673}

\bibitem[Pat87]{LMS}
S.~J. Patterson, \emph{The distribution of general {G}auss sums and similar
  arithmetic functions at prime arguments}, Proc. London Math. Soc. (3)
  \textbf{54} (1987), no.~2, 193--215. \MR{872805}

\bibitem[Pat07]{patterson}
\bysame, \emph{Note on a paper of {J}. {H}offstein}, Glasg. Math. J.
  \textbf{49} (2007), no.~2, 243--255. \MR{2347258}

\bibitem[Ros02]{Rosen}
Michael Rosen, \emph{Number theory in function fields}, Graduate Texts in
  Mathematics, vol. 210, Springer-Verlag, New York, 2002. \MR{1876657}

\bibitem[Rud10]{rudnick}
Ze\'ev Rudnick, \emph{Traces of high powers of the {F}robenius class in the
  hyperelliptic ensemble}, Acta Arith. \textbf{143} (2010), no.~1, 81--99.
  \MR{2640060}

\bibitem[Sou00]{Sound}
K.~Soundararajan, \emph{Nonvanishing of quadratic {D}irichlet {$L$}-functions
  at {$s=\frac12$}}, Ann. of Math. (2) \textbf{152} (2000), no.~2, 447--488.
  \MR{1804529}

\bibitem[Sti09]{Stichtenoth}
Henning Stichtenoth, \emph{Algebraic function fields and codes}, second ed.,
  Graduate Texts in Mathematics, vol. 254, Springer-Verlag, Berlin, 2009.
  \MR{2464941}

\bibitem[Wei71]{Weil2}
Andr\'{e} Weil, \emph{Dirichlet series and automorphic forms}, Lect. Notes.
  Math., vol. 189, Berlin--Heidelberg--New York, 1971.

\end{thebibliography}
\end{document}